\documentclass[10pt]{amsart}

\usepackage[english]{babel}
\usepackage[T1]{fontenc}
\usepackage{tikz}
\usetikzlibrary{through,calc}
\usepackage{amsmath,amssymb,stmaryrd}
\usepackage[all]{xy}
\usepackage{graphicx}
\usepackage{fancyhdr}
\usepackage{hyperref} 

\hypersetup{
colorlinks=false, 
breaklinks=true, 
urlcolor= blue, 
linkcolor= blue, 
citecolor=blue, 
pdftitle={Rapport de stage}, 
pdfauthor={Anonyme}, 
pdfsubject={Simulation} 
}

 \theoremstyle{plain}
 \newtheorem{lemma}{Lemma}[section]
\newtheorem{theorem}[lemma]
{Theorem }
\newtheorem{corollary}[lemma]
{Corollary}
\newtheorem{prop}[lemma]{Proposition}

\newtheorem*{theo}{Theorem}

\theoremstyle{definition}

\newtheorem{definition}[lemma]{Definition }

\newtheorem{ex}[lemma]
{Example }

\newtheorem{rmk}[lemma]{Remark}

\newcommand{\lgw}{\longrightarrow}

\newcommand{\lgm}{\longmapsto}
\newcommand{\s}{\sigma}

\renewcommand{\)}{)\!)}
\newcommand{\ovl}{\overline}
\newcommand{\Frac}{\operatorname{Frac}}
\newcommand{\La}{\Lambda}
\renewcommand{\deg}{\operatorname{deg}}
\newcommand{\ord}{\operatorname{ord}}
\newcommand{\Ker}{\operatorname{Ker}}
\newcommand{\alg}{\operatorname{alg}}
\renewcommand{\Im}{\operatorname{Im}}
\newcommand{\lb}{\llbracket}
\newcommand{\rb}{\rrbracket}
\newcommand{\wdh}{\widehat}
\newcommand{\udl}{\underline}
\newcommand{\NP}{\operatorname{NP}}
\renewcommand{\th}{\theta}
\newcommand{\p}{\mathfrak{p}}
\newcommand{\ini}{\operatorname{in}}
\newcommand{\wdt}{\widetilde}
\newcommand{\x}{\mathbf{x}}
\newcommand{\y}{\mathbf{y}}
\newcommand{\w}{\mathbf{w}}
\renewcommand{\l}{\lambda}
\newcommand{\fg}{\operatorname{fg}}
\renewcommand{\L}{\mathbb{L}}
\renewcommand{\O}{\mathcal{O}}
\newcommand{\Spec}{\operatorname{Spec}}

\newcommand{\KK}{\mathcal{K}}
\newcommand{\m}{\mathfrak{m}}

\newcommand{\Z}{\mathbb{Z}}
\newcommand{\Supp}{\operatorname{Supp}}
\newcommand{\V}{\mathcal{V}}
\newcommand{\Rel}{\operatorname{Rel}}

\newcommand{\G}{\Gamma}
\renewcommand{\k}{\Bbbk}
\renewcommand{\dim}{\operatorname{dim}}
\newcommand{\trdeg}{\operatorname{tr.deg}}

\newcommand{\R}{\mathbb{R}}
\newcommand{\K}{\mathbb{K}}

\newcommand{\N}{\mathbb{N}}
\renewcommand{\P}{\mathbb{P}}
\newcommand{\dd}{\mathbf{d}}
\newcommand{\C}{\mathbb{C}}
\newcommand{\Q}{\mathbb{Q}}

\newcommand{\D}{\Delta}
\newcommand{\Di}{\operatorname{Disc}}
\newcommand{\U}{\mathbb{U}}
\newcommand{\F}{\mathcal{F}}
\renewcommand{\t}{\tau}
\renewcommand{\a}{\alpha}
\renewcommand{\b}{\beta}

\newcommand{\g}{\gamma}

\renewcommand{\phi}{\varphi}
\renewcommand{\d}{\delta}
\renewcommand{\o}{\omega}
\newcommand{\e}{\varepsilon}

\begin{document}
\title[The algebraic closure of the field of power series]{About the algebraic closure of the field of power series in several variables in characteristic zero}
\author{Guillaume Rond}

\email{guillaume.rond@univ-amu.fr}
\address{Aix-Marseille Universit\'e, CNRS, Centrale Marseille, I2M, UMR 7373, 13453 Marseille, France}


\thanks{This work has been partially supported  by ANR projects STAAVF (ANR-2011 BS01 009) and SUSI (ANR-12-JS01-0002-01)}

\begin{abstract}
We begin this paper by  constructing different algebraically closed fields containing an algebraic closure of the field of power series in several variables over a characteristic zero field. Each of these  fields depends on the choice of an Abhyankar valuation and is constructed via a generalization of the Newton-Puiseux method for this valuation. \\
Then we  study the Galois group of a polynomial with power series coefficients. In particular by examining more carefully  the case of monomial valuations we are able to give several results concerning the Galois group of a polynomial whose discriminant is  a weighted homogeneous polynomial times a unit. One of our main results is a generalization of Abhyankar-Jung Theorem for such  polynomials, classical Abhyankar-Jung Theorem being devoted to polynomials whose discriminant is a monomial times  a unit. 
\end{abstract}

 \maketitle

 \tableofcontents
  
 \section{Introduction}
When $\k$ is an algebraically closed field of characteristic zero, we can always express the roots of a polynomial with coefficients in the field of power series over $\k$, denoted by $\k\(t\)$, as formal Laurent series in $t^{\frac{1}{k}}$ for some positive integer $k$. This result was known by  Newton  (at least formally see \cite{B-K} p. 372) and had been rediscovered by Puiseux in the complex analytic case \cite{Pu1}, \cite{Pu2} (see \cite{B-K} or \cite{Cu} for a presentation of this result). A modern way to reformulate this fact is to say that an algebraic closure of $\k\(t\)$ is the field of Puiseux power series $\P$ defined in the following way:
$$\P:=\bigcup_{k\in\N}\k\left(\!\left(t^{\frac{1}{k}}\right)\!\right).$$
The proof of this result, called the Newton-Puiseux method, consists essentially in constructing the roots of a polynomial $P(Z)\in\k\lb t\rb[Z]$ by successive approximations in a similar way to Newton method in numerical analysis. These approximations converge since   $\k\left(\!\left(t^{\frac{1}{k}}\right)\!\right)$ is a complete field with respect to the Krull topology.\\
This result, applied to a polynomial with coefficients in $\k\lb t\rb$ defining a germ of algebroid plane curve $(X,0)$, provides an uniformization of this germ, i.e. a parametrization  of this germ.\\
On the other hand this description of the algebraic closure of $\k\(t\)$ describes very easily the Galois group of $\k\(t\)\lgw \P$, since this one is generated by the multiplication of the $k$-th roots of unity by $t^{\frac{1}{k}}$ for any positive integer $k$. In particular if an irreducible monic polynomial $P(Z)\in\C\lb t\rb[Z]$ has a root which is a  convergent power series in $t^{\frac{1}{k}}$, i.e. an element of $\C\{t^{\frac{1}{k}}\}$, then its other roots are also in $\C\{t^{\frac{1}{k}}\}$ and the coefficients of $P(Z)$ are convergent power series.\\
When $\k$ is a characteristic zero field (but not necessarily algebraically closed), we can prove in the same way that an algebraic closure of $\k\(t\)$ is
\begin{equation}\label{NP}\P:=\bigcup_{\k'}\bigcup_{k\in\N}\k'\left(\!\left(t^{\frac{1}{k}}\right)\!\right).\end{equation}
where the first union runs over all finite field extensions $\k\lgw \k'$.\\

The aim of this work is  double:  the first one consists in finding representations of the roots of a polynomial whose coefficients are power series in several variables over a characteristic zero field. Our main results regarding these representations are  Theorem \ref{main} for Abhyankar valuations and its stronger version for monomial valuations (see Theorem \ref{main2}). The second goal is to describe  the Galois group of such polynomials. In particular we concentrate our study to irreducible polynomials that remain irreducible as polynomials with coefficients in the completion of the valuation ring associated to a monomial valuation. Our main result regarding this problem is a generalization of Abhyankar-Jung Theorem to polynomials whose discriminant is weighted homogeneous (see Theorems \ref{AJ_gen} and \ref{strong_analytic}).\\

But let us present in more details the situation, the problems and the results given in this paper.
It is tempting to find such a similar expression to \eqref{NP} for the algebraic closure of the field of power series in  $n$ variables, $\k\(x_1, \ldots ,x_n\)$, for $n\geq 2$. But it appears easily that the algebraic closure of this field admits a really more complicated description and considering only power series depending on  $x_1^{\frac{1}{k}}$, \ldots , $x_n^{\frac{1}{k}}$ is not sufficient. For instance it is easy to see that a square root  of $x_1+x_2$ can not be expressed as such a power series.\\
Nevertheless there exist positive results in some specific cases, the most famous one being the Abhyankar-Jung theorem:
\begin{theo}[Abhyankar-Jung Theorem]
If $\k$ is an algebraically closed field of characteristic zero, then any polynomial with coefficients in $\k\lb x_1, \ldots ,x_n\rb$, whose discriminant has the form $ux_1^{\a_1} \ldots  x_n^{\a_n}$ where $u\in\k\lb x_1, \ldots ,x_n\rb$ is a unit and $\a_1$, \ldots , $\a_n\in\Z_{\geq 0}$, has its roots in $\k\lb x_1^{\frac{1}{k}}, \ldots ,x_n^{\frac{1}{k}}\rb$ for some positive integer $k$.
\end{theo}
 Such a polynomial is called a quasi-ordinary polynomial and this theorem asserts that the roots of quasi-ordinary polynomials are Puiseux power series in several variables. It provides not only a description of the roots of a quasi-ordinary polynomial but also a description of its Galois group. This result has first being proven by Jung in the complex analytic case, then by Abhyankar in the general case (\cite{J}, \cite{Ab}).\\
 In the general case, a  naive approach involves the use of Newton-Puiseux theorem $n$ times (i.e. the formula (\ref{NP}) for the algebraic closure of $\k\(t\)$). For example in the case where $n=2$ and $\k$ is an algebraically closed field of characteristic zero, this means that the algebraic closure of  $\k\(x_1,x_2\)$ is included in 
$$\L:=\bigcup_{k_2\in\N}\bigcup_{k_1\in\N}\k\left(\!\left(x_1^{\frac{1}{k_1}}\right)\!\right)\left(\!\left(x_2^{\frac{1}{k_2}}\right)\!\right).$$
But this field, which is algebraically closed, is very much larger than the algebraic closure of $\k\(x_1,x_2\)$ (see \cite{Sa} for some thoughts about this). Moreover the action of the $k_1$-th and  $k_2$-th roots of unity are not sufficient to generate the Galois group of the algebraic closure since there exist elements of $\k\(x_1\)\(x_2\)$ which are algebraic over $\k\(x_1,x_2\)$ but are not in $\k\(x_1,x_2\)$. For instance consider
$$x_1\sqrt{1+\frac{x_2}{x_1}}=\sum_{i\in\Z_{\geq 0}}a_i\frac{1}{x_1^{i-1}}x_2^i\in\Q\(x_1\)\(x_2\)\backslash\Q\(x_1,x_2\)$$
for some well chosen rational numbers $a_i\in\Q$, $i\in\Z_{\geq 0}$.\\
Nevertheless a deeper analysis of the Newton-Puiseux method leads to the fact that it is enough to consider the field of fractions of the ring of  elements 
$$\displaystyle f=\sum_{(l_1,l_2)\in\Z^2}a_{l_1,l_2}x_1^{\frac{l_1}{k_1}}x_2^{\frac{l_2}{k_2}}\in \L$$
 for some $k_1$, $k_2\in\N$ whose support  is included in a rational strongly convex cone of $\R^2$. Here the support of $f$ is the set
 $$\Supp(f):=\{(l_1,l_2)\in\Z^2\ /\ a_{l_1,l_2}\neq 0\}.$$
 This result has been proven by MacDonald \cite{MD} (see also \cite{Go}, \cite{Aro}, \cite{A-I}, \cite{SV}). But once more, for any rational strongly convex cone of  $\R^2$, denoted by $\s$,  $\R_{\geq 0}^2\subsetneq\s$, there exist elements whose support is in $\s$ but that are not algebraic over $\k\(x_1,x_2\)$.\\
\\
One of the main difficulties  comes from the fact that $\k\(x_1, \ldots ,x_n\)$ is not a complete field with respect to the topology induced by the maximal ideal of $\k\lb x_1, \ldots ,x_n\rb$ (called the Krull topology; it is induced by the following norm $\left|\frac{f}{g}\right|:=e^{\ord(g)-\ord(f)}$ for any $f$, $g\in\k\lb x_1, \ldots ,x_n\rb$, $g\neq 0$, where $\ord(f)$ is the order of the series $f$ in the usual sense). Indeed, in order to apply the Newton-Puiseux method we have to work with a complete field  since the roots are constructed by  successive approximations. A very natural idea is to replace  $\k\(x_1, \ldots ,x_n\)$ by its completion. But  the completion of $\k\(x_1, \ldots ,x_n\)$ is not algebraic over $\k\(x_1, \ldots ,x_n\)$, thus the fields we construct in this way are bigger than the algebraic closure of $\k\(x_1, \ldots ,x_n\)$. In fact we need to replace the completion of $\k\(x_1, \ldots ,x_n\)$ by its henselization in the completion. The problem is that there is no general criterion to distinguish elements of the henselization from other elements of the completion. In some sense this problem is analogous to the fact that there is no general criterion to determine if a real number is algebraic or not over the rationals. One more issue is that choosing the Krull topology is arbitrary and we may replace this one by any topology induced by an other norm (or valuation) on this field. \\
\\
\\
In this paper we first investigate the use of the Newton-Puiseux method with respect to "tame" valuations (i.e. replace $\k\(x_1, \ldots ,x_n\)$ by its completion for this valuation). By a "tame" valuation we  mean a rank one (or real valued)  valuation that satisfies the equality in the Abhyankar inequality (see Definition \ref{abhy}). These valuations are called Abhyankar valuations (cf. \cite{ELS}) or quasi-monomial valuations (cf. \cite{F-J}) and, essentially, these are monomial valuations after some sequence of blowing-ups. This is the first part of this work. If $\nu$ is such a  valuation, we denote by $\wdh{\K}_{\nu}$ the completion of $\k\(x_1, \ldots ,x_n\)$ for the topology induced by this valuation. This field will play the  role of $\k\(t\)$ in the classical Newton-Puiseux method. Then we have to define the elements that will play the role of $t^{\frac{1}{k}}$. This is where the first difficulty appears, since instead of working over $\wdh{\K}_{\nu}$, we need to work over the graded ring associated to $\nu$. Both are isomorphic but there is no canonical isomorphism between them. In the case of $\k\(t\)$ where $t$ is a single variable, such an isomorphism is defined by identifying the $\k$-vector space of homogeneous elements of degree $i$ of the graded ring with the $\k$-vector space of homogeneous polynomials of degree $i$, i.e. $\k.t^i$. But this identification depends on the choice of an uniformizer of $\k\lb t\rb$. In the case of $\k\(x_1, \ldots ,x_n\)$ an isomorphism will be determined by the choice of  "coordinates" such that the valuation $\nu$ is monomial in these coordinates since Abhyankar valuations are monomial valuations after a sequence of blow-ups (cf. Remark \ref{iso}). This is the reason why we restrict our study to these valuations. Nevertheless when such an isomorphism is chosen, we are able to define  the elements that will play the role of $t^{\frac{1}{k}}$, this the aim of Section \ref{part_homo}. These elements are called \textit{homogeneous elements with respect to $\nu$} (cf. Definitions \ref{homo} and \ref{integral}).  These are defined as being the roots of weighted homogeneous polynomials with coefficients in the graded ring of $\k\lb x_1, \ldots ,x_n\rb$  for the valuation $\nu$. If $\k$ is the field of complex numbers and the weights of the monomial valuation are positive integers, we can think about these homogeneous elements as weighted homogeneous algebraic (multivalued) functions. In fact we can replace $\wdh{\K}_{\nu}$ by a smaller field, the subfield of $\wdh{\K}_{\nu}$ whose elements have support  included in a finitely generated sub-semigroup  of $\R_{\geq 0}$. Let us remark that this field is similar to the field of generalized power series $\cup_{\G}\C\(t^{\G}\)$ where the sum runs over all finitely generated semigroups $\G$ of $\R_{\geq 0}$ (see \cite{Ri} for instance). Our first result is  that the inductive limit of the extensions of $\wdh{\K}_{\nu}$ by homogeneous elements with respect to $\nu$ is algebraically closed (see Theorem \ref{main}). This field is  $\underset{\underset{\g_1, \ldots , \g_s}{\lgw}}{\lim}\,\wdh{\K}_{\nu}[\g_1, \ldots , \g_s]$ where the limit runs over all subsets $\{\g_1, \ldots ,\g_s\}$ of homogeneous elements with respect to $\nu$ and is denoted by $\ovl{\K}_{\nu}$. The field extension $\k\(x_1, \ldots ,x_n\)\lgw \ovl{\K}_{\nu}$ factors through the field extension $\k\(x_1, \ldots ,x_n\)\lgw\wdh{\K}_{\nu}$. While the Galois group of the field extension $\wdh{\K}_{\nu}\lgw\ovl{\K}_{\nu}$ is easily described by the Galois group of  weighted homogeneous polynomials, the Galois  group of the algebraic closure of $\k\(x_1, \ldots ,x_n\)$ in $\wdh{\K}_{\nu}$ is more complicated. So it is very natural to study irreducible polynomials over $\k\(x_1, \ldots ,x_n\)$ which remain irreducible over $\wdh{\K}_{\nu}$, since their Galois groups are described by the Galois groups of weighted homogeneous polynomials.  Proposition \ref{cor_factor_limit} shows that this property is an open property with respect to the topology induced by the chosen valuation. Let us mention that these polynomials are called $\nu$-analytically irreducible polynomials in \cite{Te} and their study is motivated by the construction of key polynomials for Abhyankar valuations (not necessarily of  rank 1) in order to prove local uniformization.\\
\\
Then we investigate more deeply the particular case of monomial valuations.  In Section \ref{growth}, using an idea of Tougeron \cite{To} based on a work of Gabrielov \cite{Ga}, for any monomial  valuation $\nu$ we construct a field, smaller than the ones constructed previously using the Newton-Puiseux method, and containing an algebraic closure of $\k\(x_1, \ldots ,x_n\)$. The main result (see Theorem \ref{main2}) is a non-archimedean version of Eisenstein Theorem (classical Eisenstein Theorem concerns algebraic power series over $\Q$). The tool we use here is an effective version of the Implicit Function Theorem (see Proposition \ref{IFT}). The elements we need to consider are of the form 
\begin{equation}\label{touge}\sum_{i\in\La}\frac{a_i}{\d^{m(i)}}\end{equation}
where the $a_i$ and $\d$ are weighted homogeneous polynomials for the weights corresponding to the given monomial valuation, $\La$ is a finitely sub-semigroup of $\R_{\geq 0}$, $\nu\left(\frac{a_i}{\d^{m(i)}}\right)=i$ for all $i\in\La$ and $i\lgm m(i)$ is bounded by a an affine function. In the particular case where the weights are $\Q$-linearly independent this corresponds to  the result of MacDonald (see Theorem \ref{MD}).
\\
\\
In Section \ref{part_AJ} we use this description of the roots of polynomials with coefficients in $\C\{x_1, \ldots ,x_n\}$ to make a topological and complex analytical study of such polynomials whose discriminant is a weighted homogeneous polynomial multiplied by a unit. This study has been inspired by the work of Tougeron in \cite{To} and more particularly by Remarque 2.7  of \cite{To} where it is noticed that the elements of the form (\ref{touge}) define analytic functions on an open domain of $\C^n$ which is the complement of some hornshaped neighborhood of $\{\d=0\}$ (see Definition \ref{D}). This study is possible in the case of monomial valuations whose weights are positive integers. To obtain the same results in the case of general monomial valuations we need to approximate general monomial valuations by divisorial monomial valuations, i.e. monomial valuations whose weights are positive integers. This is the subject  of Section \ref{section_appro}. \\
\\
One of the main results we obtain in Section \ref{part_AJ} is the following theorem which gives a criterion for an irreducible polynomial over $\k\(x_1, \ldots ,x_n\)$  to remain irreducible over $\wdh{\K}_{\nu}$:

\begin{theo}\emph{\textbf{\ref{AJ_gen}}}
Let $\k$ be a field of characteristic zero and $\a\in\R_{>0}^n$. Let $\x$ denotes the set of variables $(x_1, \ldots ,x_n)$ and let $\nu_\a$ be the monomial valuation given by the weights $\a_i$. Let $P(Z)\in\k\lb \x\rb[Z]$ be a monic polynomial whose discriminant is equal to $\d u$ where $\d\in\k[\x]$ is a weighted homogeneous polynomial for the weights $\a_1$, \ldots , $\a_n$ and $u\in\k\lb\x\rb$ is  a unit. If $P(Z)$ factors as $P(Z)=P_1(Z) \ldots  P_s(Z)$ where $P_i(Z)$ is an irreducible monic polynomial of $\k\lb\x\rb[Z]$, then $P_i(Z)$ is irreducible in $\wdh V_{\a}[Z]$ where $\wdh V_\a$ denotes the completion of the valuation ring of $\nu_\a$. \end{theo}

\noindent Then we show that Abhyankar-Jung Theorem is in fact a generalization of this result when the $\a_i$ are $\Q$-linearly independent (see Corollary \ref{AJ}) and we give   the following generalization of Abhyankar-Jung Theorem for polynomials whose discriminant is weighted homogeneous with respect to weights $\a_1$, \ldots , $\a_n\in\R_{>0}$:

\begin{theo}\emph{\textbf{\ref{strong_analytic}} }
We assume that the hypothesis of Theorem \ref{AJ_gen} are satisfied. Let us set $N:=\dim_{\Q}(\Q\a_1+\cdots+\Q\a_n)$. Then there exist $\g_1$, \ldots , $\g_N$ integral homogeneous elements with respect to $\nu_{\a}$ and a  weighted homogeneous  polynomial for the weights $\a_1$, \ldots , $\a_n$ denoted by $c(\x)\in\k[\x]$ such that the roots of $P(Z)$ are in $\frac{1}{c(\x)}\k'\lb\x\rb[\g_1, \ldots , \g_N]$ where $\k\lgw \k'$ is a finite field extension.

\end{theo}

\noindent Indeed in the case $N=n$, i.e. $\a_1$, \ldots , $\a_n$ are $\Q$-linearly independent, the only weighted homogeneous polynomials are the monomials and the integral homogeneous elements with respect to $\nu_{\a}$ are of the form $\x^{\b}$ where  $\b\in\Q_{\geq 0}^n$ (see Remark \ref{homo_mono'}). Abhyankar-Jung Theorem simply asserts that we may choose $c(\x)=1$, a fact that we are able to prove in this case (see Corollary \ref{AJ}).\\
We remark that this result (along with Theorem \ref{AJ_gen}) shows  that the Galois group of an irreducible monic polynomial with coefficients in $\k\lb x_1, \ldots ,x_n\rb$ whose discriminant is weighted homogeneous is generated by the Galois group of one weighted homogeneous polynomial (see Remark \ref{Galois}).\\
\\
Finally in Section \ref{part_diop} we give a result of Diophantine approximation (it is just an direct generalization of \cite{Ro} and \cite{I-I}) that gives a necessary condition for an element of $\wdh{\K}_{\nu}$ to be algebraic over $\k\(x_1, \ldots ,x_n\)$.\\
\\
At the end we give a list of notations for the convenience of the reader.\\
\\
Let us mention that this work has been motivated by the understanding of the paper \cite{To} of Tougeron where the study we make for monomial valuations is made in the case of the $(x_1, \ldots ,x_n)$-adic valuation of $\k\(x_1, \ldots ,x_n\)$.\\
\\
I would like to thank Guy Casale and Adam Parusi\'nski for their answers to my questions regarding the proofs of Lemma \ref{monodromy} and Lemma \ref{D2} respectively. I also thank H. Mourtada for the valuable discussions we had on these problems and his comments that helped to improve the presentation fo this paper. I also thank the referees for their valuable suggestions.


\section{Notations and Abhyankar valuations}
Let $\N$ denote the set of positive integers and $\Z_{\geq 0}$ the set of non-negative integers.  Let $\x$  denote the multi-variable $(x_1, \ldots ,x_n)$ where $n\geq 2$. Let $\k$ denote a characteristic zero field. Then  $\k\lb \x\rb  =\k\lb x_1, \ldots ,x_n\rb$ denotes the ring of formal power series in $n$ variables over  $\k$ and we denote by $\K_n$ its fraction field and by $\m$ its maximal ideal. \\

When $(A,\m)$ is a local domain, a \emph{valuation} on $A$ is a function $\nu : A\backslash\{0\}\lgw \G^+$, where $\G$ is an ordered subgroup of $\R$ and $\G^+:=\G\cap\R_{\geq 0}$, such that 
$$\nu(fg)=\nu(f)+\nu(g)\text{ and }\nu(f+g)\geq \min\{\nu(f),\nu(g)\}\ \ \ \forall f,g\in A.$$
 We will also impose that $\nu(f)>0$ if and only if $f\in\m$. We set  $\nu(0)=\infty$ where $\infty>i$ for any $i\in\G$.\\
 Such valuation $\nu$ extends to $\K_A$, the fraction field of $A$, by 
 $$\nu\left(\frac{f}{g}\right):=\nu(f)-\nu(g)$$ for any $f$, $g\in A$, $g\neq 0$. We will always assume that $\nu : \K_A\lgw \G$ is surjective. In this case $\G$ is called the \textit{value group} of $\nu$. The image of $A\backslash\{0\}$ by $\nu$ is called the \emph{semigroup of $\nu$} and we denote it by $\Sigma$. Then $\G$ is the group generated by $\Sigma$. Let us denote by $V_{\nu}$ the valuation ring of $\nu$: 
$$V_{\nu}:=\left\{\frac{f}{g}\ /\  f,g\in A,\ \nu(f)\geq\nu(g)\right\}.$$
This is a local ring whose maximal ideal, denoted by $\m_V$, is the set of elements $f/g$ such that $\nu(f/g)>0$. Its residue field $\frac{V_{\nu}}{\m_V}$ is denoted by $\k_\nu$.\\
Let us denote by $\wdh{V}_{\nu}$ the completion of $V_\nu$  which is defined as follows: For any $\l\in \G$ let us set $I_\l:=\{v\in V_\nu\ / \ \nu(v)\geq \l\}$. The family of ideals $\{I_{\l}\}_{\l\in \G}$ as a system of neighbourhoods of 0 makes $V_{\nu}$ into a topological ring. Then $\wdh{V}_\nu$ is the completion of $V_\nu$ for this topology. We can also remark that the family $\{V_\nu/I_\l\}_\l$ is an inverse system and its inverse limit is exactly $\wdh V_\nu$.\\
Then $\wdh{V}_{\nu}$ is an equicharacteristic complete valuation ring and its residue field is isomorphic to $\k_{\nu}$.\\

In this paper we will only consider a particular case of valuations, called Abhyankar valuations:
\begin{definition}\label{abhy}
A valuation $\nu$ is called an \textit{Abhyankar valuation} if the following equality holds:

$$\trdeg_{\k}\k_{\nu}+\dim_{\Q}\G\otimes_{\Z}\Q=n.$$
This equality is called the \emph{Abhyankar's Equality}.

\end{definition}

\begin{rmk}
If $\dim_{\Q}\G\otimes\Q=1$, then $\G\simeq\Z$. Otherwise $\G$ is a dense subgroup of $\R$. 
\end{rmk}

\begin{ex} The first example is the $\m$-adic valuation denoted by ord on the ring $A=\k\lb x\rb$, and defined by 
$$\ord(f):=\max\{n\in\N\ /\ f\in\m^n\}\ \ \ \ \forall f\in\k\lb \x\rb \backslash\{0\}.$$
In this case its value group $\G$ is equal to $\Z$ and its semigroup $\Sigma$ is equal to $\Z_{\geq 0}$.
\end{ex}

\begin{ex}\label{monomial}
Let $\a:=(\a_1, \ldots ,\a_n)\in(\R_{>0})^n$. Let us denote by $\nu_{\a}$ the monomial valuation on $A=\k\lb x\rb$ defined by $\nu_{\a}(x_i):=\a_i$ for $1\leq i\leq n$. For instance $\nu_{(1, \ldots ,1)}=\ord$. \\
Here we have $\G=\Z\a_1\oplus\cdots\oplus\Z\a_n$ and $\Sigma=\Z_{\geq 0}\a_1\oplus\cdots\oplus\Z_{\geq 0}\a_n$.
\end{ex}

\begin{ex}\label{divisorial}
If $\G$ is isomorphic to $\Z$ and $\nu$ is an Abhyankar valuation, then $\nu$ is a \textit{divisorial valuation}. For such valuation there exists a proper birational dominant map $\pi : X\lgw \Spec(\k\lb \x\rb )$ and $E$ an irreducible component of the exceptional locus of $\pi$ such that $\nu$ is the composition of $\pi^*$ with the $\m_E$-adic valuation of the ring $\O_{X,E}$.
\end{ex}

\begin{rmk}\label{ELS}
Geometrically, an Abhyankar valuation is a monomial valuation at a point lying on the exceptional divisor $E$ of some proper birational map $(Y,E)\lgw (\k^n,0)$. More precisely we have the following:\\
The restriction of $\nu$ to $\k[\x]$ is an Abhyankar valuation with the same value group as $\nu$. We denote it by $\wdt \nu$. By Proposition 2.8 \cite{ELS}  there exists a regular local domain $(A,\m_A)$, an injective morphism
$$\pi  : \k[\x]\lgw A$$
inducing an isomorphism between the fields of fractions and a regular system of parameter $z_1$, \ldots , $z_r$ of $A$ such that $\wdt \nu(z_1)$, \ldots , $\wdt \nu(z_r)$ freely generate the value group of $\wdt \nu$ (or the value group of $\nu$ since both are equal). Let us denote by $\mu$ the restriction of $\wdt \nu$ to $A$. Then $\pi $ induces an isomorphism between $V_{\nu}$ and $V_{\mu}$. Thus it induces an isomorphism between $\wdh V_{\wdt \nu}$ and $\wdh V_{\mu}$. Moreover the completion of $A$ is isomorphic to  $\L\lb z_1, \ldots ,z_r\rb$ where $\k\lgw \L$ is a field extension of transcendence degree $n-r$ (here $\L=\frac{A}{\m_A}$) and $\mu$ extends to a valuation on $\wdh A$ which is exactly the monomial valuation that sends $z_i$ onto $\nu(z_i)$ for all $i$.

\end{rmk}


\begin{rmk}
If $n=2$, in fact any discrete valuation (i.e. $\G=\Z$)  is an Abhyankar valuation \cite{HOV}.
\end{rmk}
\begin{definition}\label{alpha-hom}
Let $\a\in\R_{>0}^n$.
A polynomial $f\in\k\lb \x\rb$ is called $(\a)$-homogeneous of degree $i$ is every nonzero  monomial $c\x^\b$ of $f$ satisfies
$$\sum_{k=0}^n\a_k\b_k=i$$
or equivalently $\nu_\a(c\x^\b)=i$.
This means that $f$ is weighted homogeneous of degree $i$ where $x_j$ has weight $\a_j$ for every $j$.
\end{definition}

 \begin{ex}\label{ex_monomial}
Let $\nu_{\a}$ be a monomial valuation as before. Any power series $g\in\k\lb \x\rb $ can be written $g=\sum_{i\in\Sigma}g_i$ where $g_i$ is a $(\a)$-homogeneous polynomial of degree $i\in\Sigma$.  Let us denote by $i_0$ the least $i\in\Sigma$ such that $g_i\neq 0$. Then we can write formally
$$g=g_{i_0}\left(1+\sum_{i>i_0}\frac{g_i}{g_{i_0}}\right)$$
and this equality is satisfied in $\wdh{V}_{\nu_{\a}}$. Now if $f\in\k\lb \x\rb $, $g\neq 0$ and $\nu(f)\geq \nu(g)$ we can write
$$\frac{f}{g}=\left(\sum_i\frac{f_i}{g_{i_0}}\right)\left(1+\sum_{i>i_0}\frac{g_i}{g_{i_0}}\right)^{-1}$$
where $f=\sum_if_i$ where $f_i$  is $(\a)$-homogeneous  of degree $i\in\Sigma$.\\
Thus any element of $V_{\nu_{\a}}$ is of  the form $\displaystyle\sum_{i\geq 0, i+i_0\in \Sigma}\frac{a_i(\x)}{b_i(\x)}$ for some $i_0\in \Sigma$,  
where $a_i(\x)$ and $b_i(\x)$ are $(\a)$-homogeneous and $\nu_{\a}\left(\frac{a_i(\x)}{b_i(\x)}\right)=i$ for any $i\in\R$.\\
 On the other hand $\wdh{V}_{\nu_{\a}}$ is  the set of elements of the form  $\displaystyle\sum_{i\in \Lambda}\frac{a_i(\x)}{b_i(\x)}$ where $\Lambda$ is a finite or countable subset of $\G^+$ with no accumulation point, where $a_i(\x)$ and $b_i(\x)$ are  $(\a)$-homogeneous  and $\nu_{\a}\left(\frac{a_i(\x)}{b_i(\x)}\right)=i$ for any $i\in\R$.
 \end{ex}
Let us denote by $\wdh{\K}_{\nu}$ the fraction field of $\wdh{V}_{\nu}$. The valuation $\nu$ defines an ultrametric norm on $\wdh{\K}_{\nu}$, denoted by $|\ |_{\nu}$, defined by 
$$\left|\frac{f}{g}\right|_{\nu}=e^{\nu(g)-\nu(f)} \ \ \ \ \forall f\in\k\lb \x\rb , g\in\k\lb \x\rb \backslash\{0\}.$$
Then  $\wdh{\K}_{\nu}$ is the completion of $\K_n$ for the topology induced by this norm and this norm (thus the valuation $\nu$) extends canonically on $\wdh{\K}_{\nu}$. We shall also denote by $\nu$ the extension of $\nu$ to $\K_{\nu}$.\\
\\
Let us denote by $\K_{\nu}^{\alg}$ the algebraic closure of $\K_n$ in $\wdh{\K}_{\nu}$. We also denote by $V_{\nu}^{\alg}$ the ring of elements of  $\wdh{V}_{\nu}$ which are algebraic over $\K_{n}$: $V_{\nu}^{\alg}:=\K_{\nu}^{\alg}\cap \wdh{V}_{\nu}$. We have the following lemma:

\begin{lemma}\label{val}
The ring $V^{\alg}_{\nu}$ is a  valuation ring (associated to the valuation $\nu$) and  $\K_{\nu}^{\alg}$ is its fraction field. Moreover $V_{\nu}\lgw V_{\nu}^{\alg}$ is the henselization of $V_{\nu}$ in $\wdh{V}_{\nu}$. 
\end{lemma}
\begin{proof}
If $f$, $g\in V^{\alg}_{\nu}$ and $\nu(f)\geq \nu(g)$, then $\frac{f}{g}\in\K_{\nu}^{\alg}\cap\wdh{V}_{\nu}=V_{\nu}^{\alg}$ so $V_{\nu}^{\alg}$ is a valuation ring. For $f\in\K_{\nu}^{\alg}$ there exists $N\in\N$ such that $x_1^Nf\in\K_{\nu}^{\alg}\cap \wdh{V}_{\nu}=V^{\alg}_{\nu}$ since $\nu(x_1^N)>0$. Thus $\K_{\nu}^{\alg}$ is the fraction field of $V_{\nu}^{\alg}$.\\
By construction the elements of the henselization of $V_{\nu}$ are algebraic over $V_{\nu}$. On the other hand every element of $\wdh{V}_{\nu}$ which is algebraic over $V_{\nu}$ is in the Henselization of $V_{\nu}$ (see Corollary 1.2.1 \cite{M-B}).
\end{proof}
Thus we can summarize the situation with the following commutative diagram, where the bottom part corresponds to the quotient fields of the rings of the upper part:

$$\xymatrix{
 \k\lb \x\rb \ar[rr]  &&    V_{\nu} \ar[rr] \ar[rd] \ar[dd] && \wdh{V}_{\nu} \ar[dd]  \\
  &&  & V_{\nu}^{\alg} \ar[ru] \ar[dd]  && \\
  &&  \K_n \ar[rr] |!{[ur];[dr]}\hole \ar[rd] && \wdh{\K}_{\nu}  \\
  &&  & \K_{\nu}^{\alg}\ar[ru]  & & }$$


\section{Homogeneous elements with respect to an Abhyankar valuation}\label{part_homo}

\subsection{Graded ring of an Abhyankar valuation and support}

Let  $A$ be an integral domain and let $\nu : A\lgw \G^+$ be a valuation where $\G$ is a subgroup of $\R$.  We define $\text{Gr}_{\nu}A=\bigoplus_{i\in\G^+}\frac{\p_{\nu,i}}{\p_{\nu,i}^+}$ where $\p_{\nu,i}:=\{f\in A/\ \nu(f)\geq i\}$ and $\p_{\nu,i}^+:=\{f\in A \ /\ \nu(f)>i\}$.

\begin{definition}\label{graded}
 Let $\G^+$ be a sub-semigroup of $\R_{\geq 0}$. A  \textit{$\G^+$-graded ring} is a ring $A$ that has a direct sum of abelian groups, $A=\bigoplus_{i\in\G^+}A_i$, such that $A_iA_j\subset A_{i+j}$ for any $i$, $j\in\G^+$.\\
 For any $j\in\G^+$, $\bigoplus_{i\in\G^+, i\geq j}A_i$ is an ideal of $A$. This family of ideals as a system of neighborhoods of 0 makes $A$ into a topological whose completion is denoted by $\wdh A$ or $\wdh{\bigoplus}_{i\in\G^+}A_i$.
 The completion of $A$ is the set of elements that are written as a series $\sum_{i\in\La}a_i$ where $\La\subset \G^+$ is either a finite set, either a countable subset of $\R_{>0}$ with no accumulation point, and $a_i\in A_i$ for any $i\in\La$.\\
 A \textit{complete} ($\G^+$-)\textit{graded ring} is the completion of a ($\G^+$-)graded ring.
  \end{definition}
  
  \begin{rmk}
  Let $A$ be a complete graded ring. If $A_0$ is a field then $A$ is a local ring and its maximal ideal is $\m:=\wdh{\bigoplus}_{i>0}A_i$.\\
  For any $a\in A$  we can write $a=\sum_{i\in\La}a_i$ where $a_i\in A_i$ for any $i$. If $a\neq 0$ let us set $\nu(a):=\min\{i\in\G^+\ / \ a_i\neq 0\}$. Set $\nu(0)=\infty$. Then $\nu$ is an order function, i.e. $\nu(ab)\geq \nu(a)+\nu(b)$ and $\nu(a+b)\geq\min\{\nu(a),\nu(b)\}$. Moreover $\nu$ is a valuation if and only if $A$ is an integral domain. The order function $\nu$ is called the \emph{order function of $A$}.
  \end{rmk}
  
  \begin{ex}
  For a given Abhyankar valuation $\nu$ on $\k\lb \x\rb $ the rings $\text{Gr}_{\nu}\k\lb \x\rb $ and $\text{Gr}_{\nu} V_{\nu}$ are  $\G^+$-graded rings and $\wdh{\text{Gr}_{\nu}\k\lb \x\rb }$ and $\wdh{\text{Gr}_{\nu}V_{\nu}}$ are complete $\G^+$-graded rings. 
  \end{ex}

\begin{rmk}
The ring $\wdh{\text{Gr}_{\nu} V_{\nu}}$ is isomorphic to the ring of generalized power series $\k_{\nu}\lb t^{\G^+}\rb$ where $t$ is a single variable.
\end{rmk}

\begin{rmk}\label{supp_compl}
The elements of $\wdh{\text{Gr}_{\nu}V_{\nu}}$ are the elements of the form $\sum_{i\in\La}a_i$ where $a_i\in \frac{\p_{\nu,i}}{\p_{\nu,i}^+}$
 for all $i\in\La$ where $\La$ is either a finite set, either a countable subset of $\R_{\geq 0}$ with no accumulation point.
 \end{rmk}

\begin{rmk}\label{iso}
Let us consider a monomial valuation $\nu$ on $\k\lb \x\rb $, let us say $\nu:=\nu_{\a}$ where $\a\in\R_{>0}^n$.\\
In this case $\frac{\p_{\nu,i}}{\p_{\nu,i}^+}$ is isomorphic to the $\k$-vector space of rational fractions $\frac{a(\x)}{b(\x)}$ where $a(\x)$ and $b(\x)$ are $(\a)$-homogeneous polynomials   and $\nu_{\a}\left(\frac{a(\x)}{b(\x)}\right)=i$. Thus, by Example \ref{ex_monomial} $\wdh{\text{Gr}_{\nu}V_{\nu}}$ and $\wdh{V}_{\nu}$ are $\k$-isomorphic.\\
Let us now consider a general Abhyankar valuation $\nu$ on $\k\lb \x\rb $. By Remark \ref{ELS} there exist a regular local domain $(A,\m_A)$, an injective morphism
$$\pi : \k[\x]\lgw A$$
inducing an isomorphism between the fields of fractions and such that, if we denote by $\mu$ the restriction of $\nu$ to $A$, the following properties hold:\\
The extension of $\mu$ to $\wdh{A}$ is a monomial valuation (denoted by $\wdh \mu$) and $\pi$ induces  isomorphisms $V_\nu\simeq V_{\mu}$ and $\wdh V_\nu\simeq \wdh V_{\wdh\mu}$.\\ We have
 $\wdh{V}_{\mu}=\wdh{V}_{\wdh{\mu}}$ and $\text{Gr}_{\nu}V_{\nu}\simeq\text{Gr}_{\mu}V_\mu=\text{Gr}_{ \mu}\wdh V_\mu$. Thus  $\wdh{\text{Gr}_{\nu} V_{\nu}}$ and $\wdh{V}_{\nu}$ are $\k$-isomorphic by the monomial case.
 \\
 We can summarize this in the following proposition:
 \end{rmk}
 
 \begin{prop}\label{iso_graded}
The choice of a proper birational map $\pi$ and parameters $z_1$, \ldots , $z_r$  as in Remark \ref{ELS} yields an isomorphism
  $$\wdh{\text{Gr}_{\nu} V_{\nu}}\simeq \wdh{V}_{\nu}.$$
  
\end{prop}

\begin{rmk}

 A different choice of $\pi$ and $z_1$, \ldots , $z_r$ would give an other isomorphism between these two rings.
 \end{rmk}

\begin{definition}\label{support1}
Let $A=\wdh{\bigoplus}_{i\in\G^+}A_i$ be a complete $\G^+$-graded ring. Let $a\in A$, $a=\sum_{i\in\G^+}a_i$, $a_i\in A_i$ for any $i$. The support of $a$ is the subset $I$ of $\G^+$ defined by $i\in I$ if and only if $a_i\neq 0$. We denote this set $I$ by $\Supp(a)$.
\end{definition}

\begin{definition}\label{support2}
Let $\nu$ be an Abhyankar valuation defined on $\k\lb \x\rb $. 
Let us fix a $\k$-isomorphism $\phi$ between  $\wdh{\text{Gr}_{\nu} V_{\nu}}$ and $\wdh{V}_{\nu}$ as in  Proposition \ref{iso_graded}. Let $a\in\wdh{V}_{\nu}$ and let us write $\phi(a)=\sum_{i\in\G^+}a_i$ with $a_i\in\frac{\p_{\nu,i}}{\p_{\nu,i}^+}$. The $\nu$-support with respect to $\phi$ of $a$ is the subset of $\G^+$ defined as
$$\Supp_{\nu,\phi}(a):=\{i\in\G^+\ /\ a_i\neq0\}.$$
When the isomorphism is clear from the context we will skip the mention of $\phi$ and denote the $\nu$-support of $a$ by $\Supp_{\nu}(a)$.

\end{definition}

\begin{prop}\label{Support}
Let  $\nu$ be an Abhyankar valuation on $\k\lb \x\rb $ and let $\phi$ be  a $\k$-isomorphism between  $\wdh{\text{Gr}_{\nu}V_{\nu}}$ and $\wdh{V}_{\nu}$ as in Proposition \ref{iso_graded}.  Then there exists a finitely generated sub-semigroup of $\R_{\geq0}$, denoted by $\La$, such that the $\nu$-support of any element of $\k\lb \x\rb $ with respect to $\phi$ is included in $\La$. 
\end{prop}

\begin{proof}
By Remark \ref{ELS}, we may assume that $\nu$ is a monomial valuation. Thus the proposition comes from the following lemma applied to $\Sigma=\Z_{\geq 0}^n$:
\end{proof}
\begin{lemma}\label{Gordan}
 Let $\Sigma$ be a strongly convex rational cone of $\R^n$. Let $\a\in\R_{>0}^n$ such that $\langle \a,\b\rangle>0$ for any $\b\in\Sigma$, $\b\neq 0$. Then there exists a finitely generated subgroup of $\R_{\geq 0}$, denoted by $\La$, such that $\Supp_{\nu_{\a}}(f)\subset \La$ for any $f\in\k\lb x^{\b},\b\in\Sigma\cap\Z^n\rb$ where
 $\k\lb x^{\b},\b\in\Sigma\cap\Z^n\rb$ denotes the ring of formal Laurent series whose support is included in $\Sigma\cap \Z^n$.
 \end{lemma}
 \begin{proof}
 By Gordan Lemma, $\Sigma\cap\Z^n$ is a finitely generated semigroup, let us say $\Sigma\cap\Z^n$ is generated by $u_1$, \ldots , $u_k$. Let us set $r_i:=\langle\a,u_i\rangle$, $1\leq i\leq k$. Since any element of $\Sigma\cap \Z^n$ is a $\Z_{\geq 0}$-linear combination of $u_1$, \ldots , $u_k$, then $\langle \a,\b\rangle$ is a $\Z_{\geq 0}$-linear combination of $r_1$, \ldots , $r_k$ for any $\b\in\Sigma\cap\Z^n$. Let us denote by $\La$ the semigroup of $\R_{\geq 0}$ generated by $r_1$, \ldots , $r_k$. Then $\Supp_{\nu_{\a}}(f)\subset \La$.
 \end{proof}
 
 \begin{rmk}
 Proposition \ref{Support} does not imply that the semigroup $\Sigma$ of $\nu$ is finitely generated, which is not true in general for Abhyankar valuations which are not monomial valuations.
 \end{rmk}
 

\subsection{Homogeneous elements}
  
 From now on we  fix an Abhyankar valuation $\nu$ on $\k\lb \x\rb $ and a $\k$-isomorphism $\phi$ between  $\wdh{\text{Gr}_{\nu}V_{\nu}}$ and $\wdh{V}_{\nu}$ induced by an injective birational morphism $\pi$  as in Remark \ref{iso} and we will skip to mention  it in the following. There are several reasons for that. The first one is that we are interested in effective results on the algebraic elements over $\k\lb \x\rb $, thus we are interested by valuations which are given effectively and this will be the case essentially through a map $\pi$ as in Remark \ref{ELS}. In particular we will investigate more deeply the case of monomial valuations and, in this case, the set of variables $x_1$, \ldots , $x_n$ will be fixed from the beginning, thus $\phi$ is quite natural in this case. The last reason is that we will give properties on the $\nu$-support of algebraic elements, and Proposition \ref{Support} will allow us to consider only elements whose $\nu$-support is included in a finitely generated sub-semigroup of $\R_{>0}$, and this fact does not depend on $\phi$.
 
 \begin{definition}\label{fg} Let $\nu$ be an Abhyankar valuation defined on $\k\lb \x\rb $. We will denote by $V_{\nu}^{\fg}$ the subset of $\wdh{V}_{\nu}$ of elements whose $\nu$-support is included in a finitely generated sub-semigroup of $\R_{\geq 0}$ (when we identify $\wdh{V}_{\nu}$ and $\wdh{\text{Gr}_{\nu}V_{\nu}}$ via $\phi$). It is straightforward to check that $V_{\nu}^{\fg}$ is a valuation ring. We denote by $\K_{\nu}^{\fg}$ its fraction field.
\end{definition}

\begin{definition}\label{homo}
Let $A$ be a complete $\G^+$-graded domain and let $\nu$ be its order function (which is a valuation since $A$ is a domain). A \textit{homogeneous element with respect to $\nu$} is an element $\g$ of a finite extension of $A$ such that its minimal polynomial $Q(Z)$ is irreducible in $A[Z]$ and has the following form:
$$Z^q+g_1Z^{q-1}+\cdots+g_q$$
where $g_k\in A_{i(k)}$ with $i(k)\in \G$ for $1\leq k\leq q$ such that $k.i(l)=l.i(k)$ for all $k$ and $l$. In this case $d:=\frac{i(k)}{k}\in\frac{1}{q!}\G$ is called the \emph{order} of $\g$.
\end{definition}

\begin{ex}\label{homo_mono}
Let $\a\in\R_{>0}^n$ such that $\dim_{\Q}(\Q\a_1+\cdots+\Q\a_n)=n$ i.e. the $\a_i$ are $\Q$-linearly independent. Then the value group of $\nu_{\a}$ is the following group: $$\G=\Z\a_1+\cdots+\Z\a_n$$ and for any $i\in\G$ there exists a unique $(\b_{i,1}, \ldots ,\b_{i,n})\in\Z^n$ such that 
$$i=\b_{i,1}\a_1+\cdots+\b_{i,n}\a_n.$$
Thus if  $i\in \G^+$ this means that  $\frac{\p_{\nu_{\a},i}}{\p_{\nu_{\a},i}^+}$ is isomorphic to the one dimensional $\k_{\nu_\a}$-vector space generated by $x_1^{\b_{i,1}}\cdots x_n^{\b_{i,n}}$. Let us remark here that $\k_{\nu_\a}$ is equal to $\k$ since the $\a_i$ are $\Q$-linearly independent. 
Thus if  $g_k\in \frac{\p_{\nu_{\a},dk}}{\p_{\nu_{\a},dk}^+}$ for $1\leq k\leq q$ we have that 
$$Z^q+g_1Z^{q-1}+\cdots+g_q=x_1^{\b_{qd,1}}\cdots x_n^{\b_{qd,n}}\left(T^q+g'_1T^{q-1}+\cdots+g'_q\right)$$
where $Z=x_1^{\b_{d,1}}\cdots x_n^{\b_{d,n}}T$ and $g_1'$, \ldots , $g'_q\in \k$.  If $g_q\neq 0$ then $\b_{qd,j}\in \Z$ for any $j$ but $\b_{d,j}=\frac{\b_{qd,j}}{d}$ may not be an integer.  Then the roots of $T^q+g'_1T^{q-1}+\cdots+g'_q$ are algebraic over $\k$. Thus homogeneous elements with respect to $\nu_{\a}$ are of the form $c\x^{\b}$ where $c$ is algebraic over $\k$ and $\b\in\Q^n$ with $\langle \a,\b\rangle:=\a_1\b_1+\cdots+\a_n\b_n\geq 0$.

\end{ex}

\begin{definition}\label{integral} Let  $\nu$ be an Abhyankar valuation on $\k\lb \x\rb $.
Let $A=\wdh{\text{Gr}_{\nu} V_{\nu}}$ and $\g$ be a homogeneous element with respect to $\nu$. Let $Q(Z)$ be its minimal polynomial: $Q(Z)=Z^q+g_1Z^{q-1}+\cdots+g_q$ with $g_k\in  \frac{\p_{\nu,dk}}{\p_{\nu,dk}^+}$ for $1\leq k\leq q$. We say that $\g$ is an \textit{integral homogeneous element with respect to $\nu$} if $g_k$ is the image of an element of $\k\lb \x\rb \cap\p_{\nu,dk}$ for all $k$.
\end{definition}

\begin{ex}\label{homo_mono'}
Let $\a\in\R_{>0}^n$ such that $\dim_{\Q}(\Q\a_1+\cdots+\Q\a_n)=n$ and let us keep the notations of Example \ref{homo_mono}. Then $\g$ is an integral homogeneous element with respect to $\nu_{\a}$ if $g_k\in \frac{\k\lb \x\rb \cap\p_{\nu_{\a},dk}}{\k\lb \x\rb \cap\p_{\nu_{\a},dk}^+}$ for $1\leq k\leq q$. Since $g_q\neq 0$ this means that $\b_{qd,j}\in\Z_{\geq 0}$ for all $j$. Thus integral homogeneous elements with respect to $\nu_{\a}$ are of the form $c\x^{\b}$ where $c$ is algebraic over $\k$ and $\b\in\Q_{\geq 0}^n$. 
\end{ex}

\begin{ex}\label{homo_0}
Let $\nu$ be an Abhyankar valuation on $\k\lb \x\rb $ and let us assume that $\k$ is not algebraically closed. Let $c$ be in the algebraic closure of $\k$, $c\notin\k$.
Then $c$ is a root of a polynomial equation with coefficients in $\k$ and since $\k$ is a subfield of $\k_{\nu}$, this shows that $c$ is an integral homogeneous element of order 0 with respect to $\nu$.

\end{ex}


\begin{rmk}\label{integral'}
Let $\nu$ be an Abhyankar valuation on $\k\lb \x\rb $ and let $\g$ be a homogeneous element of order $d$ with respect to $\nu$. Let us denote by $Q(Z)$ its minimal polynomial, say $$Q(Z)=Z^q+g_1Z^{q-1}+\cdots+g_q$$
where $g_k\in\frac{\p_{\nu_{\a},dk}}{\p_{\nu_{\a},dk}^+}$ for $1\leq k\leq q$. Each $g_k$ is the image in $\wdh{\text{Gr}_{\nu} V_{\nu}}$ of some fraction $\frac{f_k}{h_k}$ where $f_k$, $h_k\in \k\lb \x\rb $. Set $h:=h_1 \ldots  h_k$, let $h_0$ be the image of $h$ in $\wdh{\text{Gr}_{\nu}V_{\nu}}$ and set $\g':=h_0\g$. Then $\g'$ is a homogeneous element annihilating $Z^q+g_1'Z^{q-1}+\cdots+g_q'$ where $g_k'$ is the image of $\frac{f_k}{h_k}h^{k-1}\in\k\lb \x\rb $ in $\wdh{\text{Gr}_{\nu}V_{\nu}}$, thus it is an integral homogeneous element with respect to $\nu$. Moreover we have
$$\Frac(\wdh{\text{Gr}_{\nu} V_{\nu}})[\g]=\Frac(\wdh{\text{Gr}_{\nu} V_{\nu}})[\g'].$$

\end{rmk}


\begin{definition}\label{a-homo}

Let $\nu$ be an Abhyankar valuation on $\k\lb \x\rb $, 
$$P(Z_1, \ldots ,Z_m)\in\wdh{V}_{\nu}[Z_1, \ldots ,Z_m]$$ and $\dd:=(d_1, \ldots ,d_m)\in\R_{>0}^m$. One says that $P(Z_1, \ldots ,Z_m)$ is $(\nu,\dd)$-homogeneous of degree $d\in\R$ if for every nonzero monomial $gZ_1^{\b_1} \ldots  Z_m^{\b_m}$ of $P(Z)$ one has $g\in \frac{\p_{\nu,k}}{\p_{\nu,k}^+}$ with $k+\b_1 d_1+\cdots+\b_m d_m=d$. \\
\end{definition}


\begin{rmk}
Let $\nu$ be an Abhyankar valuation on $\k\lb \x\rb $. Let $\g$ be a homogeneous element of order $d$ with respect to $\nu$. Let us denote by $P(Z)$ its minimal monic polynomial. Then  $P(Z)$ is $(\nu,d)$-homogeneous. \\
Conversely if $P(Z)\in\wdh{V}_{\nu}[Z]$ satisfies  $P(\g)=0$ for some element $\g$ algebraic over $\wdh V_\nu
$, and if $P(Z)$ is  a nonzero  $(\nu,d)$-homogeneous, then the divisors of $P$ in $\wdh{V}_{\nu}[Z]$ are also $(\nu,d)$-homogeneous, thus the minimal polynomial of $\g$ is $(\nu,d)$-homogeneous. Hence $\g$ is a homogeneous element of order  $d$ with respect to $\nu$. 
\end{rmk}


\begin{lemma}\label{homogenes1}
Let $\g_1$ and $\g_2$ be two homogeneous elements of order $d_1$ and $d_2$ respectively  with respect to the valuation $\nu$ and let $k\in\Z$. Then 
\begin{itemize}
\item[i)] $\g_1^k$ is homogeneous of order $kd_1$, 
\item[ii)] if $e_1d_1=e_2d_2$ with $e_1,e_2\in\N$, then $\g_1^{e_1}+\g_2^{e_2}$ is homogeneous of order $d_1e_1$,
\item[iii)] $\g_1\g_2$ is homogeneous of order $d_1+d_2$.
\end{itemize}
\end{lemma}

\begin{proof}
If $\g$ is homogeneous of order $d\in\Q$, then $\g^k$, $k\in\N$, is homogeneous of order $kd$. Indeed 
a polynomial having $\g^k$ as a root is   $Q(Z):=$Res$_X(P(X),Z-X^k)$ where  $P$ is the minimal monic polynomial of $\g$ over $\k(\x)$. But $P(X)$ is $(\nu,d)$-homogeneous  and $Z-X^k$ is $(\nu,d,kd)$-homogeneous. Thus $Q(Z)$ is $(\nu,d,kd)$-homogeneous, hence $(\nu,kd)$-homogeneous since it does not depend on $X$. This proves that  $\g^k$ is homogeneous of order $kd$.\\
\\
In order to show ii) we may assume, by i), that $\g_1$ and $\g_2$ are homogeneous of same order  $d=e_1d_1=e_2d_2$. Let us denote by $P_1(Z)$ and $P_2(Z)$ the minimal monic polynomials of $\g_1$ and $\g_2$ respectively. Then  $Q(Z):=$Res$_X(P_1(Z-X),P_2(X))$ is $(\nu,d,d)$-homogeneous, thus $(\nu,d)$-homogeneous since it does not depend on $X$. Since $Q(\g_1+\g_2)=0$,  $\g_1+\g_2$ is homogeneous of order  $d$.\\
\\
In order to show iii) let us denote by $P_1(X)$ the minimal monic polynomial  of $\g_1$ (this is a $(\nu,d_1)$-homogeneous polynomial) and $P_2(Z)$ the minimal monic polynomial of  $\g_2$ ($(\nu,d_2)$-homogeneous). Let us denote by  $k$ the degree in $Z$ of  $P_1(Z)$ and set $R(X,Y):=X^kP_1(Y/X)$. Then $\g_1\g_2$ is a root of  $Q(Z):=$Res$_X(R(X,Z),P_2(X))$. Moreover $R(X,Z)$ is $(\nu,d_2,d_1+d_2)$-homogeneous. Thus $Q(Z)$ is $(\nu,d_1+d_2)$-homogeneous, which proves that $\g_1\g_2$ is homogeneous of order  $d_1+d_2$.
\end{proof}


\begin{lemma}
Let $P(T,Z)$  be a nonzero  $(\nu,d_1,d_2)$-homogeneous polynomial of $\wdh{V}_{\nu}[T,Z]$ and let  $\g_1$ be a homogeneous element of order  $d_1$ with respect to $\nu$. If an element  $\g_2$ belonging to a finite extension of $\k(\x)$ satisfies  $P(\g_1,\g_2)=0$, then $\g_2$ is a homogeneous element of order  $d_2$ with respect to $\nu$.
\end{lemma}

\begin{proof}
Let $Q(T)\in\wdh{V}_{\nu}[T]$ be a nonzero  $(\nu,d_1)$-homogeneous polynomial such that $Q(\g_1)=0$. Let us denote  $R(Z)=$Res$_T(P(T,Z),Q(T))$. Then $R(Z)$ is a $(\nu,d_2)$-homogeneous polynomial such that $R(\g_2)=0$. This proves the result.
\end{proof}


\begin{rmk}\label{extend}
Let $A$ be a complete $\G^+$-graded integral domain, let say $A$ is the completion of $A':=\bigoplus_{i\in\G^+}A_i$, and let $\nu$ be its order valuation. Let $Q(Z)$ be an irreducible polynomial of $A[Z]$ having the following form:
$$Z^q+g_1Z^{q-1}+\cdots+g_q$$
where $g_k\in A_{dk}$ for $1\leq k\leq q$ and $d\in\frac{1}{q!}\G^+$.
The ring $B:=\frac{A[Z]}{(Q(Z))}$ is an integral domain and $\nu$ extends to a valuation of this ring by defining $\nu(Z):=d$ and 
$$\nu\left(\sum_{i=0}^{q-1}a_iZ^i\right):=\inf_i\{\nu(a_i)+di\}.$$
Let us set $B':=A'[Z]/(Q(Z))$.
Then $B$ is a complete $\frac{1}{q!}\G$-graded domain since $B$ is the completion of 
$$B'=\bigoplus_{i\in\G^+}\bigoplus_{\begin{subarray}{l} 0\leq j\leq \min\{\lfloor\frac{i}{d}\rfloor,q\}\\ 
\qquad j\in \frac{1}{d!}\G^+\end{subarray}}A_{i-dj}Z^j.$$
\end{rmk}

\begin{definition}\label{Vlr}
Let $\g$ be an algebraic element over $A$ whose minimal polynomial is the polynomial $Q(Z)$ as in the previous remark.
Then the integral domain $B$ constructed in the previous remark is  denoted by $A[\g]$. \\
By induction we can define $A[\g_1, \ldots , \g_s]$ where $\g_{i+1}$ is a homogeneous element over $A[\g_1, \ldots ,\g_i]$ for $1\leq i<s$. When $\nu $ is an Abhyankar valuation on $\k\lb \x\rb $ and $A=\wdh{V}_{\nu}$, $V^{\alg}_{\nu}$ or $V^{\fg}_{\nu}$, the valuation $\nu$ extends to $A[\g_1, \ldots ,\g_i]$ as in Remark \ref{extend}. Then we denote by $A[\langle\g_1, \ldots , \g_s\rangle]$ the valuation ring associated to the order valuation of $A[\g_1, \ldots ,\g_s]$. In this case the elements of $A[\langle\g_1, \ldots ,\g_s\rangle]$ are the elements which are finite sums of terms of the form $b\g_1^{j_1}...\g_s^{j_s}$ where 
$b\in \Frac(A)$ and $\nu(b)\geq -(j_1\nu(\g_1)+\cdots+j_s\nu(\g_s))$.
\end{definition}


\begin{definition}\label{limit}
If $\nu$ is an Abhyankar valuation we denote by 
$$\ovl{V}_{\nu}:=\underset{\underset{\g_1, \ldots , \g_s}{\lgw}}{\lim}\,\wdh{V}_{\nu}[\langle\g_1, \ldots , \g_s\rangle]$$ the  direct limit over all subsets $\{\g_1, \ldots , \g_s\}$ of homogeneous elements with respect to $\nu$  and by $\ovl{\K}_{\nu}$ its fraction field. By Remark \ref{integral'} we may restrict the limit over the subsets of integral homogeneous elements.\\
In the same way we define 
$$\ovl{V}^{\fg}_{\nu}:=\underset{\underset{\g_1, \ldots , \g_s}{\lgw}}{\lim}\,V^{\fg}_{\nu}[\langle\g_1, \ldots , \g_s\rangle],$$ 
$$\ovl{V}^{\alg}_{\nu}:=\underset{\underset{\g_1, \ldots , \g_s}{\lgw}}{\lim}\,V^{\alg}_{\nu}[\langle\g_1, \ldots , \g_s\rangle],$$
 the limits being taken over  all subsets $\{\g_1,....,\g_s\}$ of (integral) homogeneous elements with respect to $\nu$, and we denote by $\ovl{\K}^{\fg}_{\nu}$ and $\ovl{\K}^{\alg}_{\nu}$ their respective fraction fields.
\end{definition}

The following result provides an upper bound on the number of homogeneous elements we need to consider:

\begin{prop}\label{bound}
Let $\nu$ be an Abhyankar valuation on $\k\lb \x\rb $ and let $\G$ denote its value group. Set $N:=\dim_{\Q}\G\otimes_{\Z}\Q$ and let $\g_1$, \ldots , $\g_{s}$ be homogeneous elements with respect to $\nu$. Then there exist integral homogeneous elements $\g'_1$, \ldots , $\g'_N$ with respect to $\nu$ such that  $\wdh{V}_{\nu}[\langle \g_1, \ldots , \g_s\rangle]=\wdh{V}_{\nu}[\langle \g'_1, \ldots ,\g'_N\rangle]$.\\
This equality remains true if we replace $\wdh{V}_{\nu}$ by $V^{\alg}_{\nu}$ or $V^{\fg}_{\nu}$.

\end{prop}
\begin{proof} 
We will prove this proposition by induction on $s$. Let $\g_1$, \ldots , $\g_{N+1}$ be nonzero  homogeneous elements with respect to $\nu$. Let $d_i$ be the order of $\g_i$, for $1\leq i\leq N+1$. By assumption on $N$ the $d_i$  are $\Q$-linearly dependent. Thus, after a permutation of the $g_i$ ,  there exists an integer $1\leq l\leq N$ and integers $p_i\in\Z_{\geq 0}$, $q_i\in\N$ for all $1\leq i\leq N+1$, such that 
\begin{equation}\label{expo}\frac{p_1}{q_1}d_1+\cdots+\frac{p_l}{q_l}d_l=\frac{p_{l+1}}{q_{l+1}}d_{l+1}+\cdots+\frac{p_{N+1}}{q_{N+1}}d_{N+1}.\end{equation}
Set $r_i:=\frac{p_1\cdots p_{N+1}}{p_i}$ for $1\leq i\leq N+1$. Let us denote $\g'_i:=\g_i^{\frac{1}{q_ir_i}}$. Then we have $$\wdh{V}_{\nu}[\langle\g_1, \ldots , \g_{N+1}\rangle]\subset \wdh{V}_{\nu}[\langle\g'_1, \ldots ,\g'_{N+1}\rangle].$$
By (\ref{expo}) and Lemma \ref{homogenes1}, $\g_1'\cdots \g_l'$ and $\g_{l+1}'\cdots \g'_{N+1}$ are homogeneous elements of same order. By the Primitive Element Theorem there exists $c\in\k$ such that $$\k(\x)[\g_1'\cdots \g_l',\g_{l+1}'\cdots \g_{N+1}']=\k(\x)[\g_1'\cdots \g_l'+c\g_{l+1}'\cdots \g_{N+1}'].$$ 
Moreover $\g:=\g_1'\cdots \g_l'+c\g_{l+1}'\cdots \g_{N+1}'$ is a homogeneous element with respect to $\nu$ of  same order as $\g'_1\cdots \g'_l$ and $\g_{l+1}'\cdots \g_{N+1}'$ by Lemma \ref{homogenes1}.
Since 
$$\k(\x)[\g'_1, \ldots ,\g'_l]=\k(\x)[\g'_1, \ldots ,\g'_{l-1},\g'_1\cdots \g'_l]$$ and $$\k(\x)[\g_{l+1}', \ldots ,\g'_{N+1}]=\k(\x)[\g'_{l+1}, \ldots ,\g'_N,\g'_{l+1}\cdots \g'_{N+1}],$$ we have
$$\k(\x)[\g_1', \ldots ,\g_{N+1}']=\k(\x)[\g'_1, \ldots ,\g'_{l-1},\g'_{l+1}, \ldots ,\g'_{N},\g].$$
Thus $\g'_l$ is a finite sum of products of elements $a_i(\x)\in\k(\x)$ and powers of $\g_1'$, \ldots , $\g'_{l-1}$, $\g'_{l+1}$, \ldots , $\g'_{N}$, $\g$ and by homogeneity we may assume that $a_i(\x)$ are $(\nu)$-homogeneous. Thus $$\wdh{V}_{\nu}[\langle \g'_1, \ldots ,\g'_{N+1}\rangle]=\wdh{V}_{\nu}[\langle  \g'_1, \ldots ,\g'_{l-1},\g'_{l+1}, \ldots ,\g'_{N},\g\rangle].$$
By Remark \ref{integral'} we may assume that the $\g_i'$ are integral homogeneous elements.\\
The proof is the same if we replace $\wdh{V}_{\nu}$ by $V^{\alg}_{\nu}$ or $V^{\fg}_{\nu}$.

\end{proof}

\section{Newton method and algebraic closure of $\k\lb \x\rb $ with respect to an Abhyankar valuation}

\subsection{Newton method}

\begin{lemma}\label{lemma1}
Let $(A,\m)$ be a complete graded local ring. Let $B$ be the set of the elements of $A$ whose support is included in a finitely generated sub-semigroup of $\R_{\geq 0}$. Then $B$ is a Henselian  local domain.
\end{lemma}
\begin{proof}
Let us prove that $B$ is a ring: let $b_1$ and $b_2$ be two elements of $B$ whose supports are included in $\La_1$ and $\La_2$ respectively. Thus we can write $b_i=\sum_{j\in \La_i}b_{i,j}$ where $b_{i,j}$ is a homogeneous element of degree $j$ for any $i=1$, $2$ and $j\in\La_1$ or $\La_2$. Let $\La$ be the finitely generated sub-semigroup of $\R_{\geq 0}$ generated by $\La_1$ and $\La_2$. Then $\Supp(b_1+b_2)$ and $\Supp(b_1b_2)$ are included in $\La$. This proves that $B$ is a ring. Since $B\subset A$, $B$ is a domain.\\
It is clear that  $\m\cap B$ is an ideal of $B$. If $b\in B\backslash (\m\cap B)$, then there exists $a\in A$ such that $ab=1$. Let us write $b=\sum_{i\in \La}b_i$ where $b_i$ is homogeneous of degree $i$ and $\La$ is a finitely generated sub-semigroup of $\R_{\geq 0}$. Since $b\notin \m$, then $b_0\neq 0$. In this case we have 
$$a=b^{-1}=\frac{1}{b_0}\left(1+\sum_{i\in\La\backslash\{0\}}\frac{b_i}{b_0}\right)^{-1}=\frac{1}{b_0}\sum_{k=1}^{\infty}(-1)^k\left(\sum_{i\in\La\backslash\{0\}}\frac{b_i}{b_0}\right)^k.$$
Thus $\Supp(a)\subset \La$. This proves that $B$ is a local ring with maximal ideal $\m\cap B$.\\
Now let $P(Z)\in B[Z]$, such that $P(0)\in \m\cap B$ and $P'(0)\notin \m$. We denote by $\nu$ the order function of $A$, i.e. if $a\in A$, $a\neq 0$, $a=\sum_{i}a_i$ where $a_i$ is homogeneous of degree $i$, $\nu(a):=\inf\{i\ / \ a_i\neq 0\}$ and  the initial term of $a$ is $\ini(a):=a_{\nu(a)}$. Since $A$ is a complete   local ring it is a Henselian local ring and  there exists $a\in \m$ such that $P(a)=0$. We can construct $a$ by using the fact that
\begin{equation}\label{Taylor}P(Z)=P(0)+P'(0)Z+Q(Z)Z^2\end{equation}
where $Q(Z)\in B[Z]$. Indeed, let $\La$ denote a finitely generated sub-semigroup of $\R_{\geq 0}$ containing  the supports of all the coefficients of $P(Z)$. In this case $a_1:=\ini(a)=-\frac{\ini(P(0))}{\ini(P'(0))}$ is a homogeneous element of degree $d_1\in\La$, $d_1>0$. If we set $P_1(Z):=P(Z+a_1)$, we see that 
$$\nu(P_1(0))=\nu(P(a_1))>d_1,$$
$P_1'(0)=P'(0)=0$ 
and $a-a_1$ is the solution of $P_1(Z)=0$ given by the Hensel Lemma.  Then we replace $P$ by $P_1$ in Equation (\ref{Taylor}) and repeat the same argument, using the fact that the coefficients of $P_1(Z)$ have support included in $\La$. Thus we see that  $\ini(a-a_1)=-\frac{\ini(P_1(0))}{\ini(P'(0))}$ is a homogeneous element of degree $d_2\in\La$, $d_2>d_1$. We repeat this operation a countable number of times (since $\La$ is countable) in order to construct $a$ and we see that $\Supp(a)\subset \La$.

\end{proof}
 \noindent Now we can prove the following theorem:

\begin{theorem}\label{main}
Let $\k$ be a  field of characteristic zero and $\nu$ be an Abhyankar valuation of $\k\lb \x\rb $. Let $N:=\dim_{\Q}\G\otimes_{\Z}\Q$.
Let 
$$P(Z)\in V_{\nu}^{\fg}[\langle\g_1, \ldots , \g_N\rangle][Z]$$ (resp. $\wdh{V}_{\nu}[\langle\g_1, \ldots , \g_N\rangle][Z]$) be a monic polynomial of degree $d$ where $\g_i$ is a homogeneous element with respect to $\nu$ for $1\leq i\leq N$.  Then there exist integral homogeneous elements $\g'_{1}$, \ldots , $\g'_N$ such that the roots of $P(Z)$ are in $V_{\nu}^{\fg}[\langle\g'_1, \ldots ,\g'_N\rangle]$ (resp. $\wdh{V}_{\nu}[\langle \g_1', \ldots ,\g'_N\rangle]$). 
\end{theorem}

\begin{proof}
Let us prove the case $P(Z)\in V_{\nu}^{\fg}[\langle\g_1, \ldots , \g_N\rangle][Z]$.
We write $P(Z)=Z^d+a_1Z^{d-1}+\cdots+a_d$. By replacing $Z$ by $Z-\frac{1}{d}a_1$ we can assume that $a_1=0$. Let $i_0$ be an integer such that 
$$\frac{\nu(a_{i_0})}{i_0}\leq \frac{\nu(a_i)}{i},\ \ \text{ for every } 2\leq i\leq d.$$
Let $\g$ be a $i_0$th root of $\ini_{\nu}(a_{i_0})$, i.e. $\g$ is a homogeneous element such that $\g^{i_0}=\ini_{\nu}(a_{i_0})$.
By the definition of $i_0$, for  every $2\leq i\leq d$ we can write
$$a_i=\g^ia'_i$$
with $a'_i\in V^{\fg}_{\nu}[\langle\g_1, \ldots , \g_N,\g\rangle]$. Then we have 
$$P(\g Z)=\g^dZ^d+\g^{d-2}a_2Z^{d-2}+\cdots+a_d=\g^d\left(Z^d+a_2'Z^{d-2}+\cdots+a'_d\right)$$
Let $S(Z):=Z^d+a_2'Z^{d-2}+\cdots+a'_d$ and let $\ovl{S}(Z)$ be the image of $S(Z)$ in the residue field $\L=V^{\fg}_{\nu}[\langle\g_1, \ldots , \g_s,\g\rangle]/\m$ where $\k_{\nu}\lgw \L$ is finite and $\m$ is the maximal ideal of $V^{\fg}_{\nu}[\langle\g_1, \ldots , \g_s,\g\rangle]$. If $\ovl{S}(Z)=(Z+\ovl{a})^d$ where $\ovl{a}\in\L$, since $a_1=0$ and char$(\L)=0$, this would imply $\ovl{a}=0$. But $\ovl{S}(Z)\neq Z^d$ since its coefficient of $Z^{d-i_0}$ is nonzero . Thus we can factor $\ovl{S}(Z)=\ovl{S}_1(Z)\ovl{S}_2(Z)$ such that $\ovl{S}_1(Z)$ and $\ovl{S}_2(Z)$ are coprime monic polynomials in $\L[\g'][Z]$ where $\g'$ is algebraic over $\L$, i.e. $\g'$ is a homogeneous element of degree 0 with respect to $\nu$. Since  $V^{\fg}_{\nu}[\langle\g_1, \ldots , \g_N,\g,\g'\rangle]$  is a  Henselian local ring by Lemma \ref{lemma1}, by Hensel Lemma the polynomial $S(Z)$ factors as $S(Z)=S_1(Z)S_2(Z)$ where the images of $S_1(Z)$ and $S_2(Z)$ in $V^{\fg}_{\nu}[\langle\g_1, \ldots , \g_N,\g,\g'\rangle]$ are $\ovl{S}_1(Z)$ and $\ovl{S}_2(Z)$ and the $\nu$-support of the coefficients of $S_1(Z)$ and $S_2(Z)$ are contained in a finitely generated sub-semigroup of $\R_{\geq 0}$.\\
Since $\deg_Z(S_1(Z))$, $\deg_Z(S_2(Z))<d=\deg_Z(P(Z))$, the theorem is proven by induction on $d$ by using Proposition \ref{bound} and Remark \ref{integral'}.\\
\\
The case $P(Z)\in \wdh{V}_{\nu}[\langle\g_1, \ldots , \g_N\rangle][Z]$ is proven in a similar way by  using the fact that $\wdh{V}_{\nu}[\langle\g_1, \ldots , \g_N,\g,\g'\rangle]$ is a complete local ring, thus a Henselian local ring.

\end{proof}

\begin{rmk}
The proof of this theorem is what we call the \emph{Newton-Puiseux method}. Usually the term of Newton-Puiseux method is used when one compute the roots of a monic polynomial with coefficients in the ring of power series in one variable: one root is constructed by computing step by step its coefficients. The fact that the ring of formal power series is a complete local ring allows to conclude that this process converges. But when we want to find roots of a polynomial in a local ring that is not complete but only Henselian, it is more convenient to use the Hensel Lemma as we have done here. The proof we used here appeared for the first time in \cite{BM} (to the knowledge of the author). Of course if $\nu$ is a divisorial valuation $\wdh{V}_{\nu}$ is isomorphic to the ring of formal power series in one variable over the residue field $\k_{\nu}$ and the previous theorem may be proven by using the classical Newton-Puiseux method.
\end{rmk}

\begin{corollary}\label{alg_clo}
The field $\ovl{\K}^{\fg}_{\nu}$ (resp. $\ovl{\K}_{\nu}$) is algebraically closed and it is the algebraic closure of $\K^{\fg}_{\nu}$ (resp. $\wdh{\K}_{\nu}$).
\end{corollary}
\begin{proof}
Let $P(Z)\in\ovl{\K}^{\fg}_{\nu}[Z]$ be an irreducible polynomial. By multiplying $P(Z)$  by an element of $V^{\fg}_{\nu}$, we may assume that 
$$P(Z)\in V^{\fg}_{\nu}[\langle\g_1, \ldots , \g_N\rangle][Z]$$ for some homogeneous elements $\g_1$,  \ldots , 
$\g_N$ with respect to $\nu$. We write $P(Z)=a_dZ^d+\cdots+a_0$, $a_i\in V^{\fg}_{\nu}[\langle \g_1, \ldots ,\g_N\rangle]$, $0\leq i\leq d$.  We set $Q(Z):=a_d^{d-1}P(Z/a_d)$. Then $Q(Z)$ is a monic polynomial of $V^{\fg}_{\nu}[\langle\g_1, \ldots , \g_N\rangle][Z]$ and if $z$ is a root of $Q(Z)$, then $\frac{z}{a_d}$  is a root of $P(Z)$. Hence the result comes from Theorem \ref{main}.\\
The assertion concerning $\ovl{\K}_{\nu}$ is proven similarly.
\end{proof}

We have the similar result for $\ovl{\K}^{\alg}$: 
 
 \begin{lemma}
 The algebraic closure of $\K_n$ in $\ovl{\K}_{\nu}$ is equal to $\ovl{\K}^{\alg}_{\nu}$. In particular $\ovl{\K}^{\alg}_\nu$ is algebraically closed.
 \end{lemma}
 \begin{proof}
 Let $\g_1$,  \ldots , $\g_s$ be homogeneous elements with respect to $\nu$. Let us denote by $q_{i+1}$ the degree of the minimal polynomial of $\g_{i+1}$ over $\K_n[\g_1, \ldots , \g_i]$ for $0\leq i\leq s-1$.  Thus any element $z$ of $\wdh{\K}_{\nu}[\g_1, \ldots , \g_s]$ can be uniquely written as $z=\sum_{i\in I}A_{i_1, \ldots ,i_s}\g_1^{i_1}\cdots \g_s^{i_s}$   where $A_{i_1, \ldots ,i_s}\in\wdh{\K}_{\nu}$ for  all $i\in I$ and $I=\{0, \ldots ,q_1-1\}\times\cdots\times\{0, \ldots ,q_s-1\}$.\\
  In order to prove the lemma we need to show that $A_{i_1, \ldots ,i_s}\in\K_{\nu}^{\alg}$ for any $i_1$, \ldots , $i_s$ when $z$ is algebraic over $\k\lb \x\rb $.
 In this case let $\L:=\wdh{\K}_{\nu}[\g_1, \ldots , \g_{s-1}]$ and let us write $\displaystyle z:=\sum_{i=0}^{q_s-1}B_i\g_s^i$ where $B_i\in \L$ for all $i$.
  Let us set $\zeta_1:=\g_s$ and let $\zeta_2$, \ldots , $\zeta_{q_s}$ be the conjugates of $\zeta_1$ over $\ovl{\K}_{\nu}[\g_1, \ldots , \g_{s-1}]$. Let us define $\displaystyle z_j=\sum_{i=0}^{q_s-1}B_i\zeta_j^i$ for $1\leq j\leq q_s$. Then we have 
 $$\left(\begin{array}{c}z_1 \\ z_2\\ \vdots \\ z_{q_s}\end{array}\right)=\left(\begin{array}{cccc} 1 & \zeta_1 & \cdots & \zeta_1^{q_s-1} \\
 1 & \zeta_2 & \cdots & \zeta_2^{q_s-1}\\
 \vdots &\vdots & \vdots & \vdots\\
 1 & \zeta_{q_s} & \cdots & \zeta_{q_s}^{q_s-1}\end{array}\right)\left(\begin{array}{c}B_0 \\ B_1\\ \vdots\\ B_{q_s-1}\end{array}\right).$$
 The matrix $\left(\begin{array}{cccc} 1 & \zeta_1 & \cdots & \zeta_1^{q_s-1} \\
 1 & \zeta_2 & \cdots & \zeta_2^{q_s-1}\\
 \vdots &\vdots & \vdots & \vdots\\
 1 & \zeta_{q_s} & \cdots & \zeta_{q_s}^{q_s-1}\end{array}\right)$ is invertible and its entries are algebraic over $\k(\x)$,  $z_j$ is algebraic over $\k\lb \x\rb $ for all $j$, hence  $B_j$ is algebraic over  $\k\lb \x\rb $ for all $j$. By induction on $s$ we see that $A_{i_1, \ldots ,i_s}\in\K_{\nu}^{\alg}$ for any $i_1$, \ldots , $i_s$. 
 \end{proof}
  
  We can summarize the situation with the following commutative diagram where  the bottom part corresponds to the quotient fields of the rings of the upper part and all the morphisms are injective:
$$\xymatrix{
     \k\lb \x\rb\subset V_{\nu} \ar[rr]  \ar[dd] && V^{\alg}_{\nu} \ar[dd] |!{[rd];[ld]}\hole \ar[rd] \ar[rr] &&  V_{\nu}^{\fg} \ar[rr] \ar[rd]  \ar[dd] |!{[rd];[ld]}\hole& & \wdh{V}_{\nu} \ar[rd]  \ar[dd] |!{[rd];[ld]}\hole\\
    &  && \ovl{V}^{\alg}_{\nu} \ar[dd] \ar[rr] & &  \ovl{V}^{\fg}_{\nu} \ar[dd] \ar[rr]&& \ovl{V}_{\nu} \ar[dd] \\
    \K_n \ar[rr] && \K^{\alg}_{\nu} \ar[rd]  \ar[rr] |!{[ur];[dr]}\hole&&\K_{\nu}^{\fg}\ar[rd]  \ar[rr] |!{[ur];[dr]}\hole& & \wdh{\K}_{\nu} \ar[rd]\\
    & & & \ovl{\K}^{\alg}_{\nu} \ar[rr] && \ovl{\K}^{\fg}_{\nu} \ar[rr] && \ovl{\K}_{\nu}}$$

   \begin{ex}
 Let $g(T)=\sum_{i=1}^{\infty}c_iT^i\in\Q\lb T\rb  $ be a formal power series which is not algebraic over $\Q[T]$. Let $\a:=(\a_1,\a_2)\in\N^n$. Let us set
 $$f:=g\left(\frac{x_2^{\a_1}}{x_1^{\a_2}}\right)=\sum_{i=1}^{\infty}c_i\frac{x_2^{\a_1i}}{x_1^{\a_2i}}\in\k\(x_1\)\(x_2\).$$
 But $f\notin\ovl{\K}_{\nu_{\a}}$: let $P(Z)=a_0(\x)Z^d+\cdots+a_d(\x)\in\wdh{V}_{\nu_{\a}}[Z]$ be a polynomial such that $P(f)=0$. Let us write $a_i(\x)=\sum_{k=0}^{\infty}a_{i,k}(\x)$ where $a_{i,k}(\x)$ is a $(\a_1,\a_2)$-homogeneous rational fraction of degree $k$. By homogeneity we have 
 $$a_{0,k}f^d+a_{1,k}f^{d-1}+\cdots+a_{d,k}=0\ \ \ \forall k\in\N.$$
 This implies that 
 $$a_{0,k}(1,T)g(T^{\a_1})^d+a_{1,k}(1,T)g(T^{\a_1})^{d-1}+\cdots+a_{d,k}(1,T)=0\ \ \ \forall k\in\N.$$
 Thus $a_{i,k}(\x)=0$ for all $0\leq i\leq d$ and $0\leq k$. Hence $P(Z)=0$ and $f\notin \ovl{\K}_{\nu_{\a}}$.\\
 \\
 \\
 On the other hand, $\displaystyle h:=g\left(\frac{x_1^{2\a_2}}{x_2^{\a_1}}\right)=\sum_{i=1}^{\infty}c_i\frac{x_1^{2\a_2i}}{x_2^{\a_1i}}\in\wdh{\K}_{\nu_{\a}}$ but $h$ is not algebraic over $\k\(x_1\)\(x_2\)$.

  \end{ex}


\subsection{Analytically irreducible polynomials}
 
\begin{prop}\label{V_irredu}
Let $P(Z)\in V^{\fg}_{\nu}[Z]$ (resp. $V^{\alg}_{\nu}[Z]$) be an irreducible monic polynomial. Then $P(Z)$ is irreducible in $\wdh{V}_{\nu}[Z]$.
\end{prop}

\begin{proof}
By Corollary \ref{alg_clo}, $P(Z)$ splits in $V_{\nu}^{\fg}[\langle\g_1, \ldots , \g_s\rangle]$ for some homogeneous elements $\g_1$, \ldots , $\g_s$ with respect to $\nu$. Since 
$$V_{\nu}^{\fg}[\langle\g_1, \ldots , \g_s\rangle]\cap \wdh{V}_{\nu}=V_{\nu}^{\fg}$$ the result follows.\\
The proof is the same for $V_{\nu}^{\alg}$.
\end{proof}

\begin{lemma}\label{val_ini}
Let $\s$ be a $\wdh{\K}_{\nu}$-automorphism of $\ovl{\K}_{\nu}$. For any $z\in\ovl{\K}_{\nu}$ we have $\nu(\s(z))=\nu(z)$.
\end{lemma}

\begin{proof}
Let $z\in\wdh{\K}_{\nu}[\g_1, \ldots , \g_s]$ where $\g_1$, \ldots , $\g_s$ are homogeneous elements with respect to $\nu$. Let us write $z:=\sum_{i\in\La}z_i$ where $z_i$ is homogeneous of degree $i$ for every $i$ and $\La$ is  a countable subset of $\R$ with no accumulation point (see Remark \ref{supp_compl}). If $i_0=\nu(z)$, then $z_{i_0}\neq 0$ and $\nu(z_i)=0$ for all $i<i_0$. Since $\s$ acts only on the homogeneous elements $\g_1$, \ldots , $\g_s$, we have $\s(z)=\sum_i\s(z_i)$. For all $i$, $\s(z_i)$ is homogeneous of degree $i$ and $\s(z_i)=0$ if and only if $z_i=0$. This proves that $i_0=\nu(\s(z))$.
\end{proof}

\begin{definition}
Let $P(Z)\in A[Z]$ where $A$ is an integral domain. We write
$$P(Z)=a_0Z^d+a_1Z^{d-1}+\cdots+a_d.$$
 Let $\nu : A\lgw \R_{\geq 0}$ be a valuation. The \emph{Newton polygon} of $P$ is the convex hull of the set
 $$\left\{(\nu(a_i),d-i)\in \R_{\geq 0}^2\ /\  i=0, \ldots ,d\right\}+\R_{\geq 0}^2.$$

\end{definition}

 \begin{corollary}\label{Newton}
Let $P(Z)\in \wdh{V}_{\nu}[Z]$ be an irreducible monic  polynomial. Then the Newton polygon of $P(Z)$ has only one edge. The result remains valid if we replace $\wdh{V}_{\nu}$ by $V^{\alg}_{\nu}$ or $V^{\fg}_{\nu}$.
 \end{corollary}
 
 \begin{proof}
Let  $z$ be a root of $P(Z)$ in $\ovl{V}_{\nu}$. Let $\s$ be a $\wdh{\K}_{\nu}$-automorphism of $\ovl{\K}_{\nu}$. Then $\nu(\s(z))=\nu(z)$ by Lemma \ref{val_ini}. The finite product of the distinct linear forms $Z-\s(z)$ obtained in this way is a monic polynomial with coefficients in $\wdh{\K}_{\nu}$  and divides $P(Z)$. Since $P(Z)$ is irreducible, both polynomials are equal. This proves that all the roots of $P(Z)$ have same valuation, hence the Newton polygon of $P(Z)$ has only one edge.\\
The cases $V^{\alg}_{\nu}$ and $V^{\fg}_{\nu}$ are deduced from Lemma \ref{V_irredu}.
   \end{proof}

 \begin{ex}
Let $P(Z):=Z^3+3x_1x_2Z-2x_1^4\in\k\lb x_1,x_2\rb  [Z]$.  We see that $P(Z)$ has one root of order 2 and two roots of order 1 in $\ovl{V}_{\ord}^{\fg}$. By Corollary \ref{Newton}, $P(Z)$ has at least one root in  $V_{\ord}^{\alg}$ of order 2.\\
 Let $\sqrt{1+U}:=1+\sum_{i\geq 1}a_iU^i$, $a_i\in\Q$ for all $i$, the formal powers series whose square is equal to $1+U$, and let $\sqrt[3]{1+U}:=1+\sum_{i\geq 1}b_iU^i$ , $b_i\in\Q$ for all  $i$,  the formal power series whose cube is equal to $1+U$. Then the roots of $P(Z)$ are
$$ a\sqrt[3]{q+\sqrt{q^2+p^3}}+b\sqrt[3]{q-\sqrt{q^2+p^3}}$$
with $(a,b)=(1,1)$, $(j,j^2)$ or $(j^2,j)$ and $p=x_1x_2$ and $q=x_1^4$.
But
$$\sqrt[3]{q+\e\sqrt{q^2+p^3}}=\sqrt[3]{x_1^4+\e\sqrt{x_1^3x_2^3+x_1^8}}=\sqrt[3]{\e}\sqrt{x_1x_2}+\eta$$
where $\e=1$ or $-1$ and $\ord(\eta)>1$. Both order $1$ roots of  $P(Z)$ have initial term of the form $\a\sqrt{x_1x_2}$ where $\a\in\C^*$. Thus $P(Z)$ has only one root in $V^{\alg}_{\ord}$ and even in $\K^{\fg}_{\ord}$.\\
Let $z$ be the only root of $P(Z)$ in $V^{\alg}_{\ord}$. If $z\in\K_n$, since $P(Z)$ is monic and $\k\lb \x\rb $ is an integral domain, then $z\in\k\lb \x\rb $. But $\ini(z)=\frac{2}{3}\frac{x_1^3}{x_2}\notin\k\lb \x\rb $. Thus $z\notin\K_n$, hence $P(Z)$   is irreducible in ${\K}_{n}[Z]$. This shows that  $\K_n\lgw \K^{\alg}_{\ord}$ is not a normal extension in general.

  \end{ex}
  
   \begin{corollary}
 Let $P(Z):=Z^d+a_1(\x)Z^{d-1}+\cdots+a_d(\x)\in\k\lb \x\rb [Z]$  be an irreducible polynomial having its roots in  $\k\lb x_1^{\frac{1}{e}}, \ldots ,x_n^{\frac{1}{e}}\rb$ for some  positive integer  $e$. Then the Newton polyhedron of $P(Z)$ is the convex hull of the cone of $\N^{n+1}$  centered in $(0, \ldots ,0,d)$ and generated by the convex hull of the Newton polyhedra of $a_d(\x)$ in $\N^{n}$.
   \end{corollary}

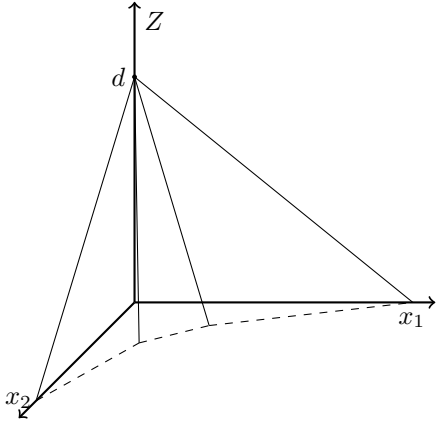
  $$ \begin{tikzpicture}[scale=1]
  
  \draw[thick,->] (0,0,0) -- (4,0,0) node[anchor=north east]{$x_1$};
\draw[thick,->] (0,0,0) -- (0,4,0) node[anchor=north west]{$Z$};
\draw[thick,->] (0,0,0) -- (0,0,4) node[anchor=south]{$x_2$};

\draw [dashed] (0,0,3.4) coordinate (a1)  -- (0.6,0,1.4) coordinate (a2);
\draw  [dashed](a2) -- (1.3,0,0.8) coordinate (a3);
\draw  [dashed] (a3) -- (3.7,0,0) coordinate (a4);
  
     \filldraw [black] (0,3,0)  coordinate (c) circle (.7pt) node[left] {$d$};
     
    \draw  (c) -- (a1);
   \draw  (c) -- (a2);
      \draw  (c) -- (a3);
     \draw (c) -- (a4);    
        
      \end{tikzpicture}  $$
   
   \begin{proof}
  Let $\a\in\N^n$.
Let $z_1$,  \ldots , $z_d\in\k\lb x_1^{\frac{1}{e}}, \ldots ,x_n^{\frac{1}{e}}\rb$ be the roots of $P(Z)$. Then $z_i\in\wdh{V}_{\nu_{\a}}[x_1^{\frac{1}{e}}, \ldots ,x_n^{\frac{1}{e}}]$ for any $i$, the $x_i^{\frac{1}{e}}$  being homogeneous elements with respect to 
  $\nu_{\a}$. Let $G\simeq \left(\Z/e\Z\right)^n$ be the Galois group of the extension $\wdh{V}_{\nu_{\a}}\lgw \wdh{V}_{\nu_{\a}}[x_1^{\frac{1}{e}}, \ldots ,x_n^{\frac{1}{e}}]$. The $z_i$  are conjugated under the action  of $G$, thus $P(Z):=\prod_{i=1}^d(Z-z_i)$ is irreducible in $\wdh{V}_{\nu_{\a}}[Z]$. This being true for any $\a\in\N^n$, the result follows from Corollary \ref{Newton}.
    
  \end{proof}

We finish this section by giving two results relating the roots of a polynomial $P(Z)$ to the roots of polynomials approximating $P(Z)$. First of all we give the following definition:

\begin{definition}
Let $P(Z)\in A[Z]$  where $A$ is an integral domain and let $\nu$ be a valuation on $A$. We define
$$\nu(P(Z)):=\min_a\nu(a) $$
where $a$ runs over all the 
 coefficients of $P(Z)$.
\end{definition}

The following proposition is the analogue of Proposition 2.6 of \cite{To}:
 \begin{prop}\label{cor_factor_limit}
 Let $P(Z)\in V^{\fg}_{\nu}[Z]$ be a monic polynomial  with no multiple factor. Let us write  $P(Z)=P_1(Z)...P_r(Z)$ where $P_i(Z)\in V^{\fg}_{\nu}[Z]$, $1\leq i\leq r$, are irreducible monic polynomials. Let $Q(Z)\in  V^{\fg}_{\nu}[Z]$ be a monic polynomial  and let $z_1$, \ldots , $z_d$ be the roots of $P(Z)$. If
 $$\deg(Q(Z))=\deg(P(Z))$$
 and
  $$ \nu(Q(Z)-P(Z))>d \max_{i\neq j}\{\nu(z_i-z_j)\}$$
then we may factor $Q(Z)=Q_{1}(Z)...Q_{r}(Z)$ such that  $Q_{i}(Z)\in V^{\fg}_{\nu}[Z]$ is an irreducible monic polynomial, $1\leq i\leq r$, and 
$$\displaystyle\nu(Q_i(Z)-P_i(Z))\geq \frac{\nu(Q(Z)-P(Z))}{d}.$$
 The result is still valid if we replace $V_{\nu}^{\fg}$ by $V_{\nu}^{\alg}$ or $\wdh{V}_{\nu}$.
 \end{prop}
 
 \begin{proof}
 
 Since $P(Z)$ has no multiple factor and since char$(\k)=0$, we have $z_i\neq z_j$ for all $i\neq j$. Let us set $r:=\max_{i\neq j}\{\nu(z_i-z_j)\}$. Let $z'_{i}$, $1\leq i\leq d$, be the roots of $Q(Z)$. Let $z$ be a root of $P(Z)$ in $V^{\fg}_{\nu}[\langle\g_1, \ldots , \g_N\rangle]$. Let us write
 $P(Z)=Z^d+a_1Z^{d-1}+\cdots+a_d$ and $Q(Z)=Z^d+b_{1}Z^{d-1}+\cdots+b_{d}$. Then 
$$\prod_{1\leq i\leq d}(z-z_{i}')=Q(z)=Q(z)-P(z)=\sum_{i=1}^d(b_{i}-a_i)z^{d-i}.$$
Thus there exists at least one $i$ such that 
$$\nu(z_{i}'-z)\geq \frac{ \min_{1\leq i\leq d}\{\nu(a_i-b_{i})\}}{d}=\frac{\nu(Q(Z)-P(Z))}{d}>r.$$ Let $t$ be another root of $P(Z)$. Then
$$\nu(z_{i}'-t)=\nu(z_{i}'-z+z-t)=\nu(z-t)\leq r$$
since $\nu(z_{i}'-z)\geq \frac{ \min_{1\leq i\leq d}\{\nu(a_i-b_{i})\}}{d}>r\geq \nu(z-t)$. Thus for any root of $P(Z)$ denoted by $z$, there is only one $i$  such that 
$$\nu(z-z_{i}')\geq \frac{ \min_{1\leq i\leq d}\{\nu(a_i-b_{i})\}}{d}.$$
Let $\s_1(z)$, \ldots , $\s_e(z)$ be the conjugates of $z$ over $\K^{\fg}_{\nu}$.  Set 
$$R(Z):=(Z-z)\prod_{j=1}^e(Z-\s_j(z))\in V_{\nu}^{\fg}[Z].$$
 Then $R(Z)$ is an irreducible factor of $P(Z)$.  Moreover $\s_1(z_{i}')$, \ldots , $\s_e(z_{i}')$ are conjugates of $z_{i}'$ over $\K^{\fg}_{\nu}$. Let $\s$ be a $\K^{\fg}_{\nu}$-automorphism of $\ovl{\K}^{\fg}_{\nu}$ . Then $\s(z)$ is a conjugate of $z$ thus there exists $j$ such that $\s(z)=\s_j(z)$. Moreover $\s(z)$ is a root of $P(Z)$ and $\nu(\s(z_{i}')-\s(z))\geq\frac{ \min_{1\leq i\leq d}\{\nu(a_i-b_{i})\}}{d}$ by Lemma \ref{val_ini}. Thus we have 
  \begin{equation*}\begin{split}\nu(\s(z_{i}')-\s_j(z))=\nu(\s(z_{i}')-\s(z))=\nu(z_i'-z)&=\\
\nu(\s_j(z_i')-\s_j(z))&\geq \frac{ \min_{1\leq i\leq d}\{\nu(a_i-b_{i})\}}{d}  \end{split}\end{equation*}
and since there is only one root of $Q(Z)$ whose difference with $\s_j(z)$ has valuation  greater than $\frac{ \min_{1\leq i\leq d}\{\nu(a_i-b_{i})\}}{d}$, we necessarily have $\s(z_{i}')=\s_j(z_{i}')$. Thus $\s_1(z_{i}')$, \ldots , $\s_e(z_{i}')$ are the conjugates of $z_{i}'$ over $\K^{\fg}_{\nu}$. Thus the polynomial
$$S(Z):=(Z-z_{i}')\prod_{j=1}^e(Z-\s_j(z_{i}'))$$
is irreducible in $V^{\fg}_{\nu}[Z]$ and 
$$\nu(S(Z)-R(Z))\geq \frac{ \min_{1\leq i\leq d}\{\nu(a_i-b_{i})\}}{d}.$$
\\
The proof for $\wdh{V}_{\nu}$ is the same and the case $V^{\alg}_{\nu}$ is proven with the help of Lemma \ref{V_irredu}.

 \end{proof}
 
 \begin{rmk}\label{maj}
 Let us remark that $\nu(Q(Z)-P(Z))> \frac{d}{2}\nu(\D_P)$, where $\D_P$ is the discriminant of $P(Z)$, implies that 
 $$\nu(Q(Z)-P(Z))>d \max_{i\neq j}\{\nu(z_i-z_j)\}.$$
 \end{rmk}
 
 \begin{rmk}
 This result is not true if $P(Z)$ has multiple factors. For example, let $\nu$ be a divisorial valuation and let us consider $P(Z)=Z^2$ and let $Q(Z)=X^2+a$ where $\nu(a)=2k+1$ and $k\in\N$. Since $\nu(a)$ is odd and since the value group of $\nu$ is $\Z$, then it is not a square in $\wdh{V}_{\nu}$ and $Q(Z)$ is irreducible  but $P(Z)$ is not irreducible.
 \end{rmk}
 
 \begin{prop}\label{approx2}
Let $\nu$ be an Abhyankar valuation and let 
$$N:=\dim_{\Q}\G\otimes_{\Z}\Q.$$
 Let $P(Z)\in \wdh{V}_{\nu}[\langle\g_1, \ldots , \g_N\rangle][Z]$  be a monic polynomial  where $\g_1$, \ldots , $\g_N$ are  homogeneous elements with respect to $\nu$. Then there exist  integral homogeneous elements with respect to $\nu$, denoted by $\g_1'$,  \ldots , $\g_N'$, and $c\in\R_{>0}$ such that the roots of $P(Z)$ are in $\wdh{V}_{\nu}[\langle\g'_1, \ldots ,\g'_N\rangle]$ and for any monic polynomial $Q(Z)\in \wdh{V}_{\nu}[\langle\g_1, \ldots , \g_N\rangle][Z]$  such that $\deg(Q(Z))=\deg(P(Z))$ and $\nu(P(Z)-Q(Z))\geq c$, the roots of $Q(Z)$ are in $\wdh{V}_{\nu}[\langle\g'_1, \ldots ,\g'_N\rangle]$.
 \end{prop}
 
 \begin{proof}
 The proof of this proposition is based on the proof of Theorem \ref{main}. So let us use the notations of that proof. Let us write  $Q(Z)=Z^d+b_1Z^{d-1}+\cdots+b_d$ and let us define $R(Z):=Z^d+b'_1Z^{d-1}+\cdots+b'_d$ where $b_i':=\frac{b_i}{\g^i}$ for $1\leq i\leq d$. We have $Q(\g Z)=\g^dR(Z)$. Let us assume that $\nu(b'_i-a'_i)>0$ for all $1\leq i\leq d$ (i.e. if $\nu(b_i-a_i)>\nu(\g^i)$ for all $i$, thus we assume here that $c>d\nu(\g)$). Then $\ovl{R}(Z)=\ovl{S}(Z)$ ($\ovl{R}(Z)$ denotes the image of $R(Z)$ in $\L[Z]$)  and the factorization $\ovl{R}(Z)=\ovl{S}_1(Z)\ovl{S}_2(Z)$ lifts to a factorization $R(Z)=R_1(Z)R_2(Z)$ of $R(Z)$ as the product of two monic polynomials as in the proof of Theorem \ref{main}.
 \begin{lemma}\label{lemma_MB}
 In the previous situation there exist two constants $a>0$, $b\geq 0$ depending only on $S_1(Z)$ and $S_2(Z)$ such that for any $c>\max\{b,\nu(\g^d)\}$, we have $\nu(R_i(Z)-S_i(Z))>\frac{c-b}{a}$ for $i=1,2$.
 \end{lemma}
 
 \begin{proof}[Proof of Lemma \ref{lemma_MB}]
Let us denote by $r_{i,k}$ the coefficient of $Z^k$ of the polynomial $R_i(Z)$, for $i=1,2$ and $0\leq k\leq \deg_Z(R_i(Z))$, and let us denote by $r$ the vector whose coordinates are the $r_{i,k}$. The coefficient of $Z^k$ of $R_1(Z)R_2(Z)-S_1(Z)S_2(Z)$, for $0\leq k\leq d$, is a polynomial $f_k(r)$ whose coefficients are in $ \wdh{V}_{\nu}[\langle\g_1, \ldots , \g_N,\g,\g'\rangle]$ and depend themselves on the coefficients of $S(Z)$. By Theorem 1.2 \cite{M-B}, there exist $a>0$, $b\geq 0$ such that 
$$\forall c>b, \ \forall r\in  \wdh{V}_{\nu}[\langle\g_1, \ldots , \g_N,\g,\g'\rangle]^{d+2} \ \text{ such that } \ \nu(f_k(r))\geq c\ \ \forall k$$
$$\exists r'\in  \wdh{V}_{\nu}[\langle\g_1, \ldots , \g_N,\g,\g'\rangle]^{d+2}\ \text{ such that } \ f_k(r')=0 \  \ \forall k$$
$$\text{ and }\nu(r'_{i,j}-r_{i,j})\geq \frac{c-b}{a}\ \ \forall i,j.$$
Let us denote by $R'_i(Z)$ the polynomial whose coefficients are the $r'_{i,j}$  where $0\leq j\leq \deg(R_i)$. Then $R_1'(Z)R_2'(Z)=S_1(Z)S_2(Z)$. Moreover $\ovl{R}'_i(Z)=\ovl{R}_i(Z)=\ovl{S}_i(Z)$ if $\frac{c-b}{a}>0$. Since the roots of $\ovl{S}_1(Z)$ and $\ovl{S}_2(Z)$ are different, and since $\wdh{V}_{\nu}[\langle\g_1, \ldots , \g_N,\g,\g'\rangle][Z]$ is a GCD domain, then $R'_i(Z)=S_i(Z)$ for $i=1,2$. This proves the lemma.
 \end{proof}

 Here we remark that, the constants $a$, $b$, $\nu(\g)$ depend only on $P(Z)$. Thus the result is proven by induction on the degree of $P(Z)$ (since $\deg(S_i(Z))<\deg(P(Z))$ for $i=1,2$) and using Proposition \ref{bound} and Remark \ref{integral'}.

 \end{proof}




\section{Monomial valuation case: Eisenstein Theorem}\label{growth}
We will first construct a subring of $V^{\fg}_{\nu}$ containing $V^{\alg}_{\nu}$ when $\nu$ is a monomial valuation.
\begin{definition}\label{growth'}
Let $\a\in\R_{>0}^n$ and let $\d$ be a $(\a)$-homogeneous polynomial of degree $d$. We define
\begin{equation*}\begin{split}\V_{\a,\d}:=\left\{A\in \wdh{V}_{\nu_{\a}} \ / \ \exists \La  \right.\text{ a finitely generated}&\text{ sub-semigroup of } \R_{\geq 0}, \\
\forall i\in \La \ \exists  a_i& \in \k[\x] \text{ $(\a)$-homogeneous}, \\
 \exists a\geq 0,b\in\R\ \forall i\in \La \ \exists m(i)\in\N &\text{ s.t. }   m(i)\leq ai+b,\ \\
 \nu_{\a}\left(\frac{a_i}{\d^{m(i)}}\right)&\left.=i \text{ and } A=\sum_{i\in\La}\frac{a_i}{\d^{m(i)}} \right\}.\end{split}
\end{equation*}
\end{definition}

\noindent  With this notation we say that $i\lgm ai+b$ is a \textit{bounding function} for $\displaystyle \sum_{i\in\La}\frac{a_i}{\d^{m(i)}}$. \\
By Lemma \ref{Gordan} we have $\k\lb  \x\rb  \subset\V_{\a,\d}\subset V^{\fg}_{\nu_{\a}}$, by identifying a formal power series $\displaystyle\sum_{\b\in\Z_{\geq 0}^n}c_{\b}x^{\b}$ to $\displaystyle\sum_{i\in\La}\frac{a_i(x)}{\d(x)^{m(i)}}$ with $\displaystyle a_i(x):=\sum_{ \a_1\b_1+\cdots+\a_n\b_n=i}c_{\b}x^{\b}$ et $m(i)=0$ for all $i\in\La$.
We extend in an obvious way the addition and multiplication of $\k\lb x\rb  $ to $\V_{\a,\d}$: this defines a  $\k$-algebra structure over $\V_{\a,\d}$. We have easily the following lemma:

\begin{lemma}\label{bounding_fct} If $i\lgm ai+b$ is a bounding function of  $A$ and $B\in\V_{\a,\d}$ then it is also a bounding function of $A+B$ and the function $i\lgm ai+2b$ is a bounding function of $AB$.
\end{lemma}

\begin{proof}
Let us write
$$A=\sum_{i\in\La}\frac{a_i}{\d^{ai+b}},\ Ê\ \ B=\sum_{i\in\La}\frac{b_i}{\d^{ai+b}}$$
where $\La$ is a semigroup and the $a_i$ and $b_i$ are $(\a)$-homogeneous polynomials and
$$\nu_\a\left(\frac{a_i}{\d^{ai+b}}\right)=\nu_\a\left(\frac{b_i}{\d^{ai+b}}\right)=i\ \ \ \forall i\in\La.$$
Then we have
$$A+B=\sum_{i\in\La}\frac{a_i+b_i}{\d^{ai+b}}$$
$$\text{ and }AB=\sum_{i\in\La} \sum_{j\in\La, j\leq i}\frac{a_jb_{i-j}}{\d^{aj+b}\d^{a(i-j)+b}}=\sum_{i\in\La} \sum_{j\in\La, j\leq i}\frac{a_jb_{i-j}}{\d^{ai+2b}}.$$
This proves the lemma.
\end{proof}

\begin{rmk}\label{bounding}
If $A\in \V_{\a,\d}$ satisfies $\nu_{\a}(A)>0$ then $A$ admits a bounding function which is linear. Indeed let $i\lgm ai+b$ be a bounding function of $A$ and let $i_0:=\nu_{\a}(A)$. Then $i\lgm \left(a+\frac{b}{i_0}\right)i$ is a bounding function of $A$.
\end{rmk}

\begin{definition}
Let $\displaystyle A:=\sum_{i\in\La}\frac{a_i}{\d^{m(i)}}\in\V_{\a,\d}$, $A\neq 0$. Let $i_0$ be  the least element of $\La$ such that $a_{i_0}\neq 0$. We say that $\frac{a_{i_0}}{\d^{m(i_0)}}$ is the \emph{initial term} of $A$ with respect to $\nu_{\a}$ or its \emph{$(\a)$-initial term}. We denote it by $\ini_{\a}(A)$.
\end{definition}

\begin{lemma}\label{VV}
Let $\d$ and $\d'$ be two $(\a)$-homogeneous polynomials. We have the following properties:
\begin{enumerate}
\item[i)] The  (($\a$)-homogeneous) irreducible divisors of $\d$ divide $\d'$  if and only if $\V_{\a,\d'}\subset \V_{\a,\d}$. We denote by $\V_{\a}$ the inductive limit of the $\V_{\a,\d}$.
\item[ii)] The valuation $\nu_{\a}$  is well defined on  $\V_{\a,\d}$ and extends to $\V_{\a}$. Its valuation ring is exactly $\V_{\a}$.
\end{enumerate}
\end{lemma}

\begin{proof}
It is clear that if the irreducible divisors of $\d$ divide $\d'$ then $\V_{\a,\d}\subset\V_{\a,\d'}$. On the other hand if $\V_{\a,\d}\subset\V_{\a,\d'}$, then $\frac{1}{\d}\in\V_{\a,\d'}$, thus there exist a $(\a)$-homogeneous polynomial  $a\in\k[x]$ and an integer $m\in\N$ such that $\frac{1}{\d}=\frac{a}{{\d'}^m}$, hence $a \d={\d'}^m$. This proves i).\\
If $A\in\V_{\a,\d}$ and $B\in\V_{\a,\d'}$ satisfy $\nu_{\a}(B)\geq \nu_{\a}(A)$, let $\frac{a_k(x)}{\d(x)^{m(k)}}$ denote the first nonzero  term in the expansion of  $A$. Then we can check easily that $\frac{B}{A}\in\V_{a,\d\d' a_k}$. This proves ii).\end{proof}

\begin{definition}\label{KK}
For any $\a\in\R_{>0}^n$ we denote by $\KK_{\a}$ the fraction field of $\V_{\a}$ and $$\ovl{\KK}_{\a}:=\underset{\underset{\g_1, \ldots , \g_s}{\lgw}}{\lim}\,\KK_{\a}[\g_1, \ldots , \g_s]$$ the limit
 being taken over  all subsets $\{\g_1,....,\g_s\}$ of (integral) homogeneous elements with respect to $\nu$.\\
  If $\g_1$, \ldots , $\g_s$ are homogeneous elements with respect to $\nu_{\a}$ we denote by $\V_{\a,\d}[\langle\g_1, \ldots , \g_s\rangle]$ the ring of elements 
 $\sum_{\udl{k}}A_{\udl{k}}\g^{\udl{k}}$ where the sum is finite, $\udl{k}:=(k_1, \ldots ,k_s)$,
 $A_{\udl{k}}=\sum_{i\in\La}\frac{a_{i}}{\d^{m(i)}}$
 where $a_i\in\k[x]$ is $(\a)$-homogeneous, there exist two constants $a\geq 0$, $b\in\R$ such that $m(i)\leq ai+b$ for all $i$  and there exists $i_0\in\La$ such that $\nu_{\a}\left(\frac{a_i}{\d^{m(i)}}\right)=i-i_0$ and $\nu(\g^{\udl{k}})\geq i_0$. \\
This means that $\V_{\a,\d}[\langle\g_1, \ldots , \g_s\rangle]$ is the subring of $\KK_{\a}[\g_1, \ldots , \g_s]$ whose elements have non negative valuation $\nu_{\a}$. In the same way we denote by $\V_{\a}[\langle\g_1, \ldots , \g_s\rangle]$ the ring of elements of $\KK_{\a}[\g_1, \ldots , \g_s]$ having a non negative valuation $\nu_{\a}$. The field of fractions of  $\V_{\a}[\langle\g_1, \ldots , \g_s\rangle]$ is exactly $\KK_{\a}[\g_1, \ldots , \g_s]$.
\end{definition}

\begin{rmk}
We will see later (see Remark \ref{rmkII}) that these fields $\ovl{\KK}_{\a}$ coincide with those introduced in \cite{A-I} when $\dim_{\Q}(\a_1\Q+\cdots+\a_n\Q)=n$ where it is proven that they are algebraically closed.
\end{rmk}

\begin{rmk}
For any $\a\in\R_{>0}^n$ it is clear that $\V_{\a}\subset V_{\nu_{\a}}^{\fg}$ but both rings are never equal  if $\dim_{\Q}(\a_1\Q_1+\cdots+\a_n\Q)<n$. For instance, let $n=2$ and $\a=(1,1)$ and set 
$$z=\sum_{i\in\N} \frac{x_1^{(i+1)^2}}{x_2^{i^2}}\text{ or } \sum_{i\in\N} \frac{x_1^i}{x_1+ix_2}.$$
Then  obviously $z\in V^{\fg}_{\nu_{\a}}$ but $z\notin \V_{\a}$. 
\end{rmk}

\begin{prop}\label{prop_MD}
If  $\a_1$,  \ldots , $\a_n$ are linearly independent over $\Q$ then $\V_{\a}= V_{\nu_{\a}}^{\fg}$.
\end{prop}

\begin{proof}
Let us denote by $\a^*:\Q^n\lgw \R$ the $\Q$-linear map defined by $\a^*(u)=\langle \a,u\rangle$ for any $u\in\Q^n$. Since the $\a_i$ are $\Q$-linearly independent then $\a^*$ is injective. \\
If $\La$ is a finitely generated sub-semigroup  of $\Z\a_1+\cdots+\Z\a_n$ let $\b_1$,  \ldots , $\b_s$ be generators of $\La$. Then ${\a^*}^{-1}(\La)$ is a finitely generated semigroup  whose generators are $b_1={\a^*}^{-1}(\b_1)$,  \ldots , $b_s={\a^*}^{-1}(\b_s)\in\Z^n$.\\
 If the support of $z\in V^{\fg}_{\nu_{\a}}$ is in $\La$, since $\a^*$ is injective $z$ can be written as  
$$z=\sum_{\underline{k}\in\Z_{\geq0}^s}a_{\underline{k}}\x^{k_1b_1+\cdots+k_sb_s}$$
 where  $a_{\underline{k}}\in \k$ for  all $\underline{k}=(k_1,\cdots,k_s)$. Let us remark that the monomial $a_{\underline{k}}\x^{k_1b_1+\cdots+k_sb_s}$ is $(\a)$-homogeneous of degree $k_1\b_1+\cdots+k_s\b_s$.\\
 Let us write $b_i=b_{1,i}-b_{2,i}$ where $b_{1,i}$, $b_{2,i}\in\Z_{\geq 0}^n$. Then we have

\begin{equation*}\begin{split}\x^{k_1b_1+\cdots+k_sb_s}&=\frac{\x^{k_1b_{1,1}+\cdots+k_sb_{1,s}}}{\x^{k_1b_{2,1}+\cdots+k_sb_{2,s}}}\\
&=\frac{\x^{k_1b_{1,1}+\cdots+k_sb_{1,s}+(\max_i\{k_i\}-k_1)b_{2,1}+\cdots+(\max_i\{k_i\}-k_s)b_{2,s}}}{\x^{(b_{2,1}+\cdots+b_{2,s})\max_i\{k_i\}}}.\end{split}\end{equation*}
Moreover 
$$\max_i\{ k_i\} \leq \max_j\left\{\frac{1}{\b_j}\right\}(k_1\b_1+\cdots+k_s\b_s).$$
This shows that $z\in \V_{\a,\x^{b_{2,1}+\cdots+b_{2,s}}}$ and 
$$i\lgm \max_j\left\{\frac{1}{\b_j}\right\}i$$
is a bounding function of $z$.
\end{proof}


  Then we give the following version of the Implicit Function Theorem inspired by Lemma 1.2 \cite{Ga} (see also Lemma 2.2. \cite{To}):

\begin{prop}\label{IFT}
Let $\a\in\R_{>0}^n$ and let $P(Z)\in \V_{\a,\d}[\langle\g_1, \ldots , \g_s\rangle][Z]$, $P(Z)=\sum_{k=0}^da_kZ^k$, where   $\g_i$ is  homogeneous for all $i$ with respect to $\nu_{\a}$ and $d\geq 2$. \\
Let $u\in \V_{\a,\d}[\langle\g_1, \ldots , \g_s\rangle]$ such that $\nu_{\a}(P(u))>2\nu_{\a}(P'(u))$. Let $\frac{\wdt{\d}}{\d^m}$ denote the initial term of $P'(u)$ with respect to $\nu_{\a}$.\\ Then
there exists a unique solution $\ovl{u}$ in $\V_{\a,\d\wdt{\d}}[\langle\g_1, \ldots , \g_s\rangle]$ of $P(Z)=0$ such that $$\nu_{\a}(\ovl{u}-u)\geq \nu_{\a}(P(u))-\nu_{\a}(P'(u)).$$

\end{prop}

\begin{proof}
$\bullet$ By replacing $P(Z)$ by $P(u+Z)$ we can assume that $u=0$. In this case we have that $P(u)=P(0)=a_0$ and $P'(u)=P'(0)=a_1$.\\
The valuation $\nu_{\a}$ is defined on the ring $\V_{\a,\d}[\langle\g_1, \ldots , \g_s\rangle]$ and we denote by $V$ its valuation ring. We  denote by $\wdh{V}$ the  completion of $V$. Let $V^{\fg}$ be the subring of $\wdh{V}$ of all elements of $\wdh{V}$ whose $\nu_{\a}$-support is included in a finitely generated semigroup. Then $V^{\fg}$ is a Henselian  local ring by Lemma \ref{lemma1}. We set $Z=\frac{\wdt{\d}}{\d^m}Y$. Thus we are looking for solving the following equation:

$$\wdt P(Y):=\frac{\d^{2m}}{\wdt{\d}^2}P\left(\frac{\wdt{\d}}{\d^m}Y\right)=a_0\frac{\d^{2m}}{\wdt{\d}^2}+a_1\frac{\d^{m}}{\wdt{\d}}Y+a_2Y^2+\cdots+a_d\frac{\wdt{\d}^{d-2}}{\d^{(d-2)m}}Y^d=0.$$
From now on we denote by $\wdt a _k$ the coefficients of $\wdt P(Y)$:
$$\wdt a_k:=a_k\frac{\wdt \d^{k-2}}{\d^{(k-2)m}}\ \ k=0, \ldots , d.$$
Since 
$$\nu_{\a}(a_0)=\nu_\a(P(0))>2\nu_{\a}(P'(0))=2\nu_{\a}(a_1)=\nu_{\a}\left(\frac{\wdt{\d}^2}{\d^{2m}}\right)$$
we have that $\wdt a_0\in V^{\fg}$. By assumption we have that $\nu_\a(\wdt a_1)=0$ thus $\wdt a_1\in V^{\fg}$. Since $\nu_\a\left(\frac{\wdt \d}{\d^m}\right)\geq 0$ we have that $\wdt a_k\in V^{\fg}$ for all $k\geq 2$.
Moreover we have 
$$\nu_\a(\wdt{P}(0))>0 \text{ and } \nu_{\a}(\wdt{P}'(0))=0.$$
Thus by  Hensel Lemma  this equation has a unique solution  $y\in  V^{\fg}$ such that 
$$\nu_\a(y)=\nu_\a\left(a_0\frac{\d^{2m}}{\wdt{\d}^2}\right)=\nu_{\a}(P(0))-2\nu_{\a}(P'(0))>0.$$ \\
Hence there exists a unique solution $z:=\frac{\wdt{\d}}{\d^m}y\in V^{\fg}$ of the equation $P(Z)=0$ such that 
$$\nu_\a(z)\geq \nu_\a(P(u))-\nu_\a(P'(u)).$$
Now we have to show that $z$ or $y\in  \V_{\a,\d\wdt{\d}}[\langle\g_1, \ldots , \g_s\rangle]$.\\

%
%
$\bullet$ We can write $y=\wdt a_0 \wdt y$ where $\wdt y\in V^{\fg}$ and $\nu_{\a}(\wdt y)=0$. Then $\wdt y$ is a root of the polynomial
$$\wdt P(\wdt a_0 Y)=\wdt a_0+\wdt a_1\wdt a_0Y+\wdt a_2\wdt a_0^2Y^2+\cdots+\wdt a_d\wdt a_0^d Y^d=$$
$$=\wdt a_0\left(1+\wdt a_1 Y+\wdt a_2\wdt a_0 Y^2+\cdots+\wdt a_d\wdt a_0^{d-1}Y^d\right)$$
and $y\in\V_{\a,\d\wdt\d}[\langle\g_1, \ldots ,\g_s\rangle]$ if and only if $\wdt y\in \V_{\a,\d\wdt\d}[\langle\g_1, \ldots ,\g_s\rangle]$.\\
Since $\nu_{\a}(\wdt a_0)>0$, by  replacing $\wdt P$ (resp. $y$) by $1+\wdt a_1 Y+\wdt a_2\wdt a_0 Y^2+\cdots+\wdt a_d\wdt a_0^{d-1}Y^d$ (resp. $\wdt y$), we may assume that
$$\nu_{\a}(\wdt a_i)>0 \text{ for } i\geq 2.$$
In this case we have $\wdt a_0=1$, $\ini_\a(\wdt a_1)=1$ and $\ini_\a(y)=-1$.\\

 Let $\La$ be a finitely generated sub-semigroup of $\R_{\geq 0}$ containing the  $\nu_{\a}$-supports of $y$ and the $\wdt a_k$. We denote by $\l_l$, $l\in\Z_{\geq 0}$, its elements ordered as follows:
 $$\l_0:=0<\l_1<\l_2<\cdots<\l_l<\l_{l+1}<\cdots.$$

Let us expand the coefficients of $\wdt P(Y)$ as 
$$\wdt a_k=\sum_{l\in\Z_{\geq 0}}\wdt{a}_{k,\l_l}$$
 where $\wdt{a}_{k,\l_l}$ is homogeneous of degree $\l_l$ with respect to $\nu_{\a}$. For every $l\in\N$ let $Y_{\l_l}$ be a new variable and set  $Y^*:=\sum_{l\in\N}Y_{\l_l}$.  We extend the valuation $\nu_{\a}$ to $V^{\fg}[Y_{\l_1}, \ldots ,Y_{\l_l},...]$ by setting $\nu_{\a}(Y_{\l_l}):=\l_l$ for any $l\in\N$. We may write formally $\wdt P(Y^*)=\sum_{l}\wdt P_{\l_l}(Y^*)$ where $\wdt P_{\l_l}(Y^*)\in\Z[\wdt{a}_{k,\l_i}][Y_{\l_j}]$ is the homogeneous term of degree $\l_l$ with respect to $\nu_{\a}$. \\
Since $\ini_\a(\wdt a_1)=1$ the equation
 
\begin{equation}\label{eq_solving}\wdt P(Y)=\wdt a_0+\wdt a_1 Y+\wdt a_2 Y^2+\dots+\wdt a_d Y^d=0,\end{equation}
 
 where $Y$ is replaced by $Y^*$, yields the following equation, for every $l\in\Z_{\geq 0}$:
 
 \begin{equation}\label{eq_eisenstein}\wdt P_{\l}(Y^*)=Y_{\l_l}+Q_{\l_l}(Y^*)=0. \end{equation}
 
 where $Q_{\l_l}(Y^*)\in\Z[\wdt{a}_{k,\l_i}][Y_{\l_j}]$ is a polynomial  depending only on the variables $\wdt{a}_{k,\l_i}$ ($\l_i\leq \l_l$) and $Y_{\l_j}$ ($j<l$). Since $y$ is a solution of Equation \eqref{eq_solving}, by replacing $Y^*$ by $y$ we have $\wdt P_{\l_l}(y)=0$, hence 
$$y_{\l_l}=-Q_{\l_l}(y_{\l_j}, j<l)\ \ \ \forall l\in\N.$$
So by induction on $l$ we see that we may write
$$y_{\l_l}=\frac{c_l}{(\d\wdt{\d})^{m(\l_l)}}$$
for some $c_l\in\k[x][\g_1, \ldots , \g_s]$ and $m(\l_l)\in\N$ for all $l$. \\
 
Let $i\lgm ai+b$ be a common bounding functions of the coefficients of $\wdt a_0$, $\wdt a_1$, $\wdt a_2$, \ldots , $\wdt a_d$ seen as elements of $\V_{\a,\d\wdt{\d}}[\langle \g_1, \ldots , \g_s\rangle]$.  By Remark \ref{bounding} we may assume that $b=0$ since $\nu_\a(\wdt a_k)>0$ for $k\geq 2$ and $\ini_\a(\wdt a_0)=\ini_\a(\wdt a_1)=1$.\\
Thus we have
$$(\d\wdt\d)^{a\l_i}\wdt a_{k,\l_i}\in\k[x][\g_1, \ldots ,\g_s]\ \ \ \forall i.$$
Let $m(\l_l)$ be the least integer such that $(\d\wdt\d)^{m(\l_l)} y_{\l_l}\in \k[x][\g_1, \ldots ,\g_s]$.  We will show by induction on $l$ that  

\begin{equation}\label{eq_m(l)}m(\l_l)\leq a\l_l.\end{equation}

This inequality is satisfied for $l=0$ since $\ini_\a(y)=-1$ implies that $m(\l_0)=0$.\\
We fix  an integer $l>0$ and we assume that \eqref{eq_m(l)} is satisfied for any integer less than $l$.\\
Let $Q$ be a monomial of $Q_i(Y^*)$. We may write 
$$Q=\wdt{a}_{k,\l_i}y_{\l_{j_1}}\cdots y_{\l_{j_k}}$$
 where $k\leq d$, $j_1\leq \cdots \leq j_k<l$ and
 $\l_i+\l_{j_1}+\cdots+\l_{j_k}=\l_l.$ 
 \\
 Then 
 $$(\d\wdt\d)^{a\l_i+a(\l_{j_1}+\cdots+\l_{j_k})}Q=(\d\wdt\d)^{a\l_l}Q\in \k\lb x\rb[\g_1,\cdots,\g_s].$$
 This proves \eqref{eq_m(l)}. So $y\in \V_{\a,\d\wdt\d}[\langle\g_1, \ldots , \g_s\rangle]$.
 \end{proof}

  We  deduce from this proposition the main result of this part (Theorem \ref{main2}) which is a general version of Eisenstein Theorem for algebraic power series over $\Q$. First we recall the classical Eisenstein Theorem:
  
  \begin{theorem}\cite{Ei}
  Let  $\displaystyle \sum_{k\in\Z_{\geq 0}}a_kT^k\in\Q\lb T\rb$ be a power series algebraic over $\Q[T]$. Then there exists an integers $a\in\N$ such that 
  $$a^{k+1}a_k\in\Z$$
   for every integer $k$.
  \end{theorem}

  \begin{theorem}[Eisenstein Theorem]\label{main2}
Let $\k$ be a field of characteristic zero. Let $\a\in\R_{>0}^n$ and let us set $N=\dim_{\Q}(\Q\a_1+\cdots+\Q\a_n)$.
Let 
$$P(Z)\in\V_{\a}[\langle\g_1, \ldots , \g_s\rangle][Z]$$
 be a monic polynomial  where $\g_1$, \ldots , $\g_s$ are  homogeneous elements with respect to $\nu_{\a}$. Then there exist integral homogeneous elements with respect to $\nu_{\a}$, denoted by $\g'_{1}$,... $\g_N'$, such that $P(Z)$ has all its roots in  $\V_{\a}[\langle\g_1', \ldots ,\g_N'\rangle]$.
\end{theorem}

\begin{proof}
 By replacing $P(Z)$ by one of its irreducible factors we may assume that  $P(Z)$ is irreducible. Let 
 $$z\in V^{\fg}_{\nu_{\a}}[\langle \g_1', \ldots ,\g'_N\rangle]$$
  be a root of $P(Z)$ where $\g'_i$ is an integral homogeneous with respect to $\nu_{\a}$ (by Theorem \ref{main} such a $z$ exists). Since $P(Z)$ is irreducible, then $P'(z)\neq 0$. Let us set $i_0:=\max\{\nu_{\a}(z-z')\}$ where the maximum is taken over all the roots $z'$ of $P$ different from $z$. Let us take $\wdt{z}\in \V_{\a}[\langle \g_1', \ldots ,\g'_N\rangle]$ such that 
  \begin{equation}\label{key_bound}\nu_{\a}(\wdt{z}-z)>\max\{2\nu_{\a}(P'(z)),i_0+\nu_{\a}(P'(z))\}.\end{equation}
 For instance if we expand $z=\sum_{i\in\La}z_i$ where $\La$ is a finitely generated sub-semigroup  of $\R_{\geq 0}$ and $z_i$ is a homogeneous element of degree $i$ with respect to $\nu_\a$ we can choose
 $$\wdt z:=\sum_{i\leq \max\{2\nu_{\a}(P'(z)),i_0+\nu_{\a}(P'(z))\}}z_i.$$
 By replacing $P(Z)$ by $P(Z+\wdt z)$ we may assume that $\wdt z=0$. In this case $P'(\wdt z)=P'(0)=a_{d-1}$ if we write
 $$P(Z)=Z^d+a_1Z^{d-1}+\cdots+a_{d-1}Z+a_d.$$
 Now $\ini_{\nu_\a}(a_{d-1})\in\k(\x)[\g_1', \ldots ,\g'_N]$ so if we denote by $a$ the product of the conjugates of $\ini_{\nu_\a}(a_{d-1})$ over $\k(\x)$ different from $\ini_{\nu_\a}(a_{d-1})$ we have $\ini_{\nu_\a}(aa_{d-1})\in\k(\x)$ and $a$ is a homogeneous element with respect to $\nu_\a$ by Lemma \ref{val_ini}. Let $b$ be a homogeneous element such that $b^{d-1}=a$. By Proposition \ref{bound} we may assume that $b\in V_{\nu_\a}^{\fg}[\langle \g_1', \ldots ,\g'_N\rangle]$. We have that
 $$b^dP\left(\frac{Z}{b}\right)=b^d\left(\frac{1}{b^d}Z^d+\frac{a_1}{b^{d-1}}Z^{d-1}+\cdots+\frac{a_{d-1}}{b}Z+a_d\right)$$
 $$=Z^d+a_1bZ^{d-1}+\cdots+b^{d-1}a_{d-1}Z+b^da_d.$$
 By replacing $P(Z)$ by $b^dP\left(\frac{Z}{b}\right)$ we may assume that $\ini_{\nu_\a}(P'(\wdt z))=\ini_{\nu_\a}(a_{d-1})\in\k(\x)$.\\

  Since $P(\wdt z)-P(z)\in (\wdt z -z)$ then by Inequality \eqref{key_bound}
   $$\nu_{\a}(P(\wdt{z}))>2\nu_{\a}(P'(\wdt{z}))$$
   and 
   $$\nu_{\a}(P(\wdt{z}))>i_0+\nu_{\a}(P'(z)).$$
   In the same way, since $P'(\wdt z)-P'(z)\in (\wdt z -z)$,  Inequality \eqref{key_bound} yields
 $$\nu_{\a}(P'(\wdt{z}))  =\nu_{\a}(P'(z)).$$
 Then we apply Proposition \ref{IFT} (with $u:=\wdt z=0$), and we get a root $\ovl{z}\in\V_{\a}[\langle\g_1', \ldots ,\g'_N\rangle]$ of $P(Z)$ such that 
 $$\nu_{a}(\wdt z-\ovl{z})\geq \nu_{\a}(P(\wdt z))-\nu_{\a}(P'(\wdt z))>i_0.$$
 Thus
 $$\nu_{\a}(z-\ovl{z})=\nu_{\a}(z-\wdt z+\wdt z-\ovl z)>i_0 =\max_{\begin{subarray}{l}z'\neq z\\  P(z')=0\end{subarray}}\{\nu_{\a}(z-z')\}.$$
 Hence $z=\ovl{z}\in\V_{\a}[\langle\g_1', \ldots ,\g'_N\rangle]$. 
\end{proof} 

\begin{corollary}\label{K_alg}
 The field $\K_{\nu_{\a}}^{\alg}$ is a subfield of $\KK_{\a}$.
 \end{corollary}
 \begin{proof}
 Let $z\in\K_{\nu_{\a}}^{\alg}$ and let $P(Z)=a_0Z^d+\cdots+a_d\in\k\lb \x\rb [Z]$ be a polynomial such that $P(z)=0$. Then $a_0z\in\K_{\nu_{\a}}^{\alg}$ is a root of the polynomial  $a_0^{d-1}P(Z/a_0)=Z^d+a_1Z^{d-1}+a_2a_0Z^{d-2}+\cdots+a_da_0^{d-1}$ which is a monic polynomial. Hence $a_0z\in\V_{\a}$ by Theorem \ref{main2} and $z\in\KK_{\a}$.
 
 \end{proof}

 \begin{ex}
Let us assume that $\Di_Z(P(Z))$ is normal crossing after a formal change of coordinates and let us assume that $\k$ is algebraically closed. This means that  there exist power series $x_i(\y)\in (\y)\k\lb \y\rb  $ ($\y=(y_1, \ldots ,y_n)$), for $1\leq i\leq n$, such that the morphism of $\k$-algebras $\phi : \k\lb \x\rb  \lgw \k\lb \y\rb  $ defined by $\phi(f(\x))=f(x_1(\y), \ldots ,x_n(\y))$ is an isomorphism, and such that
$$\phi(\Di_Z(P(Z)))\k\lb \y\rb  =y_1^{e_1}\cdots y_m^{e_m}\k\lb \y\rb  ,\ m\leq n.$$
By Abhyankar-Jung Theorem \cite{Ab} (or \cite{K-V}, \cite{P-R}, \cite{MS}),  the roots of $P(Z)$ can be written as  
$$t_k=\sum_{l=0}^dt_{k,l}(\y)\w^l$$
 where $\w=\y^{\b}$ for some $\b\in\Q_{\geq 0}^m\times\{0\}^{n-m}$, $d\in\Z_{\geq 0}$ and the $t_{k,l}(\y)$ are power series with coefficients in $\k$. Let us write:
$$\b=\left(\frac{b_1}{e}, \ldots ,\frac{b_m}{e},0, \ldots ,0\right)$$
 for some non negative integers $b_1$,  \ldots , $b_m$ and $e\in\N$.
Let us denote by $f_i(\x)$, $1\leq i\leq n$,  the power series satisfying $\phi(f_i(\x))=y_i$.\\
Let $\a\in\N^n$ and write $f_i(\x)=l_{i,\a}(\x)+\e_{i,\a}(\x)$ where $l_{i,\a}(\x)$   is $(\a)$-homogeneous  and $\nu_{\a}(\e_i(\x))> \nu_{\a}(l_{i,\a}(\x))$  for any $i$.
Thus we have for $1\leq i\leq m$:
$$y_i^{\frac{1}{e}}=l_{i,\a}(\x)^{\frac{1}{e}}\left(1+\frac{\e_{i,\a}(\x)}{l_{i,\a}(\x)}\right)^{\frac{1}{e}}=l_{i,\a}(\x)^{\frac{1}{e}}\left(1+\sum_{k\geq 1}c_k\frac{\e_{i,\a}(\x)^k}{l_{i,\a}(\x)^k}\right)$$
where $c_k\in\Q$ for all $k$ - here
$$c_k=\frac{\frac{1}{e}\left(\frac{1}{e}-1\right)\cdots\left(\frac{1}{e}-k+1\right)}{k!}.$$
Hence
\begin{equation*}\begin{split}\w=&y_1^{\frac{b_1}{e}}\cdots y_m^{\frac{b_m}{e}}=\\
&l_{1,\a}(\x)^{\frac{b_1}{e}}\cdots l_{m,\a}(\x)^{\frac{b_m}{e}}\prod_{j=1}^m\left(1+\sum_{k\geq 1}c_k\frac{\e_{j,\a}(\x)^k\prod_{p\neq j}l_{p,\a}(\x)^k}{(\prod_{p=1}^ml_p(\x))^k}\right)^{b_j}.\end{split}\end{equation*}
We remark that  $\Di_Z(P(Z))=\prod_{p=1}^ml_{p,\a}(\x)^{e_p}+\e(\x)$ with 
$$\nu_{\a}(\e(\x))>\nu_{\a}(\prod_{p=1}^ml_{p,\a}(\x)^{e_p}).$$
  Let  $\g:=\prod_{j=1}^ml_{j,\a}(\x)^{\frac{b_j}{e}}$ be a root of the polynomial 
$$Z^e-\prod_{j=1}^ml_{j,\a}(\x)^{b_j}$$ (in particular it is an integral homogeneous element with respect to $\nu_{\a}$),  and set  $\d:=\prod_{j=1}^ml_{j,\a}(\x)^{e_p}$. Here $\d$ is the $(\a)$-initial term of the discriminant of $P(Z)$. Hence we obtain the following three cases:
\begin{itemize}
\item[i)] If $\phi$ is a linear change of coordinates (i.e. $\a=(1, \ldots ,1)$ and $\e_{i,\a}=0$ $\forall i$), then the roots of $P(Z)$ are in $\k\lb \x\rb  [\g]$ (since in this case $\w=\g$).
\item[ii)] If $\phi$ is a quasi-linear change of variables (i.e. $\a\in\N^n$ and $\e_{i,\a}=0$ $\forall i$), then the roots of $P(Z)$ are still in $\k\lb \x\rb  [\g]$ (since in this case we also have $\w=\g$).
\item[iii)] If (at least) one of the $\e_{i,\a}$ is not zero, then the roots of $P(Z)$ are in $\V_{\a,\d}[\langle\g\rangle]$.
\end{itemize}
This example will be generalized later (see Theorem \ref{strong_analytic}).
 \end{ex}
 
 \begin{ex}
 Let $P(Z)=Z^2+2aZ+b$ where $a$ and $b$ are power series over $\k$ and let $\a\in \Q_{>0}^n$. Let $\d$  denote the $(\a)$-initial term of the discriminant of  $P(Z)$, i.e. the $(\a)$-initial term of $a^2-b$. Then the  roots of  $P(Z)$ are of the form $-a+\sqrt{a^2-b}\in\V_{\a,\d}[\langle\g\rangle]$ where $\g$ is a root square of $\d$. 
 \end{ex}

 \begin{ex}
 Let $P(Z)=Z^3+3x_2^2Z-2(x_1^3+\e)$ where $\e$ is a homogeneous polynomial of degree greater or equal to 4.
 Its discriminant is $D:=x_1^6+x_2^6+2x_1^3\e+\e^2$ whose initial term is  $x_1^6+x_2^6$. The roots of $P$ are

  $$ a\sqrt[3]{x_1^3+\e+\sqrt{D}}+b\sqrt[3]{x_1^3+\e-\sqrt{D}}$$
with $(a,b)=(1,1)$, $(j,j^2)$ or $(j^2,j)$. But we have

$$\sqrt[3]{x_1^3+\e+\sqrt{D}}=\g_1\sqrt[3]{1+\e+\frac{\g_2}{x_1^3+\g_2}\sqrt{1+\frac{2x_1^3\e+\e^2}{\d}}-\frac{\g_2}{x_1^3+\g_2}}$$
with $\g_2^2=x_1^6+x_2^6$, $\g_1^3=x_1^3+\g_2$ and $\d=x_1^6+x_2^6$ is the initial term of $D$. Thus $$\sqrt[3]{x_1^3+\e+\sqrt{D}} \in\V_{(1,1),\d}\left[\left\langle\g_1,\g_2,\frac{\g_2\e}{x_1^3+\g_2}\right\rangle\right].$$
By doing the same remark for $\sqrt[3]{x_1^3+\e-\sqrt{D}}$, we see that there exist $\g_1$, \ldots , $\g_5$ homogeneous elements with respect to $\ord$ such that the roots of $P(Z)$ are in $\V_{(1,1),\d}[\langle\g_1, \ldots , \g_5\rangle]$. But there is no reason that the roots of $P(Z)$ are  in $\V_{\a,\d}[\langle \g\rangle]$ where $\g$ is one (integral) homogeneous element with respect to $\nu_{\a}$.
      
 \end{ex}


\section{Approximation of monomial valuations by divisorial monomial valuations}\label{section_appro}
In several cases, it will be easier to work with a monomial valuation $\nu_{\a}$ which is divisorial, i.e. such that $\dim_{\Q}(\Q\a_1+\cdots+\Q\a_n)=1$. In order to extend some results which are proven for divisorial monomial valuations to general monomial valuations, we will approximate monomial valuations by divisorial monomial valuations. The aim of this section is to explain how this can be done.

\begin{definition}
Let $\a\in\R_{>0}^n$. Let $\a^* :\Q^n\lgw \R$ be the $\Q$-linear morphism defined by $\a^*(q_1, \ldots ,q_n):=\sum_i\a_iq_i$. We denote by $\Rel_{\a}$ the kernel of this morphism.\\
For any $\e>0$ and $q\in\N$, we define the following set:
$$\Rel(\a,q,\e):=\left\{\a'\in\N^n\ / \ \Rel_{\a}\subset \Rel_{\a'} \text{ and } \ \max_i\left|q-\frac{\a_i'}{\a_i}\right|<q\e\right\}.$$
\end{definition}

\begin{ex}
If $n=4$, and $\a_1=\sqrt{2}$, $\a_2=\sqrt{3}$, $\a_3=13\sqrt{2}+\sqrt{3}$, $\a_4=\sqrt{2}+757\sqrt{3}$, then any $\a'$ of the form $(n_1,n_2,13n_1+n_2,n_1+757n_2)$, where $n_1$, $n_2\in\N_{>0}$, will satisfy $\Rel_{\a}\subset \Rel_{\a'}$.
\end{ex}

\begin{rmk}\label{rmk_Rel}
For $\a$ and $\b\in\R_{>0}^n$ we have
$$\Rel_{\a}\subset\Rel_{\b} \Longleftrightarrow \b\in V\otimes_\Q\R$$
where $V:=(\Ker \a^*)^\perp\subset \Q^n$. 
By definition we have that $\a\in V\otimes_{\Q}\R$. Since $V$ is dense in $V\otimes_{\Q}\R$ there exists $\b\in V$ such that
$$\max_{1\leq i\leq n}\left|1-\frac{\b_i}{\a_i}\right|<\e.$$
Let us write $\b_i=\frac{\a'_i}{q}$ where the $\a'_i$ and $q$ are positive integers. This implies that
$$\max_{1\leq i\leq n}\left|q-\frac{\a'_i}{\a_i}\right|<q\e.$$
Since $\b\in V$ we have that $\a'\in V$ thus $\Rel_\a\subset\Rel_{\a'}$. This shows that
for any given $\a\in\R^n_{>0}$ and   $\e>0$ there always exists $q\in\N$ such that $\Rel(\a,q,\e)\neq \emptyset$.\\ Moreover if $\a\in\N^n$ then $\Rel(\a,q,\e)=\{q\a\}$ if $0<\e<\frac{1}{q\max\{\a_i\}}$. Indeed in this case the only $\a'\in\N^n$ satisfying $\displaystyle\max_i|q\a_i-\a'_i|<q\a_i\e$ is $\a'=q\a$.
\end{rmk}

\begin{lemma}\label{rel}
Let $\a$, $\a'\in\R_{>0}^n$. Then $\Rel_{\a}\subset \Rel_{\a'}$ if and only if every $(\a)$-homogeneous polynomial is a $(\a')$-homogeneous polynomial.\\
Moreover if $\a'\in\Rel(\a,q,\e)$ and if $a(\x)$ is a $(\a)$-homogeneous polynomial then 
$$q(1-\e)\nu_{\a}(a(\x))\leq \nu_{\a'}(a(\x))\leq q(1+\e)\nu_{\a}(a(\x)).$$
\end{lemma}

\begin{proof}
First let us assume that $\Rel_{\a}\subset \Rel_{\a'}$ and let $a(\x)$ be a $(\a)$-homogeneous polynomial. This means that for any $p$, $q\in\N^n$, if $\x^{p}$ and $\x^q$ are two nonzero  monomials of $a(\x)$, then $\sum_i\a_ip_i=\sum_i\a_iq_i$. In particular $p-q\in \Ker(\a^*)$, thus $\sum_i\a'_ip_i=\sum_i\a_i'q_i$. Thus $a(\x)$ is a $(\a')$-homogeneous.\\
On the other hand  let us assume that  every $(\a)$-homogeneous polynomial is a $(\a')$-homogeneous polynomial. Let $r\in \Rel_{\a}$. We can write $r=p-q$ where $p$, $q\in\Q_{>0}^n$. By multiplying $r$ by a positive integer $m$, we may assume that $mp$, $mq\in\N^n$. By assumption on $r$, the polynomial $\x^{mp}+\x^{mq}$ is $(\a)$-homogeneous. Thus it is $(\a')$-homogeneous. This means that $\sum_i\a'_imp_i=\sum_i\a'_imq_i$. Hence $\sum_i\a_i'(p_i-q_i)=0$ and $r=p-q\in \Rel_{\a'}$.\\
Now let $\x^{p}$ be a monomial. Then
$$\nu_{\a'}(\x^{p})=\sum_i\a'_ip_i.$$
But $q(1-\e)\a_i\leq\a'_i\leq q(1+\e)\a_i$ for any $1\leq i\leq n$. This proves both inequalities.

\end{proof}

\begin{ex}
Let $\a\in\N^n$ and $\a'\in\R_{>0}^n$. Then $\Rel_{\a}\subset\Rel_{\a'}$ if and only if there exists $\l\in\R$ such that $\a'=\l\a$. Indeed we have $\dim_\Q(\Rel_\a)=n-1$ hence either $\dim_\Q(\Rel_{\a'})=n$ and $\a'=0$, either $\dim_\Q(\Rel_{\a'})=n-1$ and there exists $\l\in\R^*$ such that $\a'=\l\a$.
\end{ex}

\begin{lemma}\label{approx}
Let $\a\in\R_{>0}^n$ and let $A\in\V_{\a}$. Let us write 
$$A=\sum_{i\in\La}\frac{a_i(\x)}{\d(\x)^{m(i)}}$$
where $\La$ is a finitely generated sub-semigroup of $\R_{\geq 0}$ and $i\lgm m(i)$ is bounded by an affine function.
Then there exists $\e_A>0$ such that for all $0<\e\leq\e_A$, for all $q\in\N$, for all $\a'\in\Rel(\a,q,\e)$, the element $\displaystyle \sum_{i\in\La}\frac{a_i(\x)}{\d(\x)^{m(i)}}$ is in the fraction field of $\V_{\a'}$.\\
Moreover if $A\in \V_{\a}$ is not invertible, i.e. $\nu_{\a}(A)>0$, then we may even choose $\e_A>0$ such that   for all $0\leq \e\leq\e_A$, for all $q\in\N$, for all $\a'\in\Rel(\a,q,\e)$, $\displaystyle \sum_{i\in\La}\frac{a_i(\x)}{\d(\x)^{m(i)}}\in\V_{\a'}$ and this element is not invertible in $\V_{\a'}$.\\
\end{lemma}

\begin{proof}
Let $a$, $b\geq 0$ such that $m(i)\leq ai+b$ for any $i\in\La$. By Lemma \ref{rel} we have
$$ \nu_{\a'}\left(\frac{a_i(\x)}{\d(\x)^{m(i)}}\right)=\nu_{\a'}(a_i(\x))-m(i)\nu_{\a'}(\d(\x))\geq q(1-\e)\nu_\a(a_i(\x))-q(1+\e)m(i)\nu_\a(\d(\x))$$
$$=q(1-\e)i-2q\e m(i)\nu_{\a}(\d(\x)).$$
Let $\e_A$ be a positive real number such that $\e_A<\frac{1}{1+2a\nu_{\a}(\d(\x))}$ and set 
$$\eta:=1-\e_A(1+2a\nu_{\a}(\d(\x)))>0.$$ Then for any $0\leq \e\leq \e_A$, any $q\in\N$ and any $\a'\in\Rel(\a,q,\e)$ we have
$$  \nu_{\a'}\left(\frac{a_i(\x)}{\d(\x)^{m(i)}}\right)\geq  \eta q i-2qb\e\nu_{\a}(\d(\x))\ \ \forall i\in\La.$$
This proves that $\displaystyle \sum_{i\in\La}\frac{a_i(\x)}{\d(\x)^{m(i)}}$ is in the fraction field of $\V_{\a'}$.\\
If $\nu_{\a}(A)>0$, then $a_0(\x)=0$. Let $i_0:=\nu_{\a}(A)$.  Let $\e\geq 0$ be such that  $\e\leq \e_A$ and
$$i_0>\e\left((1+2a\nu_{\a}(\d(\x)))i_0+2b\nu_{\a}(\d(\x))\right).$$
In this case $\nu_{\a'}\left(\frac{a_i(\x)}{\d(\x)^{m(i)}}\right)>0$ for any $i\in\La$, $i\geq i_0$. This proves the second assertion.
\end{proof}



\begin{definition}\label{phi_morp}
Let $\a\in\R_{>0}^n$ and $\a'\in \Rel_\a\cap\N^n$. Then every  $(\a)$-homogeneous polynomial $p(\x)$ is $(\a')$-homogenous by Lemma \ref{rel}. In particular if $\d(\x)$ is an other $(\a)$-homogeneous polynomial and $s\in\N$ then
$$p(x_1\d(\x)^{\a'_1s},\ldots,x_n\d(\x)^{\a'_ns})=p(\x)\d(\x)^{s\nu_{\a'}(p(\x))}$$
is also a $(\a)$-homogeneous polynomial.\\
If $A=\sum_{i\in\La}\frac{a_i(\x)}{\d(\x)^{m(i)}}\in\V_{\a,\d}$ and $\a'\in \Rel_\a\cap\N^n$, we will set 
$$\phi_{\a',s}(A):=\sum_{i\in\La}\frac{a_i}{\d_i^{m(i)}}(x_1\d(\x)^{\a_1's}, \ldots ,x_n\d(\x)^{\a_n's}).$$
Then $\phi_{\a,s} : \V_{\a,\d}\lgw \V_{\a,\d}$ is a ring morphism.
We also define 
$$\psi_{\a',s}(A):=\d^s\phi_{\a',s}(A) \ \ \forall A\in \V_{\a,\d}.$$
\end{definition}

\begin{lemma}\label{lemma2}
Let $\a\in\R_{>0}^n$ and $A\in\V_{\a,\d}$. For any $\e>0$ small enough there exists $s(\e)\in\N$ such that for every $q\in\N$,  $\a'\in\Rel(\a,q,\e)$ and $s\geq s(\e)$:
$$\psi_{\a',s}(A)\in \k\lb \x\rb.$$
If $\nu_{\a}(A)>0$ we may even assume that $\phi_{\a',s}(A)\in\k\lb\x\rb$ for every $q\in\N$,  $\a'\in\Rel(\a,q,\e)$ and $s\geq s(\e)$.
\end{lemma}

\begin{proof}
 Let $a(\x)$, $\d(\x)\in\k[\x]$ be $(\a)$-homogeneous polynomials and let $m\in \N$ such that $\nu_{\a}\left(\frac{a(\x)}{\d(\x)^m}\right)=i$. Let $s\in\N$ and $a'\in\N^n$ such that $\Rel_\a\subset\Rel_{\a'}$. By Lemma \ref{rel} we have
\begin{equation}\label{s}\frac{a(x_1\d(\x)^{\a_1's}, \ldots ,x_n\d(\x)^{\a_n's})}{\d(x_1\d(\x)^{\a_1's}, \ldots ,x_n\d(\x)^{\a_n's})^m}=a(\x)\d(\x)^{s[\nu_{\a'}(a(\x))-\nu_{\a'}(\d(\x))m]-m}.\end{equation}
Now let $A=\displaystyle\sum_{i\in\La}\frac{a_i}{\d^{m(i)}}\in\V_{\a,\d}$ with $m(i)\leq ai+b$ for any $i\in\La$, $\La$ being a finitely generated sub-semigroup of $\R_{\geq 0}$.  Set $d_{\a}:=\nu_{\a}(\d)$.
Thus $\nu_{\a}(a_i)= d_{\a} m(i)+i$ for any $i\in\La$. Hence by Lemma \ref{rel} we have that
$$ \nu_{\a'}(a_i)-m(i)\nu_{\a'}(\d)\geq q(1-\e)\left[d_{\a} m(i)+i\right]-q(1+\e)m(i)d_{\a}$$
\begin{equation}\label{ineq_lemma2} \nu_{\a'}(a_i)-m(i)\nu_{\a'}(\d)\geq q(1-\e)i-2q\e d_{\a}  m(i).\end{equation}
Since $(1-\e)i-2\e d_{\a} m(i)\geq (1-\e)i-2\e d_{\a}(ai+b)$, for every $\e$ small enough there exists $a_{\e}>0$ such that 
$$\nu_{\a'}(a_i)-m(i)\nu_{\a'}(\d)\geq q a_{\e}i$$ for all $q\in\N$, all $\a'\in\Rel(\a,q,\e)$ and all $i\in\La$, $i>0$.  Thus for $s\in\N$  and $i\in\La\backslash\{0\}$ we have that
$$s(\nu_{\a'}(a_i)-m(i)\nu_{\a'}(\d))-m(i)\geq sqa_\e i-m(i)\geq (sqa_\e -a)i-b\geq \left(sqa_\e -a-\frac{b}{\min \La\backslash\{0\}}\right)i.$$
In particular if $s\geq \left(a+\frac{b}{\min  \La\backslash\{0\}}\right)/a_\e$ then 
$$s(\nu_{\a'}(a_i)-m(i)\nu_{\a'}(\d))-m(i)\geq 0$$ and $\frac{a_i}{\d_i^{m(i)}}(x_1\d(\x)^{\a_1's}, \ldots ,x_n\d(\x)^{\a_n's})\in \k\lb \x\rb$ for all $i>0$. Thus if $\nu_\a(A)>0$, $a_0=0$ and $\phi_{\a',s}(A)\in\k\lb\x\rb$ for $s\geq \left(a+\frac{b}{\min  \La\backslash\{0\}}\right)/a_\e$.\\
In the general case where $a_0\neq 0$, if we assume moreover that $s\geq b$, we have that
$$\d(x)^s\frac{a_0}{\d_0^{m(0)}}(x_1\d(\x)^{\a_1's}, \ldots ,x_n\d(\x)^{\a_n's})\in\k\lb \x\rb.$$This proves the lemma.

\end{proof}

When the components of $\a$ are $\Q$-linearly independent, by using Lemma \ref{lemma2}, Theorem \ref{main2}  gives the following generalization of the main result of \cite{MD}:

 \begin{theorem} \cite{MD}\label{MD}
 Let $\k$ be a field of characteristic zero and $\a\in\R_{>0}^n$ such that $\dim_{\Q}(\Q\a_1+\cdots+\Q\a_n)=n$. Then 
 $$\K^{\alg}_{\nu_{a}}\subset\bigcup_{\s}\k\left(\!\left(x^{\b},\b\in\s\cap\Z^n\right)\!\right)$$
 where the first union runs over all rational strongly convex cones $\s$ such that  $\langle \a,\t\rangle >0$ for any $\t\in\s$, $\t\neq 0$. 
 Moreover we have:
 $$\ovl{\K}_{\nu_{\a}}^{\alg}\subset\bigcup_{\s}\bigcup_{\k'}\bigcup_{q\in\N}\k'\left(\!\left(x^{\b},\b\in\s\cap\frac{1}{q}\Z^n\right)\!\right)$$
 where the first union runs over all rational strongly convex cones $\s$ such that  $\langle \a,\t\rangle >0$ for any $\t\in\s$, $\t\neq 0$,  and the second union runs over all the fields $\k'$ finite over $\k$.

 \end{theorem}
 
 \begin{proof}
 In order to prove the first inclusion, by Corollary \ref{K_alg} it is enough to prove that $\KK_{\a}\subset \bigcup_{\s}\k\(x^{\b},\b\in\s\cap\Z^n\)$ or $\V_{\a}\subset \bigcup_{\s}\k\left\lb x^{\b},\b\in\s\cap\Z^n\right\rb$.\\
 Since the $\a_i$ are $\Q$-linearly independent the only $(\a)$-homogeneous polynomials  are the monomials. Let $\o\in\N^n$ and $A$ be an element of $\V_{\a,\x^{\o}}$ : $A=\sum_{i\in\La}\frac{\x^{p(i)}}{\x^{m(i)\o}}$ where $\La$ is a finitely generated sub-semigroup of $\R_{\geq 0}$. We have to prove that $A\in \bigcup_{\s}\k\left\lb x^{\b},\b\in\s\cap\Z^n\right\rb$. Since  $x_1A\in \bigcup_{\s}\k\left\lb x^{\b},\b\in\s\cap\Z^n\right\rb$ implies that $A\in \bigcup_{\s}\k\left\lb x^{\b},\b\in\s\cap\Z^n\right\rb$ we may assume that $\nu_{\a}(A)>0$.\\
 By Lemma \ref{lemma2}, we see that the monomial map $\phi_{\a',s}$ defined by $x_j\lgm x_j\x^{s\a'_j\o}$ maps $A$ onto an element of $\k\lb \x\rb$ for $\a'\in\Rel(\a,q,\e)$, $\e>0$ small enough and $s$ large enough. Such a monomial map is induced by a linear map on the set of monomials and its matrix is
 $$M_1:=\left(\begin{array}{ccccc}
 1+s\o_1\a'_1 & s\o_1\a_2' & s\o_1\a_3' & \cdots & s\o_1\a_n'\\
 s\o_2\a'_1 & 1+s\o_2\a_2' & s\o_2\a_3' & \cdots & s\o_2\a_n'\\
 s\o_3\a_1' & s\o_3\a_2' & 1+s\o_3\a_3' & \cdots & s\o_3\a_n'\\
 \vdots & \vdots & \vdots & \ddots & \vdots\\
 s\o_n\a'_1 & s\o_n\a_2' & s\o_n\a_{3}' & \cdots & 1+s\o_n\a_n'
 \end{array}\right)   $$
 Set
 $$M_2:=\left(\begin{array}{ccccc}
 -s\o_1\a_1' & -s\o_1\a_2' & -s\o_1\a_3' & \cdots & -s\o_1\a_n'\\
 -s\o_2\a_1' & -s\o_2\a_2' & -s\o_2\a_3' & \cdots & -s\o_2\a_n'\\
 -s\o_3\a_1' & -s\o_3\a_2' & -s\o_3\a_3' & \cdots & -s\o_3\a_n'\\
 \vdots & \vdots & \vdots & \ddots & \vdots\\
 -s\o_n\a_1' & -s\o_n\a_2' & -s\o_n\a_{3}' & \cdots & -s\o_n\a_n'
 \end{array}\right)   $$ 
 and let $\chi(t)$ be the characteristic polynomial of $M_2$. Then $\chi(1)=\det(M_1)$. If $\chi(1)=0$, then the vector $\o:=(\o_1, \ldots ,\o_n)$ is an eigenvector of $M_2$ with eigenvalue $1$ since the image of $M_2$ is generated by $\o$. Thus $-s(\o_1\a_1'+\cdots+\o_n\a_n')=1$ which is not possible since $\o_i\geq 0$ and $\a'_i>0$ for any $i$. Thus $\det(M_1)\neq 0$ and $M_1$ is invertible. In particular $\s:=M_1^{-1}(\R_{\geq 0}^n)$ is a rational strongly convex cone. Moreover, since $A\in\V_{\a,\d}$, we have $\langle \a,\t\rangle >0$ for any $\t\in\s$, $\t\neq 0$. 
 Hence $A\in \k\left\lb x^{\b},\b\in\s\cap\Z^n\right\rb$.\\
 \\
  By Example \ref{homo_mono'} integral  homogeneous elements with respect to $\nu_{\a}$ are either finite over $\k$, either of the form $cx_1^{\frac{n_1}{q}}\cdots x_n^{\frac{n_n}{q}}$ for some integers $n_1$,  \ldots , $n_n\in\Z_{\geq 0}$, $q\in\N$ such that $\displaystyle\sum_{j=1}^n\a_jn_j>0$. Using Theorem \ref{main2} and since $\ovl{\K}_{\nu_{\a}}=\underset{\underset{\g_1, \ldots , \g_s}{\lgw}}{\lim}\,\K_{\nu_{\a}}^{\alg}[\g_1, \ldots , \g_s]$ where the $\g_i$ are homogeneous with respect to $\nu_{\a}$,  we have the second inclusion by replacing $\s$ by the rational strongly convex cone generated by $\s$ and the $n$-uples $(n_1, \ldots ,n_n)$ corresponding to the homogeneous elements $\g_1$,  \ldots , $\g_s$.
 
 \end{proof}
 
 \begin{rmk}\label{rmkII}
In fact the proof shows that the field $\ovl{\KK}_{\a}$, as soon as
$$\dim_{\Q}(\Q\a_1+\cdots+\Q\a_n)=n,$$ is the field of Puiseux power series with support in rational strongly convex cones $\s$ such that $\langle \a,\g\rangle\geq 0$ for all $\g\in \s$. Thus $\ovl{\KK}_{\a}$ is the field of $\a$\emph{-positive Puiseux series} according to \cite{A-I}.\\
 \end{rmk}

\begin{lemma}\label{lemma_psi1}
Let $\a\in \R_{>0}^n$ and $\a'\in \Rel_\a\cap\N^n$. Then
$$\psi_{\a',t}\circ\psi_{\a',s}=\psi_{\a',\nu_{\a'}(\d)st+s+t} \ \ \ \forall s,t\in\Z_{\geq 0}.$$
\end{lemma}

\begin{proof}
Let $A=\displaystyle\sum_{i\in\La}\frac{a_i}{\d^{m(i)}}\in \V_{\a,\d}$. Then we have (see Equation (\ref{s}) in the proof of Lemma \ref{lemma2}):
$$\d^s\phi_{\a',s}(A)= \sum_{i\in\La}   a_i(\x)\d(\x)^{s(1+\nu_{\a'}(a_i(\x))-\nu_{\a'}(\d(\x))m(i))-m(i)}.$$
If  $t\in\Z_{\geq 0}$ and $l\in\Z_{\geq 0}$, and $\a(\x)$ and $\d(x)$ are $(\a')$-homogeneous,  we have that
$$\phi_{\a',t}(a(\x)\d(\x)^l)=a(\x)\d(\x)^{t\nu_{\a'}(a)+l(t\nu_{\a'}(\d)+1)}.$$
Thus by denoting
$$p_{\a'}(i):=\nu_{\a'}(a_i(\x))-\nu_{\a'}(\d(\x))m(i) \text{ and } d_{\a'}:=\nu_{\a'}(\d(\x))$$
we obtain
\begin{equation*}\begin{split}\d^t\phi_{\a',t}(\d^s\phi_{\a',s}(A))=&\\
=\d^t\sum_{i\in\La}   a_i&(x_1\d^{\a_1't}, \ldots ,x_n\d^{\a_n't})\times\\
\d(x_1&\d^{\a_1't}, \ldots ,x_n\d^{\a_n't})^{s(1+\nu_{\a'}(a_i(\x))-\nu_{\a'}(\d)m(i))-m(i)}=\\
=\sum_{i\in\La}&a_i\d^{t+t\nu_{\a'}(a_i)+(s+sp_{\a'}(i)-m(i))(td_{\a'}+1)}\\
=\sum_{i\in\La}&a_i\d^{d_{\a'}tsp_{\a'}(i)+tp_{\a'}(i)+sp_{\a'}(i)-m(i)+d_{\a'}ts+t+s}.\end{split}\end{equation*}

In particular we have
\begin{equation}\label{eq1}\d^t\phi_{\a',t}(\d^s\phi_{\a',s}(A))=\d^{d_{\a'}st+s+t}\phi_{\a',d_{\a'}st+s+t}(A)\ \ \forall t\in\Z_{\geq 0}.\end{equation}

\end{proof}

\begin{lemma}\label{lemma_psi2}
Let $\a\in \R_{>0}^n$ and $\a'\in \Rel_\a\cap\N^n$. 
For all $s_1$, $s_2\in\N$ there exist $t_1$, $t_2\in\N$ such that
$$\psi_{\a',t_1}\circ\psi_{\a',s_1}=\psi_{\a',t_2}\circ\psi_{\a',s_2}.$$

\end{lemma}

\begin{proof}
Let $d$ denote $\nu_{\a'}(\d)$. 
Let $p$ be a prime number  and $k\in\N$ such that $p^k$ divides $ds_1+1$ and $ds_2+1$. Then $\gcd(p,d)=1$ and $p^k$ divides $ds_1-ds_2$. Thus $p^k$ divides $s_1-s_2$. This proves that $\gcd(ds_1+1,ds_2+1)$ divides $s_1-s_2$. Thus there exist $t_1\in\Z$ and $t_2\in\Z$ such that $(ds_1+1)t_1-(ds_2+1)t_2=s_2-s_1$. If $t_1t_2< 0$, let say $t_1> 0$ and $t_2< 0$, then $(ds_1+1)t_1-(ds_2+1)t_2>s_1+s_2>|s_1-s_2|$ which is not possible. Thus we have that $t_1t_2\geq 0$. If $t_1\leq 0$ and $t_2\leq 0$, we can replace $t_1$ (resp. $t_2$) by $t_1+k(ds_2+1)$ (resp. by $t_2+k(ds_1+1)$) for some positive integer $k$ large enough. This will allows to assume that $t_1$ and $t_2$ are positive integers.
Hence $$\exists t_1,t_2\in\N, \ ds_1t_1+s_1+t_1=ds_2t_2+s_2+t_2.$$
This proves the lemma by Lemma \ref{lemma_psi1}.

\end{proof}

\begin{definition}\label{ABC}
Now we consider a subring $R$ of $\k\lb \x\rb  $ that is an excellent Henselian local ring with maximal ideal $\m_R$ and satisfying the following properties:
\begin{enumerate}
\item[(A)] $\k[x_1, \ldots ,x_n]_{(\x)}\subset R$,
\item[(B)] $\m_R=(\x)R$ and $\wdh{R}=\k\lb\x\rb$,
\item[(C)] if $p(\x)\in\k[\x]$ is $(\a)$-homogeneous for some $\a\in\R_{>0}^n$ then
$$f(\x)\in R\ \Longrightarrow \ f(p(\x)x_1, \ldots ,p(\x)x_n)\in R.$$
\end{enumerate}
\end{definition}

\begin{rmk}
If $\k$ is a field, the ring of algebraic power series $\k\langle\x\rangle$ is an excellent Henselian local ring satisfying Properties (A), (B) and (C). If $\k$ is a valued field, then the field of convergent power series $\k\{\x\}$ does also. \\
For a field $\k$, the ring $\k\lb x_1, \ldots ,x_r\rb\langle x_{r+1}, \ldots ,x_n\rangle$ for formal power series algebraic over $\k\lb x_1, \ldots ,x_r\rb[x_{r+1}, \ldots ,x_n]$ is also an excellent  Henselian local ring satisfying Properties (A), (B) and (C).
\end{rmk}

%
%

\begin{definition}\label{V^R}
Let $\a\in\R_{>0}^n$ and let $\d$ be a $(\a)$-homogeneous polynomial. Let $R$ be a ring satisfying Definition \ref{ABC}. We set
\begin{equation*}\begin{split}\V^R_{\a,\d}:=\left\{A\in \wdh{V}_{\nu_{\a}} \ / \ \exists \La\right. \text{ a finitely generated}&\text{ sub-semigroup of } \R_{\geq 0},\\
 \forall i\in \La \ \exists a_i\in\k[\x] \text{ $(\a)$-homogeneous}, \exists a,b\geq &\, 0\ \forall i\in \La \ \exists m(i)\in\N \text{ s.t. }  \\
 m(i)\leq ai+b,\ \nu_{\a}\left(\frac{a_i}{\d^{m(i)}}\right)&=i, \  A=\sum_{i\in\La}\frac{a_i}{\d^{m(i)}} \\
 \text{ and } \ \exists \e>0\ \forall q\in\N\ \forall\a' \in \Rel(\a,q,\e) \ &\left.\exists s\in\N \ \text{ such that } \psi_{\a',s}(A)\in R\right\}.
\end{split}
\end{equation*}
Then $\V_{\a}^R$ is the union of the sets $\V^R_{\a,\d}$ when $\d$ runs over all the $(\a)$-homogeneous polynomials.
\end{definition}

\begin{lemma}\label{denominators}
The sets $\V^R_{\a,\d}$ and $\V^R_\a$ are subrings of $\V_{\a,\d}$ and $\V_\a$. 
\end{lemma}

\begin{proof}
Let  $\displaystyle A=\sum_{i\in\La}\frac{a_i}{\d^{m(i)}}$ and $\displaystyle B= \sum_{i\in\La}\frac{b_i}{\d^{n(i)}}\in \V^R_{\a,\d}$. Then there exists $\e>0$ such that $\forall q\in\N$, $\forall \a'\in\Rel(\a,q,\e)$, there exist $s_1,s_2\in\N$ such that
$$\psi_{\a',s_1}(A), \psi_{\a',s_2}(B)\in R.$$
Then by Lemma \ref{lemma_psi1}, Lemma \ref{lemma_psi2} and condition (C) of Definition \ref{ABC}  there exists $s\in\N$ such that 
$$\psi_{\a',s}(A), \psi_{\a',s}(B)\in R.$$
This shows that  $\psi_{\a',s}(A+B)=\psi_{\a',s}(A)+\psi_{\a',s}(B)\in R$ and $A+B\in \V_{\a,\d}^R$.

Now by Lemma \ref{lemma2} we can assume that there exists $s(\e)\in\N$ such that $\psi_{\a',s}(AB)\in\k\lb\x\rb$ for all $s>s(\e)$, for all $q\in\N$ and all $\a'\in\Rel(\a,q,\e)$. On the other hand since $\psi_{\a',s}(A), \psi_{\a',s}(B)\in R$ then $\psi_{\a',\nu_{\a'}(\d)st+s+t}(A)$, $\psi_{\a',\nu_{\a'}(\d)st+s+t}(B)\in R$ for all $t\in\N$ by Lemma \ref{lemma_psi1} and Condition (C) of Definition \ref{ABC}. Thus there exists $s\in\N$ such that
$$\psi_{\a',s}(A), \psi_{\a',s}(B)\in R\text{ and } \psi_{\a',s}(AB)\in\k\lb\x\rb.$$
But  we have that
$$\psi_{\a',s}(A)\psi_{\a',s}(B)=\d^s\psi_{\a',s}(AB)\in R.$$
Hence by Artin Approximation Theorem  (cf. \cite{Po}, \cite{Sp}) $\psi_{\a',s}(AB)\in R.$

Thus  $AB\in\V_{\a,\d}^R$. This proves that  $\V_{\a,\d}^R$ is a ring. \\
Since $\V_\a^R$ is the direct limit of the $\V_{\a,\d}^R$ it is also a ring.
\end{proof}


\begin{ex}\label{ex_conv}
If $\a\in\N^n$  and $R=\C\{\x\}$ is the ring of convergent power series over $\C$, we claim that
\begin{equation*}\begin{split}\V^{\C\{\x\}}_{\a,\d}=\left\{\sum_{i\in\Z_{\geq 0}}\frac{a_i}{\d^{a(i+1)}}\right.\ / \  \forall i\ \ a_i\in \C[\x]  \text{ is $(\a)$-homogeneous},\\
 \nu_{\a}\left(\frac{a_i}{\d^{a(i+1)}}\right)= i,\   a\in\Z_{\geq 0}\\
  \text{ and }\ \exists C, r>0 \text{ such that }  \ |a_i(z)|\leq Cr^i\|z\|^{\nu_{\a}(a_i)}_{\a}\  \  \forall z\in\C^n\big\}
\end{split}\end{equation*}

where $\displaystyle\|z\|_{\a}:=\max_{j=1, \ldots ,n}\left|z_j^{\frac{1}{\a_j}}\right|$ for any $z\in\C^n$.\\
\\
First of all every element $A$ of $\V_{\a,\d}$ is of the form 
$$A=\displaystyle \sum_{i\in\Z_{\geq 0}}\frac{a_i}{\d^{m(i)}}$$
 where $\nu_{\a}\left(\frac{a_i}{\d^{m(i)}}\right)=i$ and $m(i)\leq ai+b$ for some $a$, $b\in\Z_{\geq 0}$. By multiplying the numerator and the denominator of $\frac{a_i}{\d^{m(i)}}$ by $\d^{ai+b-m(i)}$ and replacing $a_i$ by $a_i\d^{ai+b-m(i)}$, we may assume that $m(i)=ai+b$. If $a>b$, we may replace $\frac{a_i}{\d^{ai+b}}$ by $\frac{a_i\d^{a-b}}{\d^{ai+a}}$, if $a<b$ we may replace $\frac{a_i}{\d^{ai+b}}$ by $\frac{a_i\d^{(b-a)i}}{\d^{bi+b}}$. Thus any element of $\V_{\a,\d}$ is of the form $\displaystyle\sum_{i\in\Z_{\geq 0}}\frac{a_i}{\d^{a(i+1)}}$ where $\nu_{\a}\left(\frac{a_i}{\d^{a(i+1)}}\right)=i$ for all $i\in\Z_{\geq 0}$. In this case $\nu_{\a}(a_i)=(a\nu_{\a}(\d)+1)i+a\nu_{\a}(\d)$ for any $i\in\N$.\\
By  Remark \ref{rmk_Rel} $\Rel(\a,q,\e)=\{q\a\}$  for $\e>0$ small enough since $\a\in\N^n$. Then we have (with $s=a$ in Lemma \ref{lemma2}):
$$f(\x):=\psi_{\a,a}(A)=\d(\x)^a\sum_{i\in\Z_{\geq 0}}\frac{a_i}{\d^{a(i+1)}}(x_1\d(\x)^{\a_1a}, \ldots ,x_n\d(\x)^{\a_na})=\sum_{i\in\Z_{\geq 0}}a_i(\x)
$$
and $f(\x)\in\C\lb \x\rb$. Moreover we have for every $q\in\N$
$$f_q(\x):=\psi_{q\a,a}(A)=\d(\x)^a\sum_{i\in\Z_{\geq 0}}\frac{a_i}{\d^{a(i+1)}}(x_1\d(\x)^{\a_1qa}, \ldots ,x_n\d(\x)^{\a_nqa})=\sum_{i\in\Z_{\geq 0}}a_i(\x)\d(\x)^{a(q-1)i}
$$

Thus $f\in\C\{\x\}$ if and only if this power series is convergent on a neighborhood of the origin. This neighborhood may be chosen of the form:
$$B_{\a}(0,r):=\{z\in\C^n\ / \ |z_j|\leq r^{\a_j},\ j=1, \ldots ,n\}.$$
For any $z\in B_{\a}(0,r)$ set $t_j^{\a_j}=z_j$ for $j=1, \ldots ,n$ and $b_i(t)=a_i(z)$ for any $i\in\N$. Then $f$ is convergent on $B_{\a}(0,r)$ if and only $\displaystyle\sum_{i\in\Z_{\geq 0}}b_i(t)$ is convergent on  $B(0,r):=\{t\in\C^n\ / \ |t_j|\leq r, \ j=1, \ldots , n\}$.
But this series is convergent if and only if there exist $c\geq 0$ and $\rho<1$ such that $|b_i(t)|\leq c\rho^i$ for all $i\in\Z_{\geq 0}$ and all $t\in B(0,r)$. Since $b_i(t)$ is a homogeneous polynomial  of degree $\nu_{\a}(a_i)=(ad+1)i+ad$ where $d:=\nu_{\a}(\d)$, we have 
$$\sup_{|t_j|\leq r,j=1, \ldots ,n}|b_i(t)|=r^{(ad+1)i+ad}\sup_{|t_j|\leq 1,j=1, \ldots ,n}|b_i(t)|.$$ 
We see that $f$ is convergent if and only if there exist $C\geq 0$ and $R>0$ such that 
$$\sup_{|z_j|\leq 1,j=1, \ldots ,n}|a_i(z)|=\sup_{|t_j|\leq 1,j=1, \ldots ,n}|b_i(t)|\leq CR^i.$$
This is equivalent to the following inequality for any $z\in\C^n$:
\begin{equation}\label{bound1}|a_i(z)|=|b_i(t)|\leq\max_{j=1, \ldots ,n}\left|t_j\right|\sup_{|t_j|\leq 1, j=1, \ldots ,n}|b_i(t)|\leq CR^i\|z\|_{\a}^{\nu_{\a}(a_i)}.\end{equation}

On the other hand if $f\in\C\{\x\}$ we have seen that there exist $C\geq 0$ and $R>0$ such that 
$$\sup_{|z_j|\leq 1,j=1, \ldots ,n}|a_i(z)|\leq CR^i.$$
Thus 
$$\sup_{|z_j|\leq 1,j=1, \ldots ,n}|a_i(z)\d(z)^{a(q-1)i}|\leq C(RS)^i$$
where $S:=\max_{|z_j|\leq 1,j=1, \ldots ,n}|\d(z)|^{a(q-1)}$. Hence $f_q\in\C\{\x\}$ for every $q\in\N$.
 This proves the claim.\\
\end{ex}

We have the following analogue of Theorem \ref{main2} in the Henselian case:

 \begin{theorem}\label{main2hens}
Let $\k$ be a field of characteristic zero and let $R$ be a subring of $\k\lb \x\rb  $ satisfying Definition \ref{ABC}. Let $\a\in\R_{>0}^n$ and let us set $N=\dim_{\Q}(\Q\a_1+\cdots+\Q\a_n)$.\\
Let $P(Z)\in\V^{R}_{\a}[\langle\g_1, \ldots , \g_s\rangle][Z]$ be a distinguished polynomial of degree $d$ where the $\g_i$  are homogeneous elements with respect to $\nu_{\a}$. Then the roots of $P(Z)$ are in  $\V^R_{\a}[\langle\g_1', \ldots ,\g_N'\rangle]$ for some  integral homogeneous elements $\g_{1}'$, \ldots , $\g_N'$ with respect to $\nu_{\a}$.
\end{theorem}


\begin{proof}
Let $P(Z)=Z^d+a_1Z^{d-1}+\cdots+a_d$ with $a_j\in \V^{R}_{\a}[\langle\g_1, \ldots , \g_s\rangle]$ for $1\leq j\leq d$. By Theorem \ref{main2} we may assume that $P(Z)$ has a root  $z\in\V_{\a,\d}[\langle\g_1, \ldots ,\g_N\rangle]$. We denote 
$$a_i=\sum_{i_1,\cdots,i_N}A_{i,i_1, \ldots ,i_N}\g_1^{i_1}\cdots\g_N^{i_N}\text{  with } A_{i,i_1, \ldots ,i_N}\in \V_{\a,\d},$$
$$z=\sum_{i_1, \ldots ,i_N}z_{i_1, \ldots ,i_N}\g_1t^{i_1}\cdots\g_N^{i_N}\text{ with } z_{i_1, \ldots ,i_N}\in \V_{\a,\d}.$$
 Let us fix $\e>0$, $q\in\N$, $\a'\in\Rel(\a,q,\e)$ and $s$ satisfying Lemma \ref{lemma2} for the $A_{i,i_1, \ldots ,i_N}$ and for the $z_{i_1, \ldots ,i_N}$. For convenience we denote by $\phi$ the morphism $\phi_{\a',s}$ defined in Definition \ref{phi_morp}. Then if $A$ denotes one of the $A_{i,i_1, \ldots ,i_N}$ or the $Z_{i_1, \ldots ,i_N}$  we have $\phi(A)\in \V_{\a',\d}$ by Lemma \ref{lemma2}. We set 
 $$R:=\V_{\a,\d}\cap \phi^{-1}(\V_{\a',\d})$$
 and $R'$ denotes the subring of $\V_{\a,\g}[\langle\g_1, \ldots , \g_s\rangle]$ of elements $\sum_{i_1, \ldots ,i_N}A_{i_1, \ldots ,i_N}\g_1^{i_1}\cdots\g_N^{i_N}$ whose coefficients $A_{i_1, \ldots ,i_N}$ are in $R$. \\
 Of course $\phi$ induces a morphism $R\lgw \V_{\a',\d}$ but we have the following lemma:

\begin{lemma}\label{lemma_ext}
Let $\g_i$ be  homogeneous elements  with respect to $\nu_{\a}$ for $1\leq i\leq N$. Then there exist homogeneous element $\g'_i$ with respect to $\nu_{\a'}$, $1\leq i\leq N$, such that, 
 for any finite number of elements $A_{i_1, \ldots ,i_N}\in\V_{\a,\d}$, 
 $$\phi\left(\sum_{i_1, \ldots ,i_N} A_{i_1, \ldots ,i_N}{\g_1}^{i_1}\cdots{\g_N}^{i_N}\right):=\sum_{i_1, \ldots ,i_N}\phi(A_{i_1, \ldots ,i_N}){\g'_1}^{i_1}\cdots {\g'_N}^{i_N}$$
  defines an extension of $ \phi$ from $R'$ to $\V_{\a',\d}[\langle \g'_1, \ldots ,\g'_N\rangle]$.
\end{lemma}

\begin{proof}[Proof of Lemma \ref{lemma_ext}]
Let us assume that $\g_i$ is a homogeneous element of degree $e_i$ with respect to $\nu_{\a}$. Let 
$$Q_i(Z):=g_{i,0}(\x)Z^{q_i}+g_{i,1}(\x)Z^{q_i-1}+\cdots+g_{i,q_i}(\x)$$
 be a polynomial such that $Q_i(\g'_i)=0$ and such that $g_{i,j}(\x)$ is a $(\a)$-homogeneous polynomial of degree $d_i+je_i$ for some $d_i$.\\
Then $g_{i,j}(\x)$ is a $(\a')$-homogeneous polynomial of degree $d'_i+je'_i$ for some constants $d'_i$ and $e'_i$. Indeed, if $a$, $b$ and $c$ are $(\a)$-homogeneous polynomials and $\nu_{\a}(a)-\nu_{\a}(b)=\nu_{\a}(b)-\nu_{\a}(c)$, then $ac$ and $b^2$ are two $(\a)$-homogeneous polynomials of same degree, i.e. $ac-b^2$ is $(\a)$-homogeneous. Then, by Lemma \ref{rel}, $ac-b^2$ is $(\a')$-homogeneous, thus $\nu_{\a'}(a)-\nu_{\a'}(b)=\nu_{\a'}(b)-\nu_{\a'}(c)$.\\ 
 Set $\ovl{Q}_i(Z)=\d^{se_i'q_i}Q_i\left(\frac{Z}{\d^{se'_i}}\right)$. We have 
 $$\ovl{Q}_i(Z):=g_{i,0}(\x)Z^{q_i}+g_{i,1}(\x)\d(\x)^{se'_i}Z^{q_i-1}+\cdots+g_{q_i}(\x)\d(\x)^{se'_iq_i}.$$
 For any $i$ let $\g'_i$ denote a root of $\ovl{Q}_i(Z)$. So $\g'_i$ is a homogeneous element of degree $e_i'(1+\nu_{\a'}(\d(\x))s)$ with respect to $\nu_{\a'}$. Then it is straightforward to check that 
 $$\phi\left(\sum A_{i_1, \ldots ,i_N}{\g_1}^{i_1}\cdots {\g_N}^{i_N}\right)=\sum\phi(A_{i_1, \ldots ,i_N}){\g'_1}^{i_1}\cdots {\g'_N}^{i_N}$$
  defines an extension of $ \phi$ from $R'$ to $\V_{\a',\d}[\langle \g'_1, \ldots ,\g'_N\rangle]$. 
 \end{proof}
By Lemmas \ref{lemma_ext},  \ref{lemma_psi1} and \ref{lemma_psi2},  and Property (C) we can assume that $s$ is large enough for having that
$$\d^{js}\phi(A_j)\in R[\g'_1, \ldots ,\g'_N]$$
 for $1\leq j\leq d$. Again by applying Lemmas \ref{lemma_ext},  \ref{lemma_psi1} and \ref{lemma_psi2} we may even assume that 
 $$\d^s \phi(z)\in\k\lb \x\rb  [\g'_1, \ldots ,\g'_N]$$
  by taking $s$ large enough. Thus $z':=\d^s \phi(z)\in\k\lb \x\rb  [\g'_1, \ldots ,\g'_N]$ is a root of the polynomial 
  $$\ovl{P}(Z):=Z^d+\d^s \phi(A_1)Z^{d-1}+\cdots+\d^{ds} \phi(A_d)\in R[Z].$$
  
   Let us write 
$$z':=\sum_{i_1, \ldots ,i_N}z'_{i_1, \ldots ,i_N}{\g'_1}^{i_1}\cdots {\g'_N}^{i_r}$$ with $z'_{i_1, \ldots ,i_N}\in \k\lb x\rb$ for any $i_1$,  \ldots , $i_N$.  Let us set
 $$Z:=\sum_{i_1, \ldots ,i_N}Z_{i_1, \ldots ,i_N}{\g'_1}^{i_1}\cdots {\g'_N}^{i_N}$$
  where $Z_{i_1, \ldots ,i_N}$ are new variables.   Solving $P(Z)=0$ is equivalent to solve a finite system $(\mathcal{S})$ of polynomial equations in the variables $Z_{i_1, \ldots ,i_N}$ with coefficients in $R$, just by replacing $Z$ by $\sum_{i_1, \ldots ,i_N}Z_{i_1, \ldots ,i_N}{\g_1}^{i_1}\cdots {\g_N}^{i_N}$  and replacing the high powers of the $\g_i$ by smaller ones using the division by the $Q_i(Z_i)$.  By Artin Approximation Theorem  (cf. \cite{Po}, \cite{Sp}), the set of solutions of $(\mathcal{S})$ in $R$ is dense in the set of solutions in $\k\lb \x\rb$, but since $P(Z)=0$ has a finite number of solutions, then $(\mathcal{S})$ has a finite number of solutions and they are in $R$. Thus $z'_{i_1, \ldots ,i_N}\in R$ for all $i_1$,  \ldots , $i_N$, hence $z'\in R[\g'_1, \ldots ,\g'_N]$.  This proves that $z\in\V_{\a,\d}^R[\langle\g_1, \ldots ,\g_N\rangle]$.\end{proof}


\section{A generalization of Abhyankar-Jung Theorem}\label{part_AJ}

\begin{definition}\label{D}
Let $\a\in\N^n$ and let $\theta\in\C[\x]$ be a $(\a)$-homogeneous polynomial. Let $a>0$, $C>0$ and $\eta>0$. Set :
$$\mathcal{D}_{\th,C,a,\eta}:=\left(\bigcup_{\begin{subarray}{c}K>0,\e>0\\ \e<K^{a}C\end{subarray}}C_{K,\e}\right)\bigcap B(0,\eta)$$
where  $B(0,\eta)$ is the open ball centered in 0 and of radius $\eta$ and 
$$C_{K,\e}:=\left\{x\in\C^n\ /\ d_{\a}(x,\th^{-1}(0))>K\|x\|_{\a} \text{ and } \|x\|_{\a}<\e\right\}$$
where  $\|.\|_{\a}$ is defined in Example \ref{ex_conv} and $d_{\a}$ is defined as follows: for any $x$, $y\in\C^n$ let us denote by $x_i^{\frac{1}{\a_i}}$ (resp. $y_i^{\frac{1}{\a_i}}$) a complex $\a_i$-th root of $x_i$ (resp. $y_i$) and let $\U_i$ be the set of $\a_i$-roots of unity. Then we define $\displaystyle d_{\a}(x,y):=\max_i\inf_{\xi\in\U_i}\left|x_i^{\frac{1}{\a_i}}-\xi y_i^{\frac{1}{\a_i}}\right|$ and  $\displaystyle d_{\a}(x,\theta^{-1}(0)):=\inf_{x'\in\theta^{-1}(0)}d_{\a}(x,x')$.
\end{definition}
Then $\mathcal{D}_{\th,C,a,\eta}$ is the complement of a hornshaped neighborhood of $\{\th=0\}$ as we can see on the following picture (here $n=2$ and $\a=(1,1)$):
$$
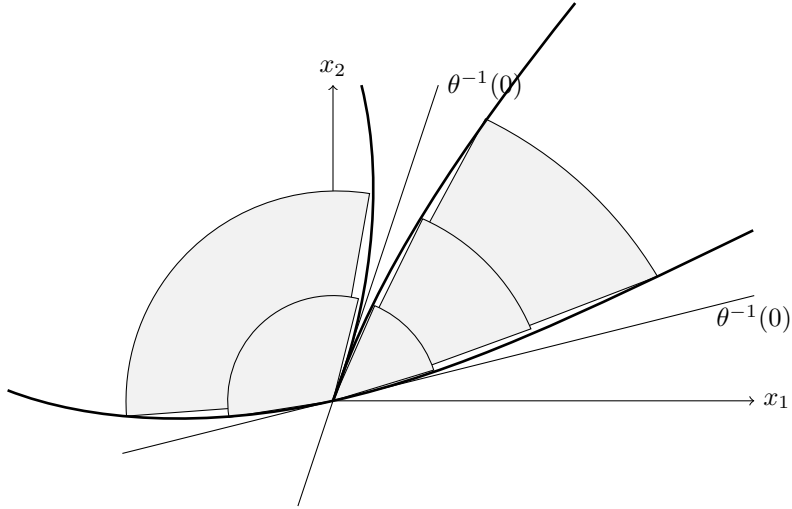
\begin{tikzpicture}[scale=1.4]
   
    \draw [<->] (0,3) node (yaxis) [above] {$x_2$}
        |- (4,0) node (xaxis)[right] {$x_1$} ;
      \draw (-1/3,-1) coordinate (a) -- (1,3) coordinate (a_2)  node[anchor=west]{$\th^{-1}(0)$};
    \draw  (-2,-2/4) coordinate (b)-- (4,1)  node[anchor=north]{$\th^{-1}(0)$};
    
    \draw[fill=black!5] (0:0cm) -- (21:3.3cm) 
  arc (31: 64:3.9cm) -- (0:0cm) -- cycle ;

  \draw[fill=black!5] (0:0cm) -- (20:2cm) 
  arc (23: 66:2cm) -- (0:0cm) -- cycle ;

  \draw[fill=black!5] (0:0cm) -- (17:1cm) 
  arc (18: 67:1cm) -- (0:0cm) -- cycle ;
  
      \draw[fill=black!5] (0:0cm) -- (80:2cm) 
  arc (81:184:2cm) -- (0:0cm) -- cycle ;   
    
    \draw[fill=black!5] (0:0cm) -- (76:1cm) 
  arc (76:189:1cm) -- (0:0cm) -- cycle   ;   
  
        \draw[line width=1pt]   (0,0) .. controls (0.4,1.2) and (1.2,2.4) .. (2.3,3.78) ;

      \draw[line width=1pt]    (0,0) .. controls (1.2,0.27) and (2.32,0.81) .. (3.99,1.62) ;

         \draw[line width=1pt]    (0,0) .. controls (0.32,0.8)  and (0.52,2) .. (0.27, 3) ;

      \draw[line width=1pt]    (0,0) .. controls (-1.2,-0.27) and (-2.32,-0.20) .. (-3.09,0.1) ;

    \end{tikzpicture}
$$

\begin{lemma}\label{D2}
Let $a\in\N^n$ and $A\in\V^{\C\{\x\}}_{\a,\th}$. Then there exist constants   $a>0$ and $C>0$ such that $A$ is analytic on $\mathcal{D}_{\th,C,a,\eta}$ for every $\eta>0$.
\end{lemma}

\begin{proof}
We write $A=\sum_i\frac{a_i}{\th^{m(i)}}$ where $a_i$ is $(\a)$-homogeneous for every $i\in\N$. By multiplying $a_i$ by a convenient power of $\th$ we may even assume that there exist positive constants $a$ and $b$ such that $m(i)=ai+b$ for every $i$. \\
If $\nu_{\a}(a_i)=d_i$ there exist $C>0$ and $r>0$ such that
\begin{equation}|a_i(x)|\leq Cr^i\| x\|_{\a}^{d_i}\ \ \ \ \forall x\in\C^n\end{equation} by Theorem \ref{main2hens}, Example \ref{ex_conv}  and Inequality (\ref{bound1}) of Example \ref{ex_conv}.
On the other hand we claim that there exists a constant $C'>0$ such that 
\begin{equation}\label{7} |\th(x)|\geq C'd_{\a}(x,\th^{-1}(0))^{\nu_{\a}(\theta)}\ \ \ \ \forall x\in\C^n.
\end{equation}

Indeed if we embed $\C\{ \x\}  $ in $\C\{ \mathbf{y}\}  $ by sending $x_i$ onto $y_i^{\a_i}$, we have 
$$\theta(\x)=\theta(y_1^{\a_1}, \ldots ,y_n^{\a_n})=\t(y_1, \ldots ,y_n)$$
 and $\t$ is a homogeneous polynomial of degree $\nu_{\a}(\theta)$.  After a linear change of coordinates, we may assume that $\t$ is a monic polynomial in $y_n$ of degree $\nu_{\a}(\theta)$ multiplied by a constant.  Then, for all $y_1, \ldots ,y_n\in\C^n$, we have
$$|\t(y_1, \ldots ,y_n)|=C'\left|\prod_{i=1}^{\nu_{\a}(\theta)}\left(y_n-\phi_i(y_1, \ldots ,y_{n-1})\right)\right|$$
where $\phi_i$ is a homogeneous function which is locally analytic outside  the discriminant locus of $\t$, for some constant $C'>0$. Thus
$$|\t(y_1, \ldots ,y_n)|\geq C'\min_i\left|y_n-\phi_i(y_1, \ldots ,y_{n-1})\right|^{\nu_{\a}(\theta)}\geq$$
$$\geq  C'\inf_{y'\in\t^{-1}(0)}\max_k\left|y_k-y'_k\right|^{\nu_{\a}(\theta)}=C'd(y,\t^{-1}(0))^{\nu_{\a}(\theta)}$$
since $(y_1, \ldots ,y_{n-1},\phi_i(y_1, \ldots  ,y_{n-1}))\in\t^{-1}(0)$ for any $i$.
This proves (\ref{7}).\\
\\
Hence we have (for positive constants $\e$, $K$ and $x\in C_{K,\e}$):
$$\left|\frac{a_i(x)}{\th^{m(i)}(x)}\right|  \leq  \frac{C}{C'^{m(i)}}\frac{r^{i}\| x\|_{\a}^{d_i}}{d_{\a}(x,\th^{-1}(0))^{\nu_{\a}(\theta)m(i)}}  =
 \frac{C}{C'^{m(i)}}\frac{r^{i}\| x\|_{\a}^{i+\nu_{\a}(\theta)m(i)}}{d_{\a}(x,\th^{-1}(0))^{\nu_{\a}(\theta)m(i)}}\leq $$
 $$ \leq \frac{Cr^i\| x\|_{\a}^i}{C'^{m(i)}K^{\nu_{\a}(\theta)m(i)}}  \leq \frac{C(r\e)^i}{C'^{m(i)}K^{\nu_{\a}(\theta)m(i)}}=\frac{C}{C'^bK^{\nu_{\a}(\th)b}}\left(\frac{r\e}{C'^{a}K^{\nu_{\a}(\theta)a}}\right)^i.$$
Then if $\e<K^{a\nu_{\a}(\theta)}\left(\frac{C'^a}{r}\right)$, $A$ defines an analytic function on the domain $C_{K,\e}$. Thus $A$ defines an analytic function on the domain 
$\mathcal{D}_{\th,C'^a/r,a\nu_{\a}(\theta),\eta}$ for every $\eta>0$.
\end{proof}

This following proposition has been proven by Tougeron in the case $\a=(1, \ldots ,1)$ (see Proposition 2.8 \cite{To}):

\begin{prop}\label{analytic_disc}
Let $\a\in\N^n$ and let $P(Z)\in\C\{\x\}[Z]$ be a monic  polynomial whose discriminant is equal to $\d u$ where $\d\in\C[\x]$ is $(\a)$-homogeneous and $u\in\C\{\x\}$ is invertible. If $P(Z)$ factors as $P(Z)=P_1(Z)\cdots P_r(Z)$ where $P_i(Z)\in\C\{\x\}[Z]$ is an irreducible monic polynomial of $\C\{\x\}[Z]$ for all $i$, then $P_i(Z)$ is irreducible in $\V_{\a}^{\C\{\x\}}[Z]$. 
\end{prop}

\begin{proof}
Let $Q(Z)$  be an irreducible monic  factor of $P(Z)$ in $ \V_{\a}[Z]$. By Theorem \ref{main2} there exists a $(\a)$-homogeneous polynomial $\th\in\C[x]$ such that the coefficients of $Q(Z)$ are in $\V_{\a,\th}$. Let us denote by $A$ one of these coefficients.\\
Since $\V_{\a,\th}\subset\V_{\a,\th\d}$ we may assume that $\d$ divides $\th$, thus 
$$\d^{-1}(0)\cap B(0,\e)\subset\th^{-1}(0)\cap B(0,\e)$$
 for every $\e>0$. \\
Let $\eta>0$ small enough such that 
 the roots of $P(Z)$ are locally analytic on the domain 
$$D_{\th,\eta}:=B(0,\eta)\backslash \th^{-1}(0)\subset B(0,\eta)\backslash \d^{-1}(0).$$
Since $A$ is a polynomial depending on the roots of $P(Z)$ it is locally analytic on $D_{\th,\eta}$.\\
On the other hand by Lemma  \ref{D2} $A$ defines an analytic function on a domain 
$\mathcal{D}_{\th,C,a,\eta}$.\\
Thus by Lemma \ref{monodromy} given below $A$ is global analytic on $D_{\th,\eta}$. Since the roots of $P(Z)$ are bounded near the origin, $A$ is bounded near the origin, thus $A$ extends to an analytic function near the origin. This proves that $A$ is analytic on a neighborhood of the origin and $Q(Z)\in\C\{\x\}[Z]$.

\end{proof}

\begin{lemma}\label{monodromy}
Set $C>0$, $a>0$ and $\eta>0$ and let $\th\in\C[\x]$ be a $(\a)$-homogeneous polynomial. Let $A:D_{\th,\eta}\lgw \C$ be a multivalued function. Let us assume that $A$ is analytic on $\mathcal{D}_{\th,C,a,\eta}$ and locally analytic on $D_{\th,\eta}$. Then $A$ is analytic on $D_{\th,\eta}$.
\end{lemma}

\begin{proof}
Since $A$ is locally analytic on $D_{\th,\eta}$, then $A$ extends to an  analytic function on a small neighborhood of every path in $D_{\th,\eta}$. If $A$ is not analytic  on $D_{\th,\eta}$, then there exists a loop based at a point $p$ of $D_{\th,\eta}$, denoted by $\phi:[0,1]\lgw D_{\th,\eta}$ with $\phi(0)=\phi(1)=p$, such that $A$ extends to an analytic function on a neighborhood of $\phi$ but $A\circ\phi(0)\neq A\circ\phi(1)$.
 Let us write $\phi(t)=(\phi_1(t), \ldots ,\phi_n(t))$ and let us define $\Phi : [0,1]\times S\lgw \C^n$ by
 $$\Phi(t,s):=(s^{\a_1}\phi_1(t), \ldots ,s^{\a_n}\phi_n(t))$$
 where $S:=\{z\in \C\ / \ |z|\leq 1, \Re(z)>0\}$.\\
 Then we have that $\d(\Phi(t,s))=s^{\nu_{\a}(\d)}\d(\phi(t))\neq 0$ for any $(t,s)\in [0,1]\times S$ since $\Im(\phi)\subset D_{\th,\eta}$ and $s\neq 0$. Thus the image of $\Phi$ is included in $D_{\th,\eta}$. Moreover, for any $t\in [0,1]$, let $\Phi_t : S\lgw D_{\th,\eta}$ be the function defined by $\Phi_t(s):=\Phi(t,s)$. Its image is simply connected since $S$ is simply connected and $\Phi_t$ is analytic. Thus $A\circ\Phi_t$, which is locally analytic, extends to an analytic function on $S$ by the Monodromy Theorem.\\
  Let us denote by $h$ the holomorphic function on $S$ defined by 
  $$h(s):=A\circ\Phi(0,s)-A\circ\Phi(1,s)$$
   for any  $s\in S$.\\
  For any $s\in S$ and any $t\in[0,1]$ we have 
  $$\|\Phi(t,s)\|_{\a}=|s\||\phi(t)\|_{\a}$$
   and $$d_{\a}\left(\Phi(t,s),\theta^{-1}(0)\right)=|s|d_{\a}\left(\phi(t),\theta^{-1}(0)\right).$$
  Let us set $$K:=\frac{1}{2}\min_{t\in[0,1]}\frac{d_{\a}\left(\Phi(t,s),\theta^{-1}(0)\right)}{\|\Phi(t,s)\|_{\a}}=\frac{1}{2}\min_{t\in[0,1]}\frac{d_{\a}\left(\phi(t),\theta^{-1}(0)\right)}{\|\phi(t)\|_{\a}}>0.$$
  Thus for any $s$ belonging to the domain  $S\cap\{|s|<K^{a}C\}$, we have $\Phi(t,s)\in \mathcal{D}_{\th,C,a,\eta}$. Since $\Phi(t,s)\in \mathcal{D}_{\th,C,a,\eta}$ and $A$ is analytic on $\mathcal{D}_{\th,C,a,\eta}$, then $A\circ\Phi(0,s)=A\circ\Phi(1,s)$, thus $h(s)=0$ on $S\cap \{s<K^{a}C\}$. Since $h$ is holomorphic on the connected domain $S$, then $h\equiv 0$ on $S$. This contradicts the assumption. Hence $A$ is  analytic on $D_{\th,\eta}$. 
\end{proof}

Then we can extend Proposition \ref{analytic_disc} to the formal setting over any field of characteristic zero:

\begin{theorem}\label{AJ_gen}
Let $\k$ be a field of characteristic zero and $\a\in\R_{>0}^n$. Let $P(Z)\in\k\lb \x\rb [Z]$ be a monic polynomial whose discriminant is equal to $\d u$ where $\d\in\k[\x]$ is $(\a)$-homogeneous and $u\in\k\lb \x\rb $ is  a unit. If $P(Z)$ factors as $P(Z)=P_1(Z)\cdots P_s(Z)$ where the $P_i(Z)$ are irreducible monic polynomials of $\k\lb \x\rb [Z]$, then the $P_i(Z)$ remain irreducible in $\V_{\a}[Z]$. \end{theorem}

\begin{proof}
Let us prove this theorem when $P(Z)\in\C\{\x\}[Z]$. If $\a\in\N^n$, this is exactly Proposition \ref{analytic_disc}. If $\a\notin\N^n$, then by Lemma \ref{approx}, any decomposition $P(Z)=Q_1(Z)\cdots Q_r(Z)$ in $\V_{\a}[Z]$ is also  a decomposition in $\V_{\a'}[Z]$ for $\a'\in\Rel(\a,q,\e)$ where $\e$ is small enough. Then every irreducible monic factor of $Q_i(Z)$ in $\V_{\a'}[Z]$ is in $\C\{\x\}[Z]$ by Proposition \ref{analytic_disc}, thus $Q_i(Z)\in\C\{\x\}[Z]$ for every $i$. In particular since the $Q_i(Z)$ are irreducible in $\V_\a[Z]$ then they are irreducible polynomials of $\C\{\x\}[Z]$.\\
\\
Now let us consider the general case.
Let 
$$P(Z)=Z^d+a_{d-1}(\x)Z^{d-1}+\cdots+a_0(\x)$$
 be a polynomial satisfying the hypothesis of the theorem with $a_k(\x)\in\k\lb \x\rb $ for $0\leq k\leq d-1$. Since $P(Z)$ is defined over a field extension of $\Q$  generated by countably many elements and since such a field extension embeds in $\C$,  we may assume that $\C$ is a field extension of $\k$ and $P(Z)\in\C\lb \x\rb$.\\ 
The discriminant of $P(Z)$ is a polynomial depending on the coefficients $a_0(\x)$, \ldots , $a_{d-1}(\x)$ that we denote by $D(a_0(\x), \ldots ,a_{d-1}(\x))$. Let 
$$R(A_0, \ldots ,A_{d-1},U):=D(A_0, \ldots ,A_{d-1})-\d(\x)U\in\C[\x][A_0, \ldots ,A_{d-1},U].$$
Then $R(a_0(\x), \ldots ,a_{d-1}(\x),u(\x))=0$. \\
On the other hand, saying that $P(Z)$ factors as $P=P_1\cdots P_s$ is equivalent to
$$\exists b_1(\x), \ldots , b_r(\x) \text{ such that } a_i(\x)=R_i(b_1(\x), \ldots ,b_r(\x)) \ \ \ \forall i$$
for some polynomials $R_i(B_1, \ldots ,B_r)\in\Q[B_1, \ldots ,B_r]$, $0\leq i\leq d-1$ (these $R_i$ are the coefficients of $Z^i$ in the product $P_1(Z)\cdots P_s(Z)$ and the $b_j$  are the coefficients of the $P_k(Z)$).

By  Artin Approximation Theorem \cite{Ar}, for any integer $c>0$ there exist convergent power series 
$$\ovl{a}_{0,c}(\x), \ldots ,\ovl{a}_{d-1,c}(\x), \ovl{u}_c(\x),\ovl{b}_{1,c}(\x), \ldots ,\ovl{b}_{r,c}(\x)\in\C\{\x\}$$
 such that
\begin{equation}\label{eq_disc}R(\ovl{a}_{0,c}(\x), \ldots , \ovl{a}_{d-1,c}(\x), \ovl{u}_c(\x))=0,\end{equation}
\begin{equation}\label{eq_factors}\ovl{a}_{i,c}(\x)-R_i(\ovl{b}_{1,c}(\x), \ldots ,\ovl{b}_{r,c}(\x))=0 \text{ for } 0\leq i\leq d-1\end{equation}
and 
$$\ovl{a}_{k,c}(\x)-a_k(\x),\ \ovl{u}_c(\x)-u(\x),\ \ovl{b}_{l,c}(\x)-b_l(\x)\in (\x)^c$$
 for $0\leq k\leq d, 1\leq l\leq r.$
Set 
$$P_{(c)}(Z):=Z^d+\ovl{a}_{d-1,c}(\x)Z^{d-1}+\cdots+\ovl{a}_{0,c}(\x).$$
 Then $P_{(c)}(Z)$ factors as

$$P_{(c)}(Z)=P_{1,(c)}(Z)\cdots P_{s,(c)}(Z)$$ 
in $\C\{\x\}[Z]$ because of Equation (\ref{eq_factors}) (the coefficients of the $P_{i,c}(Z)$ are the $b_{k,c}$) , and $P_{i,(c)}(Z)-P_i(Z)\in(\x)^c\k\lb \x\rb [Z]$ for $1\leq i\leq s$. Moreover the discriminant of $P_{(c)}(Z)$ is of the form $\d(\x)u_{(c)}$ where $u_{(c)}$ is a unit in $\C\{\x\}$ if $c\geq 1$  by Equation (\ref{eq_disc}). Since $P_i(Z)$ is irreducible in $\k\lb \x\rb [Z]$, then $P_{i,(c)}(Z)$ is irreducible in $\k\lb \x\rb [Z]$ for all $i$ for $c$ large enough (let us say for $c\geq c_0$). Moreover we can remark that $\nu_{\a}(a)\geq\min_i\{\a_i\}\ord(a)$ for any $a\in\k\lb \x\rb $, thus $\nu_{\a}(\ovl{b}_{k,c}(\x)-b_k(\x))\geq \min_i\{\a_i\}c$ .\\

Let $c\geq c_0$ and let us assume that $P_{i,(c)}(Z)$ is not irreducible in $\V_{\a}[Z]$. Thus it is the product of two monic polynomials: let us say 
$$P_{i,(c)}(Z)=P_{i,(c),1}(Z)P_{i,(c),2}(Z)$$ 
with $P_{i,(c),1}(Z)$, $P_{i,(c),2}(Z)\in\V_{\a}[Z]$ and $\deg_Z(P_{i,(c),k}(Z))>0$ for $k=1,2$.  In fact by Theorem \ref{main2hens} we may assume that $P_{i,(c),1}(Z)$, $P_{i,(c),2}(Z)\in\V^{\C\{\x\}}_{\a}[Z]$. By Proposition \ref{analytic_disc} we see that $P_{i,(c),1}(Z)$, $P_{i,(c),2}(Z)\in\C\{x\}[Z]$, and by Proposition \ref{CK} $P_{i,(c)1}(Z)$, $P_{i,(c),2}(Z)\in\L\{x\}[Z]$ where $\L$ is a subfield of $\C$ which is finite over $\k$. Thus $\L=\k[\g]$ by the Primitive Element Theorem where $\g$ is a homogeneous element of degree 0 with respect to $\nu_{\a}$ by Example \ref{homo_0}. But we have $\V_{\a}\bigcap\k[\g]=\k$. Thus $P_{i,(c),1}(Z)$, $P_{i,(c),2}(Z)\in\k\{ \x\}[Z]\subset\k\lb\x\rb[Z]$ which contradicts the assumption that $P_{i,(c)}$ is irreducible in $\k\lb \x\rb [Z]$. Thus $P_{i,(c)}(Z)$ is irreducible in $\V_{\a}[Z]$.
Hence, by Corollary \ref{cor_factor_limit}, $P_i(Z)$ is irreducible in $\V_{\a}[Z]$ since $\nu_{\a}(\ovl{b}_{k,c}(\x)-b_k(\x))$ increases at least linearly with $c$.

\end{proof}
\noindent The next  proposition is a generalization of a result of S. Cutkosky and O. Kashcheyeva \cite{C-K} (see also Proposition 1 \cite{AM}) and we will use it to prove Theorem \ref{strong_analytic}. It is again an application of Theorem \ref{main2}.
\begin{prop}\label{CK}
Let $\k\lgw \k'$ be a characteristic zero field extension. Let $f\in\k'\lb\x\rb$ be algebraic over $\k\lb \x\rb$ and let $\L$ be the  field extension  of $\k$ generated by all the coefficients of $f$. Then $\k\lgw \L$ is a finite field extension. 
\end{prop}

\begin{proof}
Let $\a\in\R_{>0}^n$ such that $\dim_{\Q}(\Q\a_1+\cdots+\Q\a_n)=n$. By Theorem \ref{main2} the roots of the minimal polynomial of $f$ are in $\V_{\a}[\langle \g_1, \ldots , \g_n\rangle]$ for some homogeneous elements $\g_1$,  \ldots , $\g_n$ with respect to $\nu_\a$. Let us denote by $\V'_\a$ the ring defined in Definition \ref{growth'} and Lemma \ref{VV} where $\k$ is replaced by $\k'$. Then $\k'\lb \x\rb$ and $\V_{\a}[\langle \g_1, \ldots , \g_n\rangle]$ are subrings of $\V_{\a}'[\langle \g_1, \ldots , \g_n\rangle]$. Thus by unicity of the roots of the minimal polynomial of $f$ we have that $f\in \V_{\a}[\langle \g_1, \ldots , \g_n\rangle]$.
 By Example \ref{homo_mono} the homogeneous elements $\g_i$ may be written as  $\g_i=c_i\x^{\b_i}$ where $c_i$ is algebraic over $\k'$ (and so over $\k$) and $\b_i\in\Q^n$ for $1\leq i\leq n$.\\
By expanding $f$ either as a formal power series of $\k'\lb \x\rb$, $f=\sum_ib_i(\x)$ where $b_i(\x)\in\k'[\x]$ is a $(\a)$-homogeneous polynomial for any $i$, either as an element of $\V_{\a}[\langle \g_1, \ldots , \g_n\rangle]$, $f=\sum_i\frac{a_i(\x)}{\d(\x)^{m(i)}}\g_1^{k_1(i)}...\g_n^{k_n(i)}$, and by identifying the homogeneous terms of same valuation (which are monomials by Example \ref{homo_mono}), we obtain a countable number of relations of the following form:
\begin{equation}\label{cut}b(\x)\d^m(\x)=\sum_{} a_{n_1, \ldots ,n_s}(\x)\g_1^{n_1}...\g_n^{n_n}\end{equation}
 where $b(\x)$ (corresponding to the $b_i(\x)$), $a_{n_1, \ldots ,n_s}(\x)$ (corresponding to the $a_i(\x)$) and $\d$ are monomials, $b(\x)\in\k'[\x]$, $a_{n_1, \ldots ,n_s}(\x)\in\k[\x]$, $m\in\N$, and the sum is finite.  By dividing Equality \eqref{cut} by $\x^{\b}$ for $\b$ well chosen, we see that the coefficient of $b(\x)$ is in $\k[c_1, \ldots ,c_n]$ and $\L$ is a subfield of $\k[c_1, \ldots ,c_n]$.
\end{proof}

We can strengthen Theorem \ref{AJ_gen} as follows:
\begin{theorem}\label{strong_analytic}
Let $\a\in\R_{>0}^n$ and let $P(Z)\in\k\lb \x\rb[Z]$ be a monic polynomial such that its discriminant $\D=\d u$ where $\d\in\k[\x]$ is $(\a)$-homogeneous and $u\in\k\lb\x\rb$ is a unit. Let us set $N:=\dim_{\Q}(\Q\a_1+\cdots+\Q\a_n)$. Then there exist $\g_1$, \ldots , $\g_N$ integral homogeneous elements with respect to $\nu_{\a}$ and a $(\a)$-homogeneous polynomial $c(\x)\in\k[\x]$ such that the roots of $P(Z)$ are in $\frac{1}{c(\x)}\k'\lb\x\rb[\g_1, \ldots , \g_N]$ where $\k\lgw \k'$ is finite.
\end{theorem}

\begin{rmk}\label{Galois}
This result shows that for a given  root $z$ of the polynomial $P(Z)$ the other roots of $P(Z)$ are obtained from $z$ by the action of the elements of the Galois groups of the elements $\g_1$, \ldots , $\g_N$ on $z$. For instance if $\a\in\N^n$ (so $N=1$ -- we can always assume this by Lemma \ref{rel}), then the Galois group of $P(Z)$ is a quotient of the Galois group of the minimal polynomial of $\g_1$, i.e.  the Galois group of one weighted homogeneous polynomial.
\end{rmk}

\begin{proof}[Proof of Theorem \ref{strong_analytic}]
If $Q(Z)$ is a monic polynomial dividing $P(Z)$ in $\k\lb \x\rb[Z]$, then the discriminant of $Q(Z)$ divides the discriminant of $P(Z)$. Thus we may assume that $P(Z)$ is irreducible.\\
We will  consider three cases: first the case where the coefficients of $P(Z)$ are complex analytic with $\a\in\N^n$, then with $\a\in\R_{>0}^n$, and finally the general case.\\
\\
$\bullet$ Let us assume that $\a\in\N^n$ and that $P(Z)\in\C\{\x\}[Z]$. By Theorem \ref{main2} the roots of $P(Z)$ are of the form 
$$\displaystyle\sum_{i_1, \ldots ,i_s}A_{i_1, \ldots ,i_s}\g_1^{i_1}\cdots \g_s^{i_s}$$
 where $\g_1$,  \ldots , $\g_s$ are integral homogeneous elements with respect to $\nu_{\a}$ and $A_{i_1, \ldots ,i_s}\in\KK_{\a}^{\C\{\x\}}$ for any $i_1$,  \ldots , $i_s$. We may even choose $s=1$ by Proposition \ref{bound}, but we treat here the general case $s\geq 1$ that will be used in the sequel. \\
\\
We replace $\g_1$, \ldots , $\g_s$ by other integral  homogeneous elements with respect to $\nu_\a$ as follows: let us denote by $\g_{1,1}:=\g_1$,  \ldots , $\g_{1,q_1}$ the conjugates of $\g_1$ over $\K_n$. If $\g_2\notin \K_n[\g_{1,1}, \ldots , \g_{1,q_1}]$ we denote by $\g_{2,1}:=\g_2$,  \ldots , $\g_{2,q_2}$ its conjugates over $\K_n[\g_{1,1}, \ldots , \g_{1,q_1}]$ and so on. So for 
 $1\leq l\leq s$, $q_l$ denotes the degree of the minimal polynomial of $\g_l$ over $\K_n[\g_{i,j}]_{1\leq i<l, 1\leq j\leq q_i}$, and for $1\leq l\leq s$, $\g_{l,1}$,  \ldots , $\g_{l,q_l}$ denote the conjugates of $\g_l=\g_{l,1}$ over $\K_n[\g_{i,j}]_{1\leq i<l, 1\leq j\leq q_i}$. Then we may assume that the roots of $P(Z)$ are of the form
$$\sum_{{\begin{subarray}{c}0\leq i_1< q_1\\ \cdots\\ 0\leq i_s<q_s\end{subarray}}}A_{i_1, \ldots ,i_s}\g_{1,j_1}^{i_1}\cdots \g_{s,j_s}^{i_s}$$
where $A_{i_1, \ldots ,i_s}\in\KK_{\a}^{\C\{\x\}}$, $\nu_{\a}\left(A_{i_1, \ldots ,i_s}\g_{1,j_1}^{i_1}\cdots \g_{s,j_s}^{i_s}\right)\geq 0$ for any $i_1$,  \ldots , $i_s$, and $1\leq j_i\leq q_i$ for any $i$.\\
Let us assume that $P(Z)$ factors into a product of monic irreducible polynomials as $P(Z)=P_1(Z)\cdots P_r(Z)$ in $\KK_{\a}^{\C\{\x\}}[\g_{i,j}]_{1\leq i<s, 1\leq j\leq q_i}[Z]$.
We write the roots of $P_1(Z)$ as $z_j=\displaystyle\sum_{i=0}^{q_s-1}B_{i}\g_{s,j}^i$ where 
$$B_{i}\in \KK_{\a}^{\C\{\x\}}[\g_{i,j}]_{1\leq i<s, 1\leq j\leq q_i}$$
 for all $i$. Then the roots of the other $P_l(Z)$  are $\displaystyle\sum_{i=0}^{q_s-1}B_{i}'\g_{s,j}^i$ where $(B_0', \ldots ,B'_{q_s-1})$ is the image $(B_0, \ldots , B_{q_s-1})$ by a $\KK_{\a}^{\C\{x\}}$-automorphism of $ \KK_{\a}^{\C\{\x\}}[\g_{i,j}]_{1\leq i<s, 0\leq j<q_i}$. If the roots of $P_1(Z)$ satisfy the theorem, then we see that  the roots of the other $P_l(Z)$  will also satisfy the theorem since they are conjugates of the roots of $P_1(Z)$ by $\KK_{\a}^{\C\{x\}}$-automorphisms of $ \KK_{\a}^{\C\{\x\}}[\g_{i,j}]_{1\leq i<s, 0\leq j<q_i}$. Thus it is enough to prove the result for the roots of $P_1(Z)$. We have

$$\left(\begin{array}{c}z_1 \\ z_2\\ \vdots \\ z_{q_s}\end{array}\right)=\left(\begin{array}{cccc} 1 & \g_{s,1} & \cdots & \g_{s,1}^{q_s-1} \\
 1 & \g_{s,2} & \cdots & \g_{s,2}^{q_s-1}\\
 \vdots &\vdots & \vdots & \vdots\\
 1 & \g_{s,q_s} & \cdots & \g_{s,q_s}^{q_s-1}\end{array}\right)\left(\begin{array}{c}B_0 \\ B_1\\ \vdots\\ B_{q_s-1}\end{array}\right).$$
Let us set $\displaystyle M:=\left(\begin{array}{cccc} 1 & \g_{s,1} & \cdots & \g_{s,1}^{q_s-1} \\
 1 & \g_{s,2} & \cdots & \g_{s,2}^{q_s-1}\\
 \vdots &\vdots & \vdots & \vdots\\
 1 & \g_{s,q_s} & \cdots & \g_{s,q_s}^{q_s-1}\end{array}\right)$. The determinant of $M$ is a homogeneous element $c$ with respect to $\nu_{\a}$ where $\nu_{\a}(c)=\frac{1}{2}q_s(q_s-1)\nu_{\a}(\g_s)$. Thus we have 
 $$B_i=\frac{1}{c}\left(R_{i,1}(\g_{s,1}, \ldots ,\g_{s,q_s})z_1+\cdots+R_{i,s}(\g_{s,1}, \ldots ,\g_{s,q_s})z_{q_s}\right)$$
 where the $R_{i,j}$ are polynomials with coefficients in $\Q$ and the element $R_{i,j}(\g_{s,1}, \ldots ,\g_{s,q_s})$ is homogeneous with respect to $\nu_{\a}$. By multiplying $c$ and $R_{i,1}(\g_{s,1}, \ldots ,\g_{s,q_s})z_1+\cdots+R_{i,s}(\g_{s,1}, \ldots ,\g_{s,q_s})z_{q_s}$ by the conjugates of $c$ over $\k[\x]$ we may assume that $c=c(\x)\in\k[\x]$ is a $(\a)$-homogeneous polynomial. The $z_i$ and the $\g_{s,j}$ are locally analytic on $D_{\th,\eta}:=B(0,\eta)\backslash \th^{-1}(0)$ and bounded near the origin, where $\{\th=0\}$ contains the discriminant locus of $P(Z)$ and of the minimal polynomials of the $\g_i$  and $\eta$ is small enough. Thus  $c(\x)B_i$ is locally analytic on $D_{\th,\eta}$ for $1\leq i\leq q_s$ and is bounded near the origin. Moreover $c(\x)B_i$ is algebraic over $\k\lb \x\rb $ since the $g_{s,j}$ and the $z_k$ are algebraic over $\k\lb \x\rb $. By induction on $s$ (we replace $z_1$,  \ldots , $z_{q_s}$ by $c(\x)B_0$, \ldots , $c(\x)B_{q_s-1}$ - here we just used the fact that the roots of $P(Z)$ are algebraic over $\k\lb \x\rb $ and locally analytic over a domain of the form $D_{\th,\eta}$) we see that there exists a $(\a)$-homogeneous polynomial $c(\x)$ such that $c(\x)A_{\underline{i}}$ is locally analytic on $D_{\th,\eta}$ and bounded near the origin for any $\underline{i}:=(i_1, \ldots , i_s)$. Since $c(\x)A_{\underline{i}}\in\KK_{\a}^{\C\{\x\}}$ and it is bounded near the origin, we see that $c(\x)A_{\underline{i}}\in\V_{\a}^{\C\{\x\}}$. Thus it is analytic on $\mathcal{D}_{\th,C,a,\eta}$ for $C$, $a$ and $\eta$ well chosen (see Lemma \ref{D2}). Hence by Lemma \ref{monodromy} it is analytic on $D_{\th,\eta}$ and since it is bounded near the origin, $c(\x)A_{\underline{i}}\in\C\{\x\}$ for any $\underline{i}$. \\
 \\
 $\bullet$ Now let us consider any $\a\in\R_{>0}^n$ and $P(Z)\in\C\{\x\}[Z]$. Then the roots of $P(Z)$ are in $\V_{\a}[\langle \g_1, \ldots , \g_s\rangle]$ for some integral homogeneous elements with respect to $\nu_{\a}$ denoted by $\g_1$,  \ldots , $\g_s$. Let us denote these roots by $z_1$,  \ldots , $z_d$. For any $\a'\in\N^n$ such that $\Rel_{\a}\subset\Rel_{\a'}$, $\g_1$,  \ldots , $\g_s$ are integral homogeneous elements with respect to $\nu_{\a'}$. Thus, for any $\e>0$ small enough (say $\e<\e_0$), for any $q\in\N$ and any $\a'\in\Rel(\a,q,\e)$, $z_1$, \ldots , $z_d\in\V_{\a'}[\langle\g_1, \ldots , \g_s\rangle]$ by Proposition \ref{approx}. Moreover, by the previous case, we see that 
 $$\forall \e\in]0,\e_0[,\ \forall q\in\N, \ \forall \a'\in\Rel(\a,q,\e),$$
  $$\exists c_{\a'}(\x) \text{ an } (\a')\text{-homogeneous polynomial such that }$$
 $$c_{\a'}(\x)z_1, \ldots , \ c_{\a'}(\x)z_d\in\C\{\x\}[\g_1, \ldots , \g_s].$$
 Moreover we see that that $c_{\a}(\x)$ may be chosen as being the product of the determinants of Vandermonde matrices as $M$ depending only on $\g_1$,  \ldots , $\g_s$, thus $c_{\a'}(\x)$ does not depend on $\a'$. Let us denote $c(\x):=c_{\a'}(\x)$.  Since $c(\x)$ is a $(\a')$-homogeneous polynomial for all $\a'\in\Rel(\a,q,\e)$ then $c(\x)$ is a $(\a)$-homogeneous polynomial. This proves the result.\\
 \\
 $\bullet$ Now let us consider the general case, $\a\in\R_{>0}^n$ and $P(Z)\in\k\lb \x\rb[Z]$ where $\k$ is  a field of characteristic zero.\\
 Let  us write $P(Z)=Z^d+a_{d-1}(\x)Z^{d-1}+\cdots+a_0(\x)$. Exactly as in the proof of Theorem \ref{AJ_gen} we may assume that $\C$ is a field extension of $\k$. Let us use the notations of the proof of Theorem \ref{AJ_gen}. Let 
$$R(A_0, \ldots ,A_{d-1},U):=D(A_0, \ldots ,A_{d-1})-\d(\x)U\in\C[\x][A_0, \ldots ,A_{d-1},U]$$
where $D$ is the universal discriminant of a monic polynomial of degree $d$.
Then $$R(a_0(\x), \ldots ,a_{d-1}(\x),u(\x))=0.$$
By  Artin Approximation Theorem \cite{Ar}, for any integer $c>0$, there exist convergent power series $\ovl{a}_{0,c}(\x)$,  \ldots , $\ovl{a}_{d-1,c}(\x)$, $\ovl{u}_c(\x)\in\C\{\x\}$ such that
\begin{equation}\label{eq_disc'}R(\ovl{a}_{0,c}(\x), \ldots , \ovl{a}_{d-1,c}(\x), \ovl{u}_c(\x))=0,\end{equation}
and 
$$\ovl{a}_{k,c}(\x)-a_k(\x),\ \ovl{u}_c(\x)-u(\x)\in (\x)^c \ \ \text{ for } 0\leq k\leq d.$$
Let $P_{(c)}(Z):=Z^d+\ovl{a}_{d-1,c}(\x)Z^{d-1}+\cdots+\ovl{a}_{0,c}(\x)$. Then $P_{(c)}(Z)$  is irreducible for $c$ large enough (say $c\geq c_0$). Moreover the discriminant of $P_{(c)}(Z)$ is of the form $\d(\x)u_{(c)}$ where $u_{(c)}$ is a unit in $\C\{\x\}$ if $c\geq 1$ by Equation (\ref{eq_disc'}). By the previous case, the roots of $P_{(c)}(Z)$ are in $\frac{1}{c_c(\x)}\C\{\x\}[\g_{1,c}, \ldots ,\g_{N,c}]$ where $\g_{1,c}$,  \ldots , $\g_{N,c}$ are  integral homogeneous elements with respect to $\nu_{\a}$  and $c_c(\x)$ is a $(\a)$-homogeneous polynomial. By Proposition \ref{approx2} and the previous cases, we may assume that the $\g_{i,c}$ does not depend on $c$, thus let us denote $\g_{i,c}$ by $\g_i$. Moreover $c_{c}(\x)$ may be chosen as being the product of the determinants  of Vandermonde matrices as $M$ depending only on $\g_1$,  \ldots , $\g_s$, thus $c_c(\x)$ does not depend on $c$. Let us denote by $c(\x)$ this common $(\a)$-homogeneous polynomial.\\
Thus, when $c$ goes to infinity, we see that  the roots of $P(Z)$ are in $\frac{1}{c(\x)}\C\lb\x\rb[\g_1, \ldots , \g_N]$. Such a root has the form $\sum_{i_1, \ldots ,i_N}A_{i_1, \ldots ,i_N}\g_1^{i_1}\cdots \g_N^{i_N}$ where $i_k$ runs from 0 to $q_k-1$. In this case $c(\x)A_{i_1, \ldots ,i_N}\in\C\lb\x\rb$ is algebraic over $\k\lb\x\rb$, thus   $c(\x)A_{i_1, \ldots ,i_N}\in\k'\lb\x\rb$ where $\k\lgw \k'$ is finite by Proposition \ref{CK}. Thus the roots of $P(Z)$ are in $\frac{1}{c(\x)}\k'\lb\x\rb[\g_1, \ldots , \g_N]$.

\end{proof}

In the case where the $\a_i$  are linearly independent over $\Q$, we can choose $c(\x)=1$. This is exactly the Abhyankar-Jung Theorem:

\begin{corollary}[Abhyankar-Jung Theorem]\label{AJ}
Let $P(Z)\in\k\lb \x\rb [Z]$ be a monic polynomial whose discriminant  has the form $\x^{\b}u(\x)$ where $\b\in\N^n$ and $u(0)\neq 0$. Then there exist an integer $q\in\N$ and a finite field extension $\k\lgw \k'$  such that the roots of $P(Z)$ are in $\k'\lb x_1^{\frac{1}{q}}, \ldots ,x_n^{\frac{1}{q}}\rb$.

\end{corollary}

\begin{proof} 
 By the previous theorem applied to any $\a\in\R_{>0}^n$ satisfying $\dim_{\Q}(\Q\a_1+\cdots+\Q\a_n)=n$, the roots of $P(Z)$ are in $\frac{1}{\x^{\g}}\k'\lb x_1^{\frac{1}{q}}, \ldots ,x_n^{\frac{1}{q}}\rb$ for some $\b\in\N^n$, $q\in\N$ and $\k\lgw \k'$ a finite field extension. Since  the discriminant of any monic factor of $P(Z)$ in $\k'\lb x_1, \ldots ,x_n\rb[Z]$ divides the discriminant of $P(Z)$, we may assume that $P(Z)$ is irreducible in $\k'\lb x_1, \ldots ,x_n\rb[Z]$, thus we assume that $\k'=\k$.\\
 Let $z$ be a root of $P(Z)$ and let us denote by $\NP(z)$ its Newton polyhedron. Then $\NP(z)\subset -\g+\R_{\geq 0}^n$. Let us assume that $\NP(z)\not\subset\R_{\geq 0}^n$. This means that there exists $\g'\in\NP(z)$ such that one its coordinates, let us say $\g'_n$, is negative. But since $z$ is a root of $P(Z)$ that is a monic polynomial with coefficients in $\k\lb \x\rb $ then $\nu_{\a}(z)\geq 0$ for any $\a\in\R_{>0}^n$. But in this case there exists $\a\in\R^{n}_{>0}$ such that $\langle \a,\g'\rangle <0$ which is a contradiction. Thus $\NP(z)\subset \R_{\geq 0}^n$ which proves the corollary.
\end{proof}

Let us finish this part by giving a  few results which are analogous to  the fact that if $z\in\C\{t^{\frac{1}{k}}\}$ for some $k\in\N$, $t$ being a single variable, then its minimal polynomial over $\C\lb t\rb$ is a polynomial with convergent power series. The next result can also be seen as the converse of Theorem \ref{main2hens}:

\begin{corollary}
Let $P(Z)\in\k\lb \x\rb [Z]$ be an irreducible monic polynomial whose discriminant has the form $\d(\x)u(\x)$, where $\d(\x)$ is a $(\a)$-homogeneous polynomial, $\a\in\R_{>0}^n$, and $u(\x)\in\k\lb \x\rb $ is invertible. Let us assume that $P(Z)$ has a root in $\V_{\a}^R[\langle \g_1, \ldots , \g_s\rangle]$ where $R$ is an excellent Henselian local ring satisfying Properties (A), (B) and (C) and $\g_1$,  \ldots , $\g_s$ are homogeneous elements with respect to $\nu_{\a}$. Then the coefficients of $P(Z)$ are in $R$.
\end{corollary}

\begin{proof}
By Theorem \ref{AJ_gen}, $P(Z)$ is irreducible in $\V_{\a}[Z]$. Let 
$$z\in\V_{\a}^R[\langle \g_1, \ldots , \g_s\rangle]$$
 be a root of $P(Z)$ as given in the statement. We can write $z=\sum A_{i_1, \ldots ,i_s}\g_1^{i_1}\cdots \g_s^{i_s}$ where the sum is finite and $A_{i_1, \ldots ,i_s}\in\V_{\a}^R$. Then the others roots of $P(Z)$ are of the form $\sum A_{i_1, \ldots ,i_s}\s(\g_1)^{i_1}\cdots \s(\g_s)^{i_s}$ where $\s$ is a $\K_{\nu_{\a}}^{\alg}$-automorphism of $\ovl{\K}^{\alg}_{\nu}$. Thus all the roots of $P(Z)$ are in $\ovl{\V}_{\a}^R$. Hence the coefficients of $P(Z)$ are in $\ovl{\V}_{\a}^R\cap\k\lb \x\rb =R$.
\end{proof}

\begin{definition}
Let  $\k$ be a valued field and let $\s$ be a strongly  convex rational cone of $\R^n$ containing $\R_{\geq 0}^n$. There exists an invertible $n\times n$ matrix $M=(m_{i,j})_{1\leq i,j\leq n}$ such that $M\g\in\R_{\geq 0}^n$ for any $\g\in\s$.  We denote by $\k\{\x^{\b},\b\in\s\cap\Z^n\}$ the subring of $\k\lb \x^{\b},\b\in\s\cap\Z^n\rb$ of power series $f(\x)$ such that $f(\t(\x))\in\k\{\x\}$ where $\t$ is the map defined by 
$$(\t(x_1), \ldots ,\t(x_n\)=(x_1^{m_{1,1}}\cdots x_n^{m_{1,n}}, \ldots , x_1^{m_{n,1}}\cdots x_n^{m_{n,n}}).$$
By Example \ref{ex_conv} $\k\{\x^{\b},\b\in\s\cap\Z^n\}$ is a subring of $\V_{\a,\d}^{\k\{x\}}$ for any $\a$ such that $\langle\a,\g\rangle>0$ for all $\g\in\s\backslash\{0\}$.
\end{definition}

Let us mention the following theorem proven by A. Gabrielov and J.-Cl. Tougeron by using transcendental methods (they use in a crucial way the maximum principle for analytic functions):

\begin{theorem}\cite{Ga}\cite{To}
Let $P(Z)\in\C\lb\x\rb[Z]$ be an irreducible monic polynomial. If one root of $P(Z)$ is in $\C\{ \x^{\b},\b\in\s\cap\frac{1}{q}\Z^n\}$ where $\s$ is a strongly convex rational cone  and $q\in\N$, then $P(Z)\in\C\{\x\}[Z]$.
\end{theorem}

Using what we have done we extend this theorem to any algebraically closed valued field of characteristic zero under the assumption that the discriminant of $P(Z)$ is close to be weighted homogeneous. First we need the following lemma:

\begin{lemma}
Let $P(Z)\in\k\lb \x\rb[Z]$ be an irreducible monic polynomial  where $\k$ is a characteristic zero algebraically closed valued field. Let $\a\in\R_{>0}^n$ such that $\dim_{\Q}(\Q\a_1+\cdots+\Q\a_n)=n$ and $P(Z)$ is irreducible in $\V_{\a}[Z]$. By Theorem \ref{MD}, the roots of $P(Z)$ are in $\k\lb \x^{\b},\b\in\s\cap\frac{1}{q}\Z^n\rb$ where $\s$ is a strongly convex rational cone such that $\langle\a,\g\rangle>0$ for any $\g\in\s$, $\g\neq 0$, and $q\in\N$. If one root of $P(Z)$ is in  $\k\{ \x^{\b},\b\in\s\cap\frac{1}{q}\Z^n\}$, then the others roots of $P(Z)$ are in  $\k\{ \x^{\b},\b\in\s\cap\frac{1}{q}\Z^n\}$ and $P(Z)\in\k\{\x\}[Z]$.
\end{lemma}

\begin{proof}
Let $z\in \k\{ \x^{\b},\b\in\s\cap\frac{1}{q}\Z^n\}$ be a root of $P(Z)$. For any $\xi=(\xi_1, \ldots ,\xi_n)$  vector of $q$-th roots of unity  let us denote by $z_{\xi}$ the element of $\k\{ \x^{\b},\b\in\s\cap\frac{1}{q}\Z^n\}$ obtain from $z$ by replacing $(x_1^{\frac{1}{q}}, \ldots ,x_n^{\frac{1}{q}})$ by $(\xi_1x_1^{\frac{1}{q}}, \ldots ,\xi_nx_n^{\frac{1}{q}})$. In particular $z_{\xi}\in \k\{ \x^{\b},\b\in\s\cap\frac{1}{q}\Z^n\}$. Then for any $\xi$, $z_{\xi}$ is a root of $P(Z)$. Let $I$ be a subset of $\U_q^n$, where $\U_q$ is the group of $q$-th root of unity, such that 
$$z_{\xi}\neq z_{\xi'}\text{ for any }\xi, \,\xi'\in I, \,\xi\neq \xi',$$ 
$$\text{ and } \forall \xi\in \U_q^n,\, \exists \xi'\in I,\ z_{\xi'}=z_{\xi}.$$
Let us set $Q(Z)=\prod_{\xi\in I}(Z-z_{\xi})$. Then $Q(Z)$ is a monic polynomial of $\V_{\a}[Z]$ whose roots are roots of $P(Z)$. Thus it divides $P(Z)$ in $\V_{\a}[Z]$ hence, since $P(Z)$ is irreducible, $Q(Z)=P(Z)$. Thus the other roots of $P(Z)$ are in $\k\{ \x^{\b},\b\in\s\cap\frac{1}{q}\Z^n\}$ and $P(Z)\in\k\{\x\}[Z]$.
\end{proof}

\begin{corollary}
Let $P(Z)\in\k\lb \x\rb[Z]$ be an irreducible monic polynomial of degree $d$ where $\k$ is a characteristic zero algebraically closed valued field. Let $\a\in\R_{>0}^n$ such that $\dim_{\Q}(\Q\a_1+\cdots+\Q\a_n)=n$. Let us assume that there exists an irreducible monic polynomial $Q(Z)\in\k\lb \x\rb[Z]$ of degree $d$ whose  discriminant $\D_Q$ is a monomial times a unit and such that 
$$\nu_{\a}(P(Z)-Q(Z))\geq \frac{d}{2}\nu_{\a}(\D_Q).$$
 Let us assume moreover that one of the roots of $P(Z)$ is in $\k\{\x^{\b}, \b\in\s\cap\frac{1}{q}\Z^n\}$ for some strongly convex rational cone $\s$, where $\langle\a,\g\rangle>0$ for any $\g\in\s\backslash\{0\}$, and $q\in\N$.  Then the coefficients of $P(Z)$ are in $\k\{\x\}$.

\end{corollary}

\begin{proof}
By Remark \ref{maj} and Proposition \ref{cor_factor_limit}, the polynomial $P(Z)$ is irreducible in $\V_{\a}[Z]$. Thus we can apply the previous Lemma. \end{proof}


\section{Diophantine Approximation}\label{part_diop}

Here we give a necessary condition for an element of $\wdh{\K}_{\nu}$ to be algebraic over $\K_n$:

\begin{theorem}\label{diop}\cite{Ro}\cite{I-I}
Let $\nu$ be an Abhyankar valuation and let $z\in \K^{\alg}_{\nu}$. Then there exist two constants $C>0$ and $a\geq 1$ such that 
$$\left|z-\frac{f}{g}\right|_{\nu}\geq C|g|^a_{\nu}\ \ \forall f,g\in\k\lb \x\rb .$$
\end{theorem}
\begin{proof}
Let $P(Z):=a_0Z^d+a_1Z^{d-1}+\cdots+a_d\in\K_n[Z]$ be an irreducible polynomial such that $P(z)=0$. Let $h\in \k\lb \x\rb $ and set
$$P_h(Z):=h^da_0^{d-1}P\left(\frac{Z}{ha_0}\right).$$
Then $P_h(Z)=Z^d+a_1hZ^{d-1}+a_2a_0h^2Z^{d-2}+\cdots+a_da_0^{d-1}h^d$ and $zha_d$ is a root of $P_h(Z)$. It is straightforward to check that $z$ satisfies the theorem if and only if $zha_d$ does.  Thus we may assume that $P(Z)$ is a monic polynomial and $\nu(z)>0$ by choosing $h$ such that $\nu(h)$ is large enough. Let us set  $Q(Z_1,Z_2):=Z_1^dP(Z_2/Z_1)$. By Theorem 3.1 \cite{Ro} there exist two constants $a\geq d$ and $b\geq 0$ such that
$$\ord(Q(f,g))\leq a\min\{\ord(f),\ord(g)\}+b\ \ \ \forall f,g\in\k\lb \x\rb .$$
Moreover, by Izumi's Theorem (\cite{Iz}, \cite{Re}, \cite{ELS}), there exists a constant $c\geq 1$ such that for all $f\in\k\lb \x\rb $, $\ord(f)\leq \nu(f)\leq c\,\ord(f)$. Thus
$$\nu(Q(f,g))\leq ac\min\{\nu(f),\nu(g)\}+bc\ \ \ \forall f,g\in\k\lb \x\rb .$$
Since $P(Z)$ is irreducible in $\K_n[Z]$ and $\K_n$ is a characteristic zero field, $P(Z)$ has no multiple roots in $\wdh{V}_{\nu}$ and we may write 
$$P(Z)=R(Z)(Z-z)$$
where $R(Z)\in \wdh{V}_{\nu}[Z]$ and $R(z)\neq 0$.  Set $r:=\nu(z)$.
Let $f, g\in\k\lb \x\rb $ with $g\neq 0$. Two cases may occur: either 
\begin{equation}\label{diop1}\left|z-\frac{f}{g}\right|_{\nu}\geq e^{-r}\end{equation} either $\nu\left(z-\frac{f}{g}\right)> r$. In the last case we have $\nu\left(\frac{f}{g}\right)=\nu(z)>0$. In particular $\nu\left(R\left(\frac{f}{g}\right)\!\right)\geq 0$ and $\nu(f)>\nu(g)$. Thus
$$(ac-d)\nu(g)+bc\geq \nu\left(P\left(\frac{f}{g}\right)\!\right)\geq \nu\left(\frac{f}{g}-z\right).$$
Thus we have
\begin{equation}\label{diop2}A\nu(g)+B\geq \nu\left(\frac{f}{g}-z\right)\ \ \text{ or }\ \  \left|z-\frac{f}{g}\right|_{\nu}\geq e^{-B}|g|_{\nu}\end{equation}
with $A=ac-d$ and $B=bc$. Then (\ref{diop1}) and (\ref{diop2}) prove the theorem.

\end{proof}

\begin{ex}
Let $\s:=(-1,1)\R_{\geq 0}+(1,0)\R_{\geq 0}\subset \R^2$. This is a rational strongly convex cone of $\R^2$. Let $f(x_1,x_2)$ be a power series, $f(x_1,x_2)\in\k\lb x_1,x_2\rb$. Let us set
$$g(x_1,x_2):=\sum_{i=0}^{\infty}\left(\frac{x_2}{x_1}\right)^{i!}+f(x_1,x_2)\in\k\lb x^{\b}, \b\in \s\cap\Z\rb.$$
Then $g\in\V_{\a}$ for any $\a\in\R_{>0}^2$ such that $\a_2>\a_1$. Moreover
$$\nu_{\a}\left(g-f-\sum_{i=0}^n\left(\frac{x_2}{x_1}\right)^{i!}\right)=(n+1)!(\a_2-\a_1)=\frac{\a_2-\a_1}{\a_1}(n+1) \nu_{\a}(x_1^{n!}).$$ Thus there do not exist  constants $A$ and $B$ such that 
$$A\nu_{\a}(x_1^{n!})+B\geq \nu_{\a}\left(g-f-\sum_{i=0}^n\left(\frac{x_2}{x_1}\right)^{i!}\right)\ \ \forall n\in\N.$$
Hence $g(x_1,x_2)$ is not algebraic over $\F_2$ by Theorem \ref{diop}.
\end{ex}

\appendix

\section*{Notations}
\begin{enumerate}
\item[$\bullet$] $\nu_{\a}$ is the monomial valuation defined by $\nu_{\a}(x_i):=\a_i$ for any $i$ (cf. Example \ref{monomial}).
\item[$\bullet$] $V_{\nu}$ is the valuation ring associated to $\nu$.
\item[$\bullet$] $\wdh{V}_{\nu}$ is the completion of $V_{\nu}$.
\item[$\bullet$] $\K_n$ is the fraction field of $\k\lb \x\rb $ and $V_{\nu}$.
\item[$\bullet$] $\wdh{\K}_{\nu}$ is the fraction field of $\wdh{V}_{\nu}$.
\item[$\bullet$] $\text{Gr}_{\nu}V_{\nu}$ is the graded ring associated to $V_{\nu}$ (cf. Part \ref{part_homo}).
\item[$\bullet$] $V^{\alg}_{\nu}$ is the algebraic closure (or the Henselization) of $V_{\nu}$ in $\wdh{V}_{\nu}$ (see Lemma \ref{val}).
\item[$\bullet$] $\K^{\alg}_{\nu}$ is the fraction field of $V^{\alg}_{\nu}$.
\item[$\bullet$] $V^{\fg}_{\nu}$ is the subring of $\wdh{V}_{\nu}$ whose elements have $\nu$-support included in a finitely generated sub-semigroup of $\R_{\geq 0}$ (cf. Definition \ref{fg}).
\item[$\bullet$] $\K_{\nu}^{\fg}$ is the fraction field of $V_{\nu}^{\fg}$.
\item[$\bullet$] For any $\a\in\R_{>0}^n$, a $(\a)$-homogeneous polynomial is a weighted homogeneous polynomial for the weights $\a_1$, \ldots , $\a_n$ (see Definition \ref{alpha-hom}).
\item[$\bullet$] $A[\langle \g_1, \ldots , \g_s\rangle]$ is the valuation ring associated to $A[\g_1, \ldots , \g_s]$ when $A=\wdh{V}_{\nu}$, $V^{\fg}_{\nu}$ or $V^{\alg}_{\nu}$ (cf. Definition \ref{Vlr}).
\item[$\bullet$] $\ovl{V}_{\nu}$ is the direct limit of the rings $\wdh{V}_{\nu}[\langle\g_1, \ldots , \g_s\rangle]$ where the $\g_i$  are homogeneous elements with respect to $\nu$ (cf. Definition \ref{limit}).
\item[$\bullet$] $\ovl{\K}_{\nu}$ is the fraction field of $\ovl{V}_{\nu}$. 
\item[$\bullet$] $\ovl{V}_{\nu}^{\alg}$ is the direct limit of the rings $V_{\nu}^{\alg}[\langle\g_1, \ldots , \g_s\rangle]$ where the $\g_i$  are homogeneous elements with respect to $\nu$.
\item[$\bullet$] $\ovl{\K}_{\nu}^{\alg}$ is the fraction field of $\ovl{V}_{\nu}^{\alg}$. 
\item[$\bullet$] $\ovl{V}_{\nu}^{\fg}$ is the direct limit of the rings $V_{\nu}^{\fg}[\langle\g_1, \ldots , \g_s\rangle]$ where the $\g_i$  are homogeneous elements with respect to $\nu$.
\item[$\bullet$] $\ovl{\K}_{\nu}^{\fg}$ is the fraction field of $\ovl{V}_{\nu}^{\fg}$. 
\item[$\bullet$] $\V_{\a,\d}$ is the subring of $V^{\fg}_{\nu_{\a}}$ of elements  of the form $\displaystyle\sum_{i\in\La}\frac{a_i}{\d^{m(i)}}$ where $\La\subset\R $ is a finitely generated semigroup, $\nu_{\a}\left(\frac{a_i}{\d^{m(i)}}\right)=i$ and $i\lgm m(i)$ is bounded by an affine function (see Definition \ref{growth'}).
\item[$\bullet$] $\V_{\a}$ is the direct limit of the $\V_{\a,\d}$   over all the $(\a)$-homogeneous polynomials $\d$. It is a valuation ring (cf. Proposition \ref{VV}).
\item[$\bullet$] $\KK_{\a}$ is the fraction field of $\V_{\a}$ (cf. Definition \ref{KK}).
\item[$\bullet$] $\ovl{\KK}_{\a}$  is the direct limit of the fields $\KK[\langle\g_1, \ldots , \g_s\rangle]$ where the $\g_i$  are homogeneous elements with respect to $\nu$ (cf. Definition \ref{KK}).
\item[$\bullet$] $\V_{\a,\d}^R$ is the subring $\V_{\a,\d}$ whose elements are in the Henselian ring $R$ after a suitable transform (cf. Definition \ref{V^R}).
\item[$\bullet$] $\V_{\a}^R$ is the direct limit of the $\V_{\a,\d}^R$ over all the $(\a)$-homogeneous polynomials $\d$.
\end{enumerate}


\end{document}